\documentclass[5p,preprint]{elsarticle}

\usepackage{amsmath,amsfonts,amsthm,dsfont,amssymb}
\usepackage{mathdots}
\usepackage{mdwlist}
\usepackage{mathrsfs}
\usepackage{comment}
\usepackage{setspace}
\usepackage{verbatim}
\usepackage{cancel}

\usepackage{float}
\usepackage{longtable}
\usepackage{supertabular}
\usepackage{subfig}
\usepackage{bm}
\usepackage{flushend}
\usepackage{overpic}
\usepackage{enumerate}
\usepackage{siunitx}
\usepackage[europeanresistors]{circuitikz}

\usepackage{tikz,pgfplots}
\usetikzlibrary{fit}
\usetikzlibrary{calc}
\usetikzlibrary{arrows}
\pgfdeclarelayer{background}
\colorlet{inbox}{lightgray!20}
\colorlet{outbox}{lightgray!50}

\usetikzlibrary{patterns,math}
\usetikzlibrary{arrows}
\usetikzlibrary{shapes}
\usetikzlibrary{decorations.pathmorphing}
\usetikzlibrary{decorations.pathreplacing}
\usetikzlibrary{decorations.shapes}
\usetikzlibrary{decorations.text}
\usetikzlibrary{patterns}
\usetikzlibrary{backgrounds}
\usepackage{tkz-euclide}

  \pgfdeclarelayer{background1}
  \pgfdeclarelayer{background2}
  \pgfdeclarelayer{background3}
  \pgfdeclarelayer{background4}
  \pgfdeclarelayer{foreground}
  \pgfsetlayers{background4,background3,background2,background1,foreground,background,main}
  \tikzset{errorstyle/.style={thick,red,solid}}
  \tikzset{yrefstyle/.style={thick,black,dashed}}
  \tikzset{safetystyle/.style={thick,blue,dotted}}
  \tikzset{funnelstyle/.style={thick,blue,densely dotted}}
  \tikzset{funnelbackground/.style={black!20,opacity=0.5}}
  \tikzset{funneldstyle/.style={thin,blue,dashed}}
  \tikzset{qstyle/.style={green!50!black,ultra thick}}
  \tikzset{qhelpstyle/.style={green!50!black,very thin}}
  \tikzset{funnelfillstyle/.style={blue!20!white,opacity=0.8}}
  \tikzset{safetyfillstyle/.style={blue,opacity=0.1}}
  \tikzset{funnelthinfillstyle/.style={blue!5!white,opacity=0.8}}
  \tikzset{safetythinfillstyle/.style={blue!50!white,opacity=0.1}}

\usepackage[colorlinks]{hyperref}

\newtheorem{thm}{Theorem}[section]
 \newtheorem{cor}[thm]{Corollary}
 \newtheorem{lem}[thm]{Lemma}
 \newtheorem{prop}[thm]{Proposition}
 \theoremstyle{definition}
 \newtheorem{defn}[thm]{Definition}
 \newtheorem{rem}[thm]{Remark}
 
 \numberwithin{equation}{section}

\newcommand{\cA}{\mathcal{A}}
\newcommand{\cB}{\mathcal{B}}
\newcommand{\cC}{\mathcal{C}}
\newcommand{\cD}{\mathcal{D}}

\newcommand{\cG}{\mathcal{G}}

\newcommand{\cI}{\mathcal{I}}

\newcommand{\cL}{\mathcal{L}}

\newcommand{\cN}{\mathcal{N}}

\newcommand{\cP}{\mathcal{P}}
\newcommand{\cR}{\mathcal{R}}
\newcommand{\cS}{\mathcal{S}}

\newcommand{\cU}{\mathcal{U}}
\newcommand{\cV}{\mathcal{V}}
\newcommand{\cW}{\mathcal{W}}

\newcommand{\cF}{\mathcal{F}}
\newcommand{\fB}{\mathfrak{B}}
\newcommand{\eps}{\varepsilon}

\newcommand{\satu}{\text{sat}_{\widehat u}}

\newcommand{\fT}{\mathbf{T}}
\newcommand{\fN}{\mathbf{N}}
\newcommand{\esup}{{\rm ess\,}\sup}
\newcommand{\cTT}{{\mathbb T}^{n,q}_h}
\newcommand{\cCC}{\cC([-h,\infty),\R^n)}
\newcommand{\cLL}{\cL_{\rm loc}^\infty(\R_{\ge 0},\R^q)}

\newcommand{\cNN}{{\mathbf N}^{p,q,m}}

\renewcommand{\Re}{\operatorname{Re}}
\renewcommand{\Im}{\operatorname{Im}}

\newcommand{\Skdef}{\langle\raisebox{0.5 ex}{.},\raisebox{0.5 ex}{.}\rangle}

\newcommand{\ds}{\displaystyle}

\newcommand{\fL}{{\mathfrak L}}
\newcommand{\half}{\tfrac{1}{2}}
\newcommand{\third}{\tfrac{1}{3}}

\newcommand{\ddt}{\tfrac{\text{\normalfont d}}{\text{\normalfont d}t}}

\newcommand{\N}{\mathbb{N}}
\newcommand{\Z}{\mathbb{Z}}

\newcommand{\R}{\mathbb{R}}

\newcommand{\Gl}{\mathbf{Gl}}

\DeclareMathOperator{\dist}{dist}
\DeclareMathOperator{\rk}{rk}
\DeclareMathOperator{\im}{im}

\newcommand{\mA}{\mathrm{A}}
\newcommand{\dd}{\textrm d}
\newcommand{\bxi}{{\boldsymbol \xi}}

\renewcommand{\imath}{\mathrm{i}}

\numberwithin{equation}{section}

\DeclareMathOperator{\ran}{ran}

\newcommand{\setdef}[2]{\left\{\ #1\ \left|\ \vphantom{#1} #2\ \right.\right\}}
\newcommand{\setd}[2]{\left\{\, #1 \left|\, \vphantom{#1} #2\right.\right\}}

\DeclareMathOperator{\sgn}{sgn}
\DeclareMathOperator{\loc}{loc}

\def\tb#1{\textcolor[rgb]{0.00,0.80,0.20}{#1}}
\def\ai#1{\textcolor[rgb]{0.00,0.00,1.0}{#1}}

\newenvironment{smallpmatrix}
{\left(\begin{smallmatrix}}
{\end{smallmatrix}\right)}

\newenvironment{smallbmatrix}
{\left[\begin{smallmatrix}}
{\end{smallmatrix}\right]}

\newcommand{\epr}{\color{red}}

\begin{document}

\begin{frontmatter}

\title{Funnel control -- a survey\tnoteref{funding}}

\tnotetext[funding]{T. Berger acknowledges funding by the Deutsche Forschungsgemeinschaft (DFG, German Research Foundation) -- Project-IDs 362536361, 471539468 and 524064985.}

\author[1]{Thomas Berger}\ead{thomas.berger@math.upb.de}
\author[2]{Achim Ilchmann}\ead{achim.ilchmann@tu-ilmenau.de}
\author[3]{Eugene~P.~Ryan}\ead{masepr@bath.ac.uk}
\address[1]{Institut f\"ur Mathematik, Universit\"at Paderborn, Warburger Str.~100, 33098~Paderborn, Germany}
\address[2]{Institut f\"ur Mathematik, Technische Universit\"{a}t Ilmenau, Weimarer Stra{\ss}e 25, 98693~Ilmenau, Germany}
\address[3]{Department of Mathematical Sciences, University of Bath, Bath BA2 7AY, UK}


\begin{abstract}
The methodology of funnel control was introduced in the early~2000s, and it has developed since then in many respects
 achieving  a level of mathematical maturity balanced by practical applications.
 Its fundamental tenet is the attainment of prescribed transient and asymptotic behaviour for continuous-time controlled dynamical processes
 encompassing  linear and nonlinear systems described by functional differential equations,  differential-algebraic systems, and
 partial differential equations.
Considered are classes of systems specified only by structural properties~--~such as  relative degree and stable internal dynamics.
Prespecified are: a funnel shaped through the choice of a function (absolutely continuous),
freely selected by the designer, and a class of (sufficiently smooth) reference signals.
The aim is to design a single `simple' feedback strategy (using only input and output information) ~-- called the \textit{funnel controller}~-- which, applied to any system of the given class and for any reference signal of the given class, achieves the \textit{funnel control objective}:
that is, the closed-loop system is well-posed in the sense that all signals (both internal and external) are bounded and globally defined,
 and~-- most importantly~-- the error between the system's output
and the reference signal evolves within the prespecified funnel.

The survey is organized as follows.  In the Introduction, the genesis of funnel control is outlined via the most simple class of systems: the linear prototype of
scalar, single-input, single-output systems.   Generalizing the prototype, there follows an exposition of diverse system classes (described by linear, nonlinear, functional, partial
 differential equations, and differential-algebraic equations) for which funnel control is feasible.
The structure and properties of funnel control -- in its various guises attuned to available output information --  are described and analysed.
Up to this point, the treatment is predicated on an implicit assumption that system inputs are unconstrained.  Ramifications of input constraints and their incorporation in the
funnel methodology are then discussed.  Finally, practical applications and implementations of funnel control are highlighted.
\end{abstract}
\end{frontmatter}

\section{Introduction}\label{Sec:intro}\setcounter{equation}{0}
%
\noindent
A fundamental question in systems and control theory is: ``To what extent does one need to know a dynamical process in order to influence benignly
its behaviour through choice of input?"  Imprecision is inevitable in mathematically modelling any such process~-- be it biological, economic, electrical, mechanical, social, or
any other environment that evolves with time.  Given a process~-- known to belong to a specific class~-- can one control  its behaviour
knowing only the class but not which particular member of the class one happens to be dealing with?  In other words, is there a single control strategy that ``works" for every
member of the underlying class?   In essence, the broad field of adaptive control addresses  this question~-- the term ``adaptive" carrying the connotation of some
adjustment contrivance (explicit or implicit) to counter the lack of precise knowledge of the process to be controlled.

Roughly speaking, adaptive control can be compartmentalised into two categories: {\it identifier-based} strategies and
 its complement,  {\it non-identifier-based} strategies.
The former category has its origins in the early~1950s when the design of autopilots for high-performance aircraft triggered research in this area.
Development continued in the~1960s through the application of state space methods and Lyapunov's stability theory.   The underlying methodology
applies in the context of a parametrized class $\{P_\theta|~\theta \in\Theta\}$ to which the particular process $P_\theta$ to be controlled is known to belong
(but the associated parameter $\theta$ is not known).  An identifier-based strategy explicitly incorporates a mechanism which seeks to identify
the unknown parameter by generating, from input-output data, an estimate $\hat\theta\simeq\theta$ and applying control appropriate to the
estimated process $P_{\hat\theta}$.  However, according to \textit{{{\AA}}str{\"o}m}~(1983)~\cite{Astr83}, the early years showed a
``lot of enthusiasm, bad hardware and nonexisting theory''.

Identifier-based adaptive control is outside the scope of the present article.
Instead, the focus of attention is non-identifier-based adaptive control which emerged in the~1980s in response to two basic questions:~\\[-3ex]
\begin{itemize}
\item
What structural assumptions on the process to be controlled are sufficient (and/or necessary) to ensure the attainment of prescribed performance objectives in some appropriate sense?
\item
Assuming feasibility, is there a ``simple" controller that achieves the requisite performance without parameter identification or estimation?
\end{itemize}
The central concern of the present paper is an exposition of the theory of {\it funnel control} in the context of continuous-time nonlinear dynamical processes, with control input~$u$ and output~$y$,
encompassing  linear and nonlinear systems described by functional differential equations and differential-algebraic systems.

In its essence, the control problem to be addressed is the following: given a class $\Sigma$ of dynamical systems,  with $\R^m$-valued input~$u$ and~$\R^m$-valued output~$y$, and
a class of reference signals~${\mathcal Y}_{\text{\rm ref}}$, determine an output-feedback strategy which ensures that, for every system of class~$\Sigma$ and any reference signal~$y_{\text{\rm ref}}$ of
class~${\mathcal Y}_{\text{\rm ref}}$, the output $y$ approaches the reference~$y_{\text{\rm ref}}$ with prescribed transient behaviour and asymptotic accuracy.  The twin objectives of
``prescribed transient behaviour and asymptotic accuracy" are reflected in the adoption of a so-called ``performance funnel'' (see Fig.\,\ref{Fig:funnel}), defined by
\begin{equation} \label{eq97:funnel}
   \mathcal{F}_{\varphi}
    :=
 \setd{ (t,e) \in \R_{\ge 0}\times\R^m }{ \ \varphi(t) \, \|e\| < 1 }\,,
\end{equation}
in which the error function $t \mapsto e(t):=y(t)-y_{\text{\rm ref}}(t)$ is required to evolve.

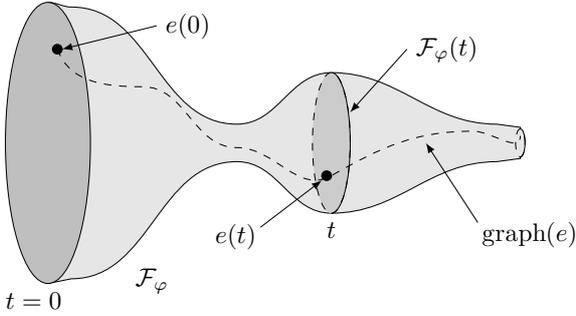
\begin{figure}[h!t]
\begin{center}
\resizebox{0.9\columnwidth}{!}{
        \begin{tikzpicture}[x=6.5mm, y=6.5mm]
            \tikzset{>=latex}

            \definecolor{lightlightgray}{gray}{0.9}
            \definecolor{lightergray}{gray}{0.8}

            \tikzstyle{dot}=[circle, fill, inner sep=1.5pt]

            \filldraw[fill=lightlightgray] (0,3) to[out=0, in=180] (0.5,2.9) to[out=0, in=180] (4,0.4) to[out=0, in=180] (6,1.5) to[out=0, in=180] (9.5,0.4) -- (10,0.333) -- (10,-0.333) -- (9.5,-0.4) to[out=180, in=0] (6,-1.5) to[out=180, in=0] (4,-0.4) to[out=180, in=0] (0.5,-2.9) -- (0,-3) -- cycle;

            \filldraw[fill=lightgray] (0,0) ellipse (0.9 and 3);

            \filldraw[fill=lightergray,dashed] (6,0) ellipse (0.4 and 1.5);

           \draw (6,-1.5) arc(270:360:0.4 and 1.5);
           \draw (6.4,0) arc(0:90:0.4 and 1.5);

            \filldraw[fill=lightlightgray,dashed] (10,0) ellipse (0.1 and 0.3);
            \draw (10,-0.3) arc(270:360:0.1 and 0.3);
           \draw (10.1,0) arc(0:90:0.1 and 0.3);

            \draw[dashed] (0.2,2) to[out=-70, in=180] (2,1.2) to[out=0, in=180] (4,-0.1) to[out=0, in=-150] (6,-0.7) to[out=30, in=180] (9,0.25) to[out=0, in=-90] (10,0.15);

            \draw (0.2,2) node[dot] {} -- (0.3,2.0);
            \draw[<-] (0.3,2)  -- (2.3,2.5) node[right] {$e(0)$};
            \draw[<-] (5.8,-0.8) -- (4.5,-1.8);
            \draw[<-] (6.395,0.65) -- (7.6,2) node[right] {$\cF_\varphi(t)$};
            \draw[->] (9.1,-1.7) -- (8,0);

            \node[below] at (-0.3,-3) {$t=0$};
            \node[below] at (2.2,-2.5) {$\mathcal{F}_\varphi$};
            \node[below] at (6,-1.5) {$t$};
            \node[left] at (4.6,-2) {$e(t)$};
            \node[dot] at (5.9,-0.7) {};
            \node[below right] at (9,-1.5) {$\text{graph}(e)$};

        \end{tikzpicture}
}
\end{center}
    \caption{Performance funnel $\mathcal{F}_\varphi$.} \label{Fig:funnel}
\end{figure}
\noindent
 The only {\it a priori} assumption on~$\varphi\colon\R_{\ge 0}\to\R$ is that it belongs to the following class of locally absolutely functions
\begin{equation}\label{Phi}
\Phi :=
\setdef{ \!\!  \varphi\in\mathcal{AC}_{\rm loc}(\R_{\ge 0},\R) \!\!}{\!\! \begin{array}{l}
\varphi (t) >0\ \forall\ t>0,\\ \liminf_{t\to\infty}\varphi(t) >0, \\
\exists\, c>0 \ \text{for a.a.}\ t\geq 0:\\
     |\dot\varphi (t)|\leq c\big(1+\varphi(t)\big)\  \end{array}\!\!\!\!}.
\end{equation}
The funnel is shaped -- through the choice of the function $\varphi\in\Phi$ -- in accordance with the specified transient behaviour and asymptotic accuracy. The funnel may be identified with the
graph of the map~$t\mapsto \cF_{\varphi}(t):=\big(-1/\varphi(t),1/\varphi(t)\big)$ from~$\R_{> 0}$ to the open intervals of~$\R$.
For~$t>0$, we refer to~$\cF_\varphi (t)$ as the {\it funnel~$t$-section}, see also Fig.\,\ref{Fig:funnel}.
We stress that, in \eqref{eq97:funnel}, $\varphi (0)=0$ is possible, in which case the funnel ${0}$-section is the whole space~$\R^m$
and so there is no restriction on the initial value~$e(0)$: with slight abuse of terminology, in this case we refer to~$\cF_\varphi$ as
an ``infinite funnel''.
As an example of an infinite funnel consider, for $\varepsilon >0$ and $T >0$,
the choice $\varphi_1(t)=\varepsilon^{-1}\min\{t/T,1\}$ for all $t\ge 0$, which  accords  with the aim of attaining a tracking accuracy quantified by $\varepsilon$ in prescribed time $T$ for all initial data. We also stress that, through the choice of unbounded $\varphi \in\Phi$, the radius of the funnel $t$-section shrinks to zero as $t\to\infty$, thus enabling asymptotic tracking; see Remark~\ref{Rem:IMP} for more details.
Whilst it is often convenient to choose a monotonically
decreasing funnel boundary, it might be advantageous to
widen the funnel over some later time intervals, for instance
in the presence of periodic disturbances or strongly-varying
reference signals. The formulation \eqref{eq97:funnel} encompasses a wide variety of funnel boundaries,
see also~\cite[Sec.~3.2]{Ilch13}.

We will  frequently use the phrase ``structural assumptions''~--  albeit without precise definition.
What we have in mind, roughly speaking, is that various components (functions, matrices, operators, etc.) of the differential equations governing the evolution of
the process to be controlled  do not need to be precisely known but are required only to exhibit certain properties (continuity, invertibility, causality, etc.). In particular, these properties should be preserved under state space transformation.\\[2ex]

\noindent
\textbf{Nomenclature} \\[2ex]
\begin{supertabular}[ht]{p{47pt}p{183pt}}
$\Re \lambda$, $\Im \lambda$
& the real, imaginary part of a complex number $\lambda\in{\mathbb{C}}$, respectively.\\
$\R_{\ge \alpha}$, $\R_{>\alpha}$, $\R_{\le\alpha}$, $\R_{<\alpha}$ & $[\alpha,\infty)$, $(\alpha,\infty)$, $(-\infty,\alpha]$, $(-\infty,\alpha)$, $\alpha\in\R$.\\
$\mathbb{C}_{\ge \alpha}$, ${\mathbb{C}}_{>\alpha}$, ${\mathbb C}_{\le\alpha}$, ${\mathbb C}_{< \alpha}$
   & complex numbers with real part in $\R_{\ge \alpha}$, $\R_{>\alpha}$, $\R_{\le\alpha}$, $\R_{<\alpha}$, respectively.\\
$\langle \cdot,\cdot\rangle$ & inner product  on a Hilbert space.\\
$\|\cdot \|$
   & norm on a normed space.\\
 $\Gl_n(\R)$
   & the general linear group of invertible real $n \times n$ matrices \\
$\R[s]$, $\R(s)$
   &  the ring of polynomials with coefficients in~$\R$ and  indeterminate~$s$,
    the quotient  field of $\R[s]$, respectively. \\
    ${\mathfrak L}(N_1,N_2)$  & the space of bounded linear operators $A:N_1\to N_2$, for normed spaces $N_1$ and $N_2$.\\
${\mathcal L}^\infty(I,  \R^\ell) $
  &  the Lebesgue space of measurable essentially bounded  functions $f:I\to\R^n$
with $\|f\|_\infty := {\text{ess\,sup}}_{t\in I}\|f(t)\|$, where $I\subseteq\R$ is some interval.\\
${\mathcal L}^\infty_{\rm loc} (I, \R^n)$ &
 the set  of  measurable locally  essentially bounded functions  $f:    I  \to \R^n$\, where $I\subseteq\R$ is some interval.\\
$\cL^p(I, \R^n)$ & the Lebesgue space of measurable and $p$th power integrable functions $f:I\to\R^n$, where $I\subseteq\R$ is some interval and $p\in[1,\infty)$.\\
$\mathcal{C}(I,  \R^n) $
  &  the set of continuous functions  $f:    I  \to \R^n$, where $I\subseteq\R$ is some interval.\\
 $\mathcal{C}^k(I,  \R^n) $
  &  the set of  $k$-times continuously differentiable functions  $f:    I  \to \R^n$, where $I\subseteq\R$ is some interval.\\
$\mathcal{AC}_{\rm loc}(I,  \R^n) $
  &  the set of locally absolutely continuous functions  $f: I  \to \R^n$, where $I\subseteq\R$ is some interval.\\
$\cW^{k,\infty}(I,  \R^n)$ &
the space of functions $f\in\cL^\infty (I,\R^n)$ with derivatives $f^{(i)}\in\cL^\infty (I,\R^n)$, $i=1,\ldots,k$,
where $I\subseteq\R$ is some interval and $k\in\N$.\\
$N\succ M $ &  $\langle x, (N-M) x\rangle >0$ for all $x\in\R^n\!\setminus\!\{0\}$, \ $N,M\in \R^{n\times n}$.\\
\end{supertabular}
%
\subsection{The genesis of funnel control: the scalar linear  prototype}
\noindent
By way of motivation, we seek to illustrate the salient characteristics of non-identifier-based adaptive control in the context of the simplest class
of continuous-time dynamical systems with control, namely, scalar linear systems of the form
\begin{equation}\label{abc}
\begin{aligned}
\dot x(t)&=ax(t)+bu(t), &&x(0)=x^0,\\
y(t)& =cx(t)  &&\text{output},
\end{aligned}
\end{equation}
or, equivalently,
$\dot y(t)=ay(t)+cb\,u(t)$, $y(0)=cx^0$,
where the parameters $a,b,c,x^0\in\R$ are arbitrary and unknown to the controller.  Only the output~$y$ is available for control purposes.
The quantity~$cb$ amplifies/attenuates and assigns a polarity to the input~$u(t)$.  In the spirit of the latter observation, we will refer to~$\sgn (cb)$ as the {\it control direction}.
(We disregard the trivial case of~$cb=0$ in which the control has no influence on the output  -- a circumstance that has
neither practical nor mathematical interest.)
\\

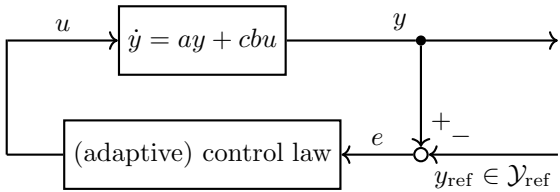
\begin{figure}[h!tb]
\begin{center}
  \begin{tikzpicture}[thick,node distance = 12ex, box/.style={fill=white,rectangle, draw=black}, blackdot/.style={inner sep = 0, minimum size=3pt,shape=circle,fill,draw=black},plus/.style={fill=white,circle,inner sep = 0, minimum size=5pt,thick,draw},metabox/.style={inner sep = 3ex,rectangle,draw,dotted,fill=gray!20!white}]

  \node (box1)[box, minimum size = 6ex]{$\begin{aligned}
\dot  {y}  =  a y+ cbu
\end{aligned}$};
  \node (yfork)[blackdot,right of = box1, xshift = 7ex] {};
  \node (end1)[right of = yfork, minimum size=0pt, inner sep = 0pt] {};
  \node (circ)[plus, below of = yfork, yshift = 2ex] {};
  \node (end2)[right of = circ, minimum size=0pt, inner sep = 0pt] {};
  \node (box2)[box, minimum size = 6ex, below of = box1, yshift = 2ex]{(adaptive) control law};
  \node (ucorner1)[left of = box2, minimum size=0pt, inner sep = 0pt, xshift = -5ex] {};
  \node (ucorner2)[above of = ucorner1, minimum size=0pt, inner sep = 0pt, yshift = -2ex] {};

  \draw(box1) -- (yfork) node[very near end,above]{$y$};
  \draw[->](yfork) -- (end1);
  \draw[->](yfork) -- (circ) node[near end,right]{$+$};
  \draw[->](end2) -- (circ) node[near end,above]{$-$}node[midway,below]{$y_{\rm \scriptsize ref}\in {\cal Y}_{\rm ref}$};
  \draw[->](circ) -- (box2) node[midway,above]{$e$};
  \draw(box2) -| ($(ucorner1)$);
  \draw(ucorner1) -| ($(ucorner2.north)$);
  \draw[->](ucorner2) -- (box1) node[midway,above]{$u$};

  \end{tikzpicture}
\end{center}
            \vspace{-0.5cm}
\caption{Closed-loop system}
\label{Fig:abc-cl-l}
\end{figure}

The
overall scenario is shown in Fig.\,\ref{Fig:abc-cl-l}, wherein $y_{\text{ref}}$ is some reference signal which the system output should emulate (in some appropriate sense).
In this simple setting, we trace the development of funnel control through two of its antecedents, namely, high-gain adaptive stabilization and $\lambda$-tracking.

\subsubsection{High-gain adaptive stabilization}\label{hgas}
\noindent
First, consider the problem of output feedback stabilization of~\eqref{abc}, that is, determine an output feedback strategy $u(t)=f(y(t))$ (if one exists)
 which ensures that, for
each $x^0\in\R$, every solution of the feedback-controlled initial-value problem
$\dot x(t)=ax(t)+bf(cx(t))$, $x(0)=x^0$, is global (i.e., exists on $\R_{\ge 0}$) and is such that
$y(t)\to 0$ as~$t\to\infty$  (in the context of Fig.\,\ref{Fig:abc-cl-l}, $y_{\text{ref}}=0$). If we assume that~\eqref{abc} satisfies the \textit{structural property}
\begin{equation}\label{cb}
cb >0,
\end{equation}
(that is, the control direction is positive)
then, given any~$\mu >0$ and setting $k^*:= (\mu+a)/(cb)$, we see that the linear output feedback $u(t)=-k^*y(t)$ gives the exponentially stable system $\dot x(t)=-\mu x(t)$.
Thus, arbitrarily fast exponential decay is achievable by output feedback $u(t)=-k^* y(t)$ with sufficiently large $k^*$ (the control {\it gain} in engineering parlance, whence
the terminology {\it high-gain} control).  This observation is referred to as the {\it high-gain property} of the system~\eqref{abc}.
In summary, the structural property~\eqref{cb} is sufficient for feasibility of stable behaviour by output feedback.    However, in the absence of any further knowledge of the parameters $a,b,c$, it is not possible
to compute a value~$k^*$ with the requisite property that $k^*$ should be larger than the threshold value $a/(cb)$.   Can this impasse be circumvented?
This question is answered in the affirmative if, instead of fixed-gain feedback, linear output feedback with \textit{variable} gain
\begin{equation}\label{vark}
u(t)=-k(t)y(t)
\end{equation}
is adopted and the monotone non-decreasing gain~$k(\cdot)$
is generated via the differential equation
\begin{equation}\label{kdot}
\dot k(t)=y(t)^2,\quad k(0)=k^0\in\R,
\end{equation}
where $k^0$ is arbitrary.  The combination of \eqref{abc}, \eqref{vark} and~\eqref{kdot} yields the initial-value problem
\begin{equation}\label{abck}
\begin{aligned}
\dot y(t)& =-(k(t)cb-a)y(t), & y(0)= y^0,\\
\dot k(t)& =y(t)^2, &    k(0)=k^0.
\end{aligned}
\end{equation}
Let $(y^0,k^0)\in\R^2$ be arbitrary.
The standard theory of ordinary differential equations applies
to conclude that~\eqref{abck} has a
unique maximal solution $(y,k)\colon [0,\omega)\to\R^2$, $0 < \omega\leq \infty$.  (Here, by ``standard theory", we mean basic results that can be found in elementary textbooks as, for example,
\cite{Walt98} or~\cite{LogeRyan14}.)
Differentiation of the  positive-definite form $z\mapsto z^2$ along the component $y(\cdot )$ of the
solution of~\eqref{abck} yields, for almost all $t\in[0,\omega)$,
\begin{multline*}
\ddt \big(y(t)^2\big)
 = 2 y(t) \dot y(t) = 2 y(t) \big(a-cb k(t)\big)\, y(t) \\
 =  -2cb \,k(t) \dot k(t) + 2a \,\dot k(t)
  =-cb\, \ddt \big(k(t)^2\big) + 2a \,\dot k(t)
\end{multline*}
which, on integration, gives
\begin{equation} \label{eq:y^2-int}
0 \le  y(t)^2  = (y^0)^2 - cb\,   \big(k(t)^2 -(k^0)^2\big)
                  +  2 a  \big(k(t) -(k^0)\big).
\end{equation}
In view of~\eqref{cb}, it immediately follows from~\eqref{eq:y^2-int} that $k\in{\mathcal L}^\infty([0,\omega),\R)$.
By boundedness of~$k$
we may infer from~\eqref{abck} that~$y$ is exponentially bounded.  Suppose~$\omega<\infty$, then the closure of the graph of
$(y,k)\colon [0,\omega)\to \R^2$ is a compact subset of $\R_{\ge 0}\times\R^2$ which contradicts maximality of the solution; hence $\omega =\infty$.
Boundedness of~$k$ is equivalent to
$y\in{\mathcal L}^2(\R_{\ge 0},\R)$ and, furthermore,
invoking~\eqref{abck}  we have
$\dot y\in{\mathcal L}^2(\R_{\ge 0},\R)$.
Therefore, we may conclude that $y(t)\to 0$ as $t\to\infty$.  Since the gain function~$k$ is bounded and monotone, it converges to a finite limit.
Thus, subject only to the structural assumption of positive control direction $cb >0$, every system~\eqref{abc} is stabilized by the adaptive strategy
\eqref{vark}--\eqref{kdot} and the controller gain function~$k$ is monotone and bounded.
However, boundedness of~$k$ may  fail to hold if the system~\eqref{abc} is subject to an extraneous disturbance.
This   failure can be
illustrated by means of a simple example.
Assume that the particular system~\eqref{abc}
is given by
$(a,b,c)=(0,1,1)$  and is subject to a spurious bounded additive signal~$d$, in which case the dynamics are
governed by
\[
\dot x(t)=u(t)+d(t).
\]
Application of control \eqref{vark}--\eqref{kdot} gives the {closed-loop initial-value problem
\begin{equation*}
\begin{aligned}
\dot x(t) & = -k(t )x(t)+d(t), &  x(0)=x^0,
\\
\dot k(t)  &=   x(t)^2, &   k(0)=k^0.
\end{aligned}
\end{equation*}
For purposes of illustration, assume that the disturbance is given by $d(t) := 3-\frac{10+9t}{3(1+t)^{4/3}}\to 3$ as $t\to\infty$.  Then,
 for initial data $(x^0,k^0)=(1,0)$, it is readily verified that there exists a unique global solution given by
\begin{multline*}
(x,k)\colon\R_{\ge 0}\to \R^2,\\ t\mapsto \big(x(t),k(t)\big) :=\left((1+t)^{-1/3}\,,\,3\big((1+t)^{1/3}-1\big)\right).
\end{multline*}
Thus, whilst the objective $x(t)\to 0$ as $t\to\infty$ is achieved, it is done at the expense of an unbounded gain function $k$ which, from the
viewpoint of implementation, renders the control strategy impracticable.

\subsubsection{Disturbances and high-gain $\lambda$-stabilization}\label{sect:lstab}
\noindent
The inability of the high-gain adaptive strategy \eqref{vark}-\eqref{kdot} to handle bounded disturbances  can be circumvented by weakening the control objective in the
following manner.  In the context of the scalar example~\eqref{abc}, in place of the objective $y(t)\to 0$ as $t\to\infty$ we substitute the weaker requirement that, as
$t\to\infty$, $y(t)$ should approach the interval $[-\lambda,\lambda]$ for some prescribed $\lambda >0$.  Introducing the distance function (parametrized by $\lambda >0$)
\[
\dist_\lambda\colon \R\to\R_{\ge 0},~ z\mapsto \max\{ |z|-\lambda\,,\,0\}
\]
we seek an output feedback of the form~\eqref{vark} which ensures the requisite performance:
$\dist_\lambda (y(t))\to 0$ as $t\to\infty$, and boundedness of the gain function~$k$.

Consider  system~\eqref{abc} but now with an additive disturbance
$d\in\cL^\infty(\R_{\ge 0},\R)$, with norm $\|d\|_\infty$:
\begin{equation}\label{abcd*}
\begin{aligned}
&\dot x(t)=ax(t)+bu(t)+d(t),\quad x(0)=x^0,\\
&\text{with output}~~y(t)=cx(t).
\end{aligned}
\end{equation}
Subject only to the structural assumption~\eqref{cb}, that is~$cb >0$,
we proceed to show that, for any prescribed~$\lambda >0$,
every system~\eqref{abcd*} with bounded disturbance~$d(\cdot)$
exhibits the requisite performance under
the output feedback~\eqref{vark} provided that the gain~$k$ is generated  via the differential equation
\begin{equation}\label{kdotd*}
\dot k(t) =  |y(t)|\, \dist_\lambda (y(t)),\quad k(0)=k^0\in\R.
\end{equation}
Note that the simplicity of the strategy \eqref{vark}-\eqref{kdot} is preserved~--
the novelty in~\eqref{kdotd*} resides in the suppression of the gain adaptation
whenever the output is \textsl{inside the $\lambda$-interval}, i.e., $|y(t)|\leq \lambda$.
The ``price" paid is that asymptotic  convergence of the output  to zero is lost: instead, only an asymptotic approach of the output to the interval~$[-\lambda,\lambda]$ is guaranteed.  However, since the latter
property holds for any prescribed accuracy parameter~$\lambda >0$, the price paid is small.

The combination of
the output feedback~\eqref{vark}  with the gain adaptation~\eqref{kdotd*}
applied to the disturbed scalar linear prototype~\eqref{abcd*}
yields the closed-loop initial-value problem
\begin{subequations}\label{abcdk*}
\begin{align}
\dot y(t)  & =-(k(t)cb-a)y(t)+cd(t),& y(0)=y^0,  \label{abcdk*-1}\\
\dot k(t) & =|y(t)|\, \dist_\lambda(y(t)),& k(0)=k^0. \label{abcdk*-2}
\end{align}
\end{subequations}
Let $(y^0,k^0)\in\R^2$ be arbitrary.  Again, the standard  theory of ordinary differential equations applies
to conclude that~\eqref{abcdk*} has a
unique  maximal solution
$(y,k)\colon [0,\omega)\to\R^2$, $0< \omega\leq\infty$.
Consider  the (Lyapunov-like) function
$z\mapsto \big(\dist_\lambda (z)\big)^2$ with derivative
\[
\delta_\lambda:\R\to\R,\ z\mapsto \left\{\begin{array}{ll}
2\,\dist_\lambda (z)\, \sgn(z),& z\neq 0
\\
0, & z = 0.\end{array}\right.
\]
Differentiation
along the component $y(\cdot)$ of the solution of~\eqref{abcdk*} yields,
for almost all $t\in[0,\omega)$,
\begin{align*}
&\ddt \big(\dist_\lambda (y(t))\big)^2
=     \delta_\lambda(y(t))\, \dot y(t) \notag
\\
& \leq   -2 \big(k(t) cb - a\big) \, |y(t)| \, \dist_\lambda (y(t))
       + 2  |c\,d(t)| \, \dist_\lambda (y(t))\notag \\
&=    - cb\,   \ddt \big( k(t)^2\big)
      + 2\left( a + \lambda^{-1} |c| \, \|d\|_\infty \right) \dot k(t), \notag
\end{align*}
which, on integration, gives
\begin{multline} \label{eq:d_lambda-int}
0\le  \big(\dist_\lambda (y(t)) \big)^2 \le \big(\dist_\lambda (y^0)\big)^2
- cb\,   \big(k(t)^2 -(k^0)^2\big)
\\
                  + 2\big( a + \lambda^{-1} |c| \, \|d\|_\infty \big)\, \big(k(t) -k^0\big).
\end{multline}
In view of~\eqref{cb}, it immediately follows from~\eqref{eq:d_lambda-int} that $k\in{\mathcal L}^\infty([0,\omega),\R)$.
By boundedness of~$k$ and essential boundedness of~$d$,
we may infer from~\eqref{abcdk*-1} that~$y$ is exponentially bounded. Suppose  $\omega<\infty$, then the closure of the graph of
$(y,k)\colon [0,\omega)\to \R^2$ is a compact subset of $\R_{\ge 0}\times\R^2$ which contradicts maximality of the solution;
hence $\omega =\infty$.
Boundedness of~$k$, together with~\eqref{eq:d_lambda-int}, implies
$\dist_\lambda (y(\cdot))\in\cL^\infty(\R_{\ge 0},\R)$.  Furthermore,
in view of~\eqref{abcdk*-1}, we have
$\dot y\in\cL^\infty(\R_{\ge 0},\R)$.
Therefore, the function
$t\mapsto \ddt \big(\dist_\lambda(y(t))\big)^2$
is bounded, and so
$\big(\dist_\lambda(y(\cdot))\big)^2$ is uniformly continuous.
Noting that,  for all $t\ge 0$,
\[
\int_0^t\!\big(\dist_\lambda (y(\tau))\big)^2 \textrm{d}\tau
\leq\!
\int_0^t |y(\tau)|  \dist_\lambda (y(\tau))\,\textrm{d} \tau = k(t)-k^0,
\]
we may infer (from boundedness of $k(\cdot)$) that the uniformly continuous function $\big(\dist_\lambda (y(\cdot))\big)^2$ is in ${\mathcal L}^1(\R_{\ge 0},\R)$.
By Barb{\v{a}}lat's Lemma we may now conclude that
$\big(\dist_\lambda(y(t))\big)^2\to 0$ as $t\to \infty$.
Therefore, subject only to the structural assumption $cb >0$, for every system~\eqref{abcd*} with bounded disturbance~$d$, the adaptive strategy~\eqref{vark},~\eqref{kdotd*}
achieves the  two performance objectives of
$\dist(y(t))\to 0$ as $t\to\infty$ and convergence of the gain~$k$ to a finite limit.

\subsubsection{High-gain $\lambda$-tracking}\label{sect:tracking}
\noindent
Consider again the class of systems with disturbance~$d(\cdot)$
 given by~\eqref{abcd*},  but now with the control objective of
 \textit{output $\lambda$-tracking}, that is,
for arbitrary prescribed $\lambda >0$ and a (suitably regular) reference signal $y_{\text{\rm ref}}$, we seek a control strategy which ensures that
$\dist_\lambda (y(t)-y_{\text{\rm ref}}(t)) \to 0$ as~$t\to \infty$.  For the class of admissible reference signals we choose
${\mathcal Y}_{\text{\rm ref}}=   {\mathcal W}^{1,\infty}(\R_{\ge 0},\R)$, that is,
$y_{\text{\rm ref}}\colon \R_{\ge 0}\to \R$
is admissible if it is bounded, absolutely continuous and has essentially bounded derivative.

Whilst the $\lambda$-tracking problem differs conceptually from the $\lambda$-stabilization problem of the previous subsection, there is no mathematical distinction
between these two problems.  Indeed, let $y_{\text{\rm ref}}\in {\mathcal Y}_{\text{\rm ref}}$ be arbitrary.  Writing $e(t)=y(t)-y_{\text{\rm ref}}(t)$, we see that
the differential equation in~\eqref{abcd*} may be expressed as
\[
\dot e(t)=ae(t)+cbu(t)+\hat d(t),
\]
with the function $\hat d\in\cL^\infty(\R_{\ge 0},\R)$ given almost everywhere by
\[
\hat d(t)=cd(t)+ay_{\text{\rm ref}}(t)-\dot y_{\text{\rm ref}}(t).
\]
Thus, we see that the $\lambda$-tracking problem for system~\eqref{abcd*}
 with reference signal $y_{\text{\rm ref}}\in {\mathcal Y}_{\text{\rm ref}}$ is equivalent
to the $\lambda$-stabilization problem for system~\eqref{abcd*}
with parameters $(a,cb,1,\hat d\,)$ and so the results of the previous subsection apply to conclude that, under the structural assumption~$cb >0$, the feedback strategy
\begin{equation}\label{tracking}
u(t)=-k(t)e(t),\quad\dot k(t)=   |e(t)|\,  \dist_\lambda(e(t)),\quad k(0)=k^0
\end{equation}
ensures attainment of the $\lambda$-tracking objectives:
$\dist_\lambda (e(t))\to 0$ as $t\to\infty$ and  convergence of the gain~$k$ to a finite limit.

\subsubsection{Unknown control direction}
\noindent
Throughout the above motivational discussion on adaptive stabilization and tracking in the restricted context of scalar systems, the structural assumption~\eqref{cb} remained in force.
Can this assumption be weakened?  As already noted, the case $cb=0$ is of neither practical nor mathematical interest.  The question then is: can
assumption~\eqref{cb} be replaced by
\begin{equation}\label{eq:cbnot0}
cb \not= 0.
\end{equation}
Clearly, the arguments adopted in Section \ref{hgas} apply {\it mutatis mutandis} to conclude that  the feedback
(a variant of \eqref{vark}, modified by the inclusion of the control direction term $\sgn (cb)$)
\begin{equation}\label{eq:u=sgn_ky}
u(t) \  = \ - \sgn(cb)\, k(t)\, y(t),\quad \dot k(t)=y(t)^2,~~k(0)=k^0
\end{equation}
ensures that $y(t)\to 0$ as $t\to\infty$ and the monotone gain function converges to a finite limit.  However, under the weakened assumption \eqref{eq:cbnot0},
this modified adaptive strategy cannot be realized as the control direction $\sgn (cb)$ is unknown to the controller.
Loosely speaking, what is required is a gain mechanism that can ``probe" in both the positive and negative control directions.  This idea points to a control strategy
of the form
\begin{equation}\label{eq:u=nussy}
u(t) \  = N(k(t))\, y(t),\quad \dot k(t)=y(t)^2,~~k(0)=k^0,
\end{equation}
where $N\colon \R\to\R$ is a continuous function with  the properties
\begin{equation}\label{Nprops}
\limsup_{k\to\infty}N(k)=+\infty\quad\text{and}\quad \liminf_{k\to\infty} N(k)=-\infty.
\end{equation}
One such function is $k\mapsto N(k)=k^2\cos k$.  This particular example exhibits the so-called ``Nussbaum properties":\\
for all $k^0\in\R$,
\begin{equation}\label{eq:Nussbaum}
\begin{aligned}
& \sup_{k>k^0} \frac{1}{k-k^0} \int_{k^0}^k N(\kappa)\, {\rm d}\kappa = \infty ,\\
 & \inf_{k>k^0} \frac{1}{k-k^0} \int_{k^0}^k N(\kappa)\, {\rm d}\kappa = -\infty .
\end{aligned}
\end{equation}
Let $N\colon\R\to\R$ be any locally Lipschitz function such that~\eqref{eq:Nussbaum} holds.
The combination of~\eqref{abc} and~\eqref{eq:u=nussy} yields the initial-value problem
\begin{equation}\label{abcnuss}
\begin{aligned}
\dot y(t)  &=(a+cb\,N(k(t)))y(t), & y(0)= y^0\\
\dot k(t)  &= y(t)^2, & k(0)= k^0.
\end{aligned}
\end{equation}
Let $(y^0,k^0)\in\R^2$ be arbitrary.
The standard theory of ordinary differential equations applies
to conclude that~\eqref{abcnuss} has
unique  maximal solution $(y,k)\colon [0,\omega)\to\R^2$, $0<\omega\leq \infty$.
Then, for almost all $t\in[0,\omega)$,
\[
\ddt \big(y(t)^2\big)
 = 2 y(t) \dot y(t) = 2 \big(a+cb N(k(t))\big) \dot k(t),
\]
which, on integration, gives
\begin{equation}\label{nussid}
0\leq y(t)^2 = (y^0)^2 + 2 cb \int_{k^0}^{k(t)} N(\kappa)\dd \kappa +2 a \big(k(t)-k^0\big).
\end{equation}
Consider the non-trivial scenario $y^0\neq 0$.} Seeking a contradiction, suppose that the monotonically non-decreasing function $k(\cdot)$ is unbounded.
Let $\tau\in (0,\omega)$ be such that $k(\tau) > k^0$
and set $\alpha := 2a+(y^0)^2/(k(\tau)-k^0)$.   Then it follows from~\eqref{nussid} that
\[
\forall\, t\in [\tau,\omega):\ 0\leq \alpha +\frac{2cb}{k(t)-k^0}\int_{k^0}^{k(t)} N(\kappa)\, \dd \kappa,
\]
which, depending on the system's control direction (unknown to the controller), runs counter to one or the other of properties~\eqref{eq:Nussbaum}: specifically, if $cb >0$,
then the second
of properties~\eqref{eq:Nussbaum} is contradicted or, if $cb <0$, then the first of these properties is contradicted.
 Thus, the supposition of unboundedness of~$k(\cdot )$ is false.
Having established boundedness of~$k(\cdot)$, an argument analogous to that used in Section~\ref{hgas} applies to conclude that $\omega =\infty$, $y(t)\to 0$ as $t\to\infty$ and~$k(\cdot)$
converges to a finite limit.  Thus, via the above gain mechanism, the efficacy of high-gain adaptive stabilization is preserved when the assumption $cb >0$ is weakened to $cb\neq 0$.  The same
modification preserves the efficacy of the adaptive $\lambda$-stabilizing and $\lambda$-tracking controllers in Sections \ref{sect:lstab} and \ref{sect:tracking} under the weakened assumption $cb\neq 0$.
\subsubsection{Funnel control}\label{Sssec:FunCon}
\noindent
Consider again a scalar system of the form~\eqref{abc}. As in Section \ref{sect:tracking}, let the class of reference signals
be ${\mathcal Y}_{\text{\rm ref}}= \cW^{1,\infty}(\R_{\ge 0},\R)$.
Prescribe a performance funnel~$\mathcal{F}_{\varphi}$ as in~\eqref{eq97:funnel}
with $m=1$ and~$\varphi\in \Phi$  as in~\eqref{Phi}, see Fig.\,\ref{Fig:funnel}.  
Let $x^0\in\R$ and  $y_{\text{\rm ref}}\in{\mathcal Y}_{\text{\rm ref}}$ be such that
$\varphi (0)|cx^0-y_{\text{\rm ref}}(0)| < 1$.
Note that the latter is automatically satisfied in the case of an ``infinite funnel'', i.e., $\varphi(0)=0$.
Under the structural assumption $cb >0$, we introduce the  \textit{funnel controller}, given by
\begin{multline}\label{eq:FC}
u(t)   =-k(t)e(t),~~k(t)= \varphi(t)\big(1-(\varphi(t)e(t))^2\big)^{-1},\\ e(t) = y(t) -y_{\rm \scriptsize ref}.
\end{multline}
The  idea underlying the gain adaptation~\eqref{eq:FC}
 is that  $k(t)$
is large if, and only if, $(t,e(t))$ is ``close'' to
the 
the boundary of the funnel $t$-section which, when coupled with the high-gain property of~\eqref{abc}, precludes boundary contact.

Under the weaker structural assumption $cb\neq 0$,  the funnel controller is modified in the following manner: the first of equations \eqref{eq:FC} is replaced
by
\[
u(t)=N(k(t))e(t),
\]
where   $N: \R\to\R$ is any continuous function  with the properties \eqref{Nprops}.
We stress that properties~\eqref{eq:Nussbaum} imply properties~\eqref{Nprops}, but the reverse implication is false: for example,
the function $s\mapsto N(s) = s\sin s$  \,
exhibits properties~\eqref{Nprops}, but fails to exhibit the Nussbaum properties~\eqref{eq:Nussbaum}.

Under either structural assumption $cb>0$ or $cb\not=0$,
the funnel controller is a proportional time-varying output error feedback.  However, in contrast with the $\lambda$-tracking control, the control gain in~\eqref{eq:FC} is not
monotone and not dynamically generated.   Instead, at generic time $t$, the gain $k(t)$  is statically generated via the nonlinear
function $\kappa\colon \mathcal{F}_{\varphi}\to \R,~(t,z)\mapsto \varphi(t)\big(1-(\varphi(t)z)^2\big)^{-1}$ evaluated at $(t,e(t))$.   In particular, $k(t)=\kappa(t,e(t))$ and, under the structural assumption $cb >0$,
the control is given by
\[
u(t)=-\kappa(t,e(t))\,e(t)
\]
or, under the weaker structural assumption $cb\neq 0$,
\[
u(t)=N\big(\kappa(t,e(t))\big)\,e(t).
\]
For purposes of exposition, we impose the weaker structural assumption $cb\neq 0$, and the combination of~\eqref{abc} and the funnel controller~\eqref{eq:FC} yields the closed-loop initial-value problem
\[
\dot e(t)=  \Big(a+cbN\big(\kappa(t,e(t))\big)\Big) \, e(t) +  ay_{\text{\rm ref}}(t)-\dot y_{\text{\rm ref}}(t)
\]
with $e(0) = cx^0 -y_{\text{\rm ref}}(0)$ and $(0,e(0))\in\cF_\varphi$, on the relatively open domain $\cF_\varphi\subseteq\R_{\ge 0}\times\R$.

By a \textit{solution} of this problem we mean  an absolutely continuous function
$e\colon [0,\omega)\to \R$ with  $\omega \in(0,\infty]$
and
$\text{graph}(e) \subseteq\cF_\varphi$.
A solution is \textit{maximal}, if it
has no proper right extension that is also a solution.
The theory of ordinary differential equations applies to conclude that the initial-value problem has a solution and every solution can be maximally extended. Let $e\colon [0,\omega)\to\R$ be a maximal solution. A maximal solution~$e$ is said to 
evolve strictly inside the performance funnel $\cF_\varphi$,
if there exists $\varepsilon >0$ such that
$\varphi(t)|e(t)|+\varepsilon \leq 1$ for all $t\in [0,\omega)$, in which case it immediately follows that
$\omega=\infty$ and the gain~$k$ and control~$u$ are bounded functions.
Therefore, in establishing the efficacy of funnel control, the crucial mathematical issue is to prove that every maximal solution 
evolves strictly inside $\cF_\varphi$.
 This can be shown via a
delicate contradiction argument which is not {elaborated here (but is subsumed by the proof of a significantly more general result in the main body of the manuscript)}.

We remark that, for prescribed $\lambda >0$, if $\varphi\in\Phi$ is chosen so that $\liminf_{t\to\infty}\varphi (t)\ge\lambda^{-1}$, then
attainment of strict evolution of~$e$ inside~$\cF_\varphi$ implies {\it a~fortiori} attainment of the $\lambda$-tracking objective $\dist_\lambda (e(t))\to 0$ as $t\to \infty$.

\subsubsection{A historical miscellany} \label{Sssec:history}
\noindent
The above considerations form an attempt to highlight fundamental characteristics of non-identifier-based adaptive control albeit in the simplified context of scalar linear systems.
The literature abounds with generalizations in various directions: for example, to higher-dimensional or infinite-dimensional systems and to encompass nonlinear systems.

The idea underpinning high-gain adaptive stabilization emerged in
the early~1980s
 in various investigations aimed at circumventing the need for cumbersome parameter estimators in order
to build adaptive controllers for certain high-gain stabilizable linear systems.
Seminal contributions towards this goal were made by
\textit{Morse} (1983)~\cite{Mors83}, \textit{Byrnes and Willems} (1984)~\cite{ByrnWill84},
and
\textit{Mareels} (1984)~\cite{Mare84}.  \textit{Morse} (1983)~\cite{Mors83} conjectured the non-existence of a smooth adaptive controller which stabilizes every system of the form~\eqref{abc} under assumption \eqref{eq:cbnot0}.
\textit{Nussbaum} (1983)~\cite{Nuss83}
showed that Morse's conjecture is false: this fact enabled the structural assumption~\eqref{cb} to be
weakened to the simple requirement~\eqref{eq:cbnot0}. As in the case of the scalar prototype outlined above (see also \textit{Willems and  Byrnes}
(1984)~\cite{WillByrn84}), multivariable systems with unknown control direction
are amenable to treatment using smooth functions with the ``Nussbaum properties"~\eqref{eq:Nussbaum} (see, for example, \cite{GeHong04,GeWang02,GeWang03,JianMare04,Ye01}).
These lines of investigation (see the survey~\cite{Ilch91}) culminated in \textit{M\aa rtensson's} (1985)~\cite{Mart85}
fundamental contribution  which,  in the context of multivariable linear systems,        established that ``the order of any
stabilising regulator is sufficient a priori information for
adaptive stabilisation".

Extension of the core idea in high-gain stabilization to the problem of tracking, by the system output, of a given reference signal were considered
by, {\it inter alia},
\textit{Mareels} (1984)~\cite{Mare84} and {\it Helmke, Pr\"atzel-Wolters \& Schmid}  (1990)~\cite{HelmPrat90a}.  These investigations invoke the {\it internal model principle}:
``a regulator is structurally stable only if the controller utilizes feedback of the regulated variable, and incorporates in the feedback loop a suitably reduplicated model of the dynamic
structure of the exogenous signals which the regulator is required to process'' (see {\it Wonham} (1979)~\cite{Wonh79}). In the context of high-gain asymptotic output tracking, this means
that a control strategy must incorporate
a dynamic component capable of replicating the reference signal that the output is required to track, which inevitably leads to complicated controller structures and places restrictions on the class~${\mathcal Y}_{\text{\rm ref}}$ of allowable reference signals.  By contrast, the high-gain~$\lambda$-tracking approach encompasses reference signals of a more general nature and is such that
the internal model principle is obviated, allowing control strategies of striking simplicity.   The concept of~$\lambda$-tracking was suggested in
\textit{Mareels} (1984)~\cite{Mare84}, is indirectly contained~--
albeit in a somewhat different context~--
in \textit{Miller and Davison} (1991)~\cite{MillDavi91},
and  was first studied for nonlinear systems in
\textit{Ilchmann and Ryan} (1994)~\cite{IlchRyan94}. For further contributions
including applications, see the survey by  \textit{Ilchmann and Ryan} (2008)~\cite{IlchRyan08}.

\noindent
The primary focus of the above historical contributions to both adaptive stabilization and $\lambda$-tracking was {\it asymptotic performance}: with the exception of
\textit{Miller and Davison} (1991)~\cite{MillDavi91}, {\it transient performance} was not considered.   Embracing transient performance was the final step in the genesis of
funnel control.  Whilst rudiments of the methodology can be found in \textit{Ilchmann} (1993) ~\cite[Thm.\,7.2.1]{Ilch93},
its full potential was not recognized until \textit{Ilchmann, Ryan, and Sangwin} (2002)~\cite{IlchRyan02b}
introduced the funnel controller.
A predecessor (which also takes transient behaviour into account)  is the above-mentioned work \cite{MillDavi91} by
\textit{Miller and Davison},  wherein an approach that differs intrinsically from the funnel control methodology  is adopted.

%
\section{Diverse system classes}\label{Sec:structural}\setcounter{equation}{0}
%
\newcommand{\Cm}{{\mathbb C}_{<0}}
\newcommand{\Cp}{{\mathbb C}_{>0}}
\noindent
Having presented the genesis of funnel control in the highly restrictive context of scalar linear systems, we proceed to describe and analyse funnel control (and variants thereof)
applied to considerably larger system classes encompassing linear and nonlinear multivariable systems, differential-algebraic systems, and infinite-dimensional systems.
The breadth of these classes attests to the mathematical maturity of the funnel control methodology.  Furthermore, the practical relevance of the approach is reflected in the 650-page monograph by \textit{Hackl} (2017)~\cite{Hack17} on applications of funnel control in mechatronics.
\subsection{The linear multivariable prototype} \label{Ssec:multivar}
\noindent
First, we focus on a class of square (that is, with equal number of inputs and outputs) linear, finite-dimensional systems of the form
\begin{equation} \label{eq:ABC}
 \left.
 \begin{array}{l}
  \dot{x}(t) = A\,  x(t) + B\, u(t), \quad x(0)=x^0\in\R^n ,
  \\[1mm]
    y(t) =C\, x(t)
 \end{array}
\right\}
\end{equation}
where~$(A,B,C)\in\R^{n\times n} \times \R^{n\times m} \times \R^{m\times n}$,  $m\leq n$,
and the space of inputs $u$ is $ \cU:=\mathcal{L}_{\rm loc}^\infty(\R_{\geq 0},\R^m)$.
For each $(x^0,u)\in\R^n\times\cU$, \eqref{eq:ABC} has a unique solution given by
\[
x:\R_{\geq 0}\to\R^n,\ t\mapsto {\rm e}^{At} x^0 + \int_0^t {\rm e}^{A(t-\tau)} B u(\tau)\, \dd \tau.
\]
We highlight some fundamental structural properties which are central to the funnel control methodology. For successful application of funnel control to~\eqref{eq:ABC}, the entries  of $(A,B,C)$, the initial value,
and even the state dimension need not be known.  What is required is output information and information
pertaining to the structural properties of relative degree, high-frequency gain, and zero dynamics.

\subsubsection{Relative degree}\label{Ssec:rel-degree}
%
\noindent
For a linear system~\eqref{eq:ABC} we define its {\it transfer function}~$G(s)$ (a rational-matrix-valued function) by
\[
G(s) := C(sI-A)^{-1}B\in\R(s)^{m\times m},
\]
which can be written as a formal power series
\[
    G(s) = \sum_{k=0}^\infty s^{-(k+1)} CA^kB
\]
with coefficients $CA^kB\in\R^{m\times m}$, $k \in\N_0$,  called {\it Markov parameters}. If the first non-zero Markov parameter in the power series for~$G(s)$ occurs at the power $s^{-r}$ and is invertible, then we say that system~\eqref{eq:ABC} has relative degree~$r$.
\begin{defn}
\label{Def:rel-deg}
The linear system~\eqref{eq:ABC}, equivalently the triple $(A,B,C)$,
is said to have {\it relative degree} $r  \in \N$,  if
$CA^kB=0$ for $k=0, \ldots ,r - 2$, and $\Gamma :=CA^{r -1}B$ is invertible.
\end{defn}
\noindent
Clearly, the first condition is vacuous in the case~$r=1$.   The Cayley-Hamilton theorem ensures that $r\leq n$.
The matrix~$\Gamma= C A^{r-1} B$ is referred to as
the {\it high-frequency gain matrix}.   It is a higher-dimensional analogue of the control direction $cb$ of Section \ref{hgas}.
We are now in a position to state our first structural assumption.
\\[1ex]
\textbf{(SA1)} \quad $(A,B,C)$ has relative degree $r\in\N$ and $r$ is known to the controller.
\begin{rem}
The terminologies ``relative degree" and ``high-frequency gain" have their origins in the early control engineering literature
in which a ``frequency domain" or ``transfer function" approach to linear systems was central:  the former terminology originates in the difference~$r$ between the degrees
of the denominator and numerator polynomials in transfer functions of single-input, single-output systems; the latter terminology reflects the fact that,
in ``steady state", the output from a stable single-input, single-output system driven by a sinusoidal input is sinusoidal of the same frequency $\omega$ with magnitude
amplified/attenuated by a ``gain" $|g(\omega)|$ with the property that $|g(\omega)|/\omega^r \to \Gamma$ as $\omega\to\infty$ and so, in this sense, $\Gamma$
characterizes ``high-frequency" behaviour.
\end{rem}

Assume that~\eqref{eq:ABC} has relative degree~$r\ge 2$.
Let~$x\in \mathcal{AC }_{\rm loc}(\R_{\geq 0},\R^n)$ be the solution corresponding to
$(x^0,u)\in \R^n\times\cU$, with associated output $y(\cdot)=Cx(\cdot)$.  Define functions $\xi_1,\ldots,\xi_r\in \mathcal{AC}_{\rm loc}(\R_{\geq 0},\R^m)$ by
\[
\xi_k(t):= CA^{k-1}x(t),~~ k=1,\ldots,r.
\]
Then,   for all $t\geq 0$,
\[
\dot\xi_k (t)= CA^{k-1}\dot x(t) =CA^k x(t)=\xi_{k+1}(t),~~k=1,\ldots, r-1,
\]
and so embedded in system~\eqref{eq:ABC} of relative degree $r$ is a chain (of length $r-1$) of $m$-dimensional integrators. In the following,
a coordinate transformation is described which makes this embedded chain explicit.

\subsubsection{Byrnes-Isidori form}\label{Ssec:BIF}
\noindent
Consider system~\eqref{eq:ABC} with relative degree $r\ge 2$. Introduce matrices
\begin{align*}
B_r &:= \begin{bmatrix} B, & AB, & \cdots \ , &  A^{r-1}B\end{bmatrix}\in \R^{n\times mr}
\\
\text{and} \quad
C_r&:=\begin{bmatrix} C\\CA\\ \vdots\\CA^{r-1}\end{bmatrix}\in\R^{mr\times n}
\end{align*}
and observe that $C_rB_r\in\R^{mr\times mr}$ is invertible.
Let $W\in\R^{n\times (n-mr)}$ be such that $\im W=\ker C_r$ and write
\[
V:= (W^\top W)^{-1}W^\top \big(I-B_r (C_rB_r)^{-1}C_r\big)\in\R^{(n-rm)\times n}.
\]
Then,
$U:= \begin{bmatrix} C_r \\ V\end{bmatrix}$ is invertible, with inverse  $U^{-1} = \begin{bmatrix} B_r(C_rB_r)^{-1}, ~W\end{bmatrix}$.
Define
\[
P:= VA^rB\Gamma^{-1},\quad Q:= VAW,\quad S:= CA^r W,
\]
and partition $CA^r B_r(C_rB_r)^{-1}$ into $r$ constituent blocks each of dimension $m\times m$:
\[
CA^r B_r(C_rB_r)^{-1} =\begin{bmatrix} R_1, & R_2, & \cdots\ , & R_{r-1} , & R_r\end{bmatrix}.
\]
The similarity transformation $(A,B,C) \to (UAU^{-1},UB,CU^{-1})$ yields the equivalent representation of \eqref{eq:ABC}
\begin{equation}\label{eq:ABC_BIform}
\begin{aligned}
\dot\xi_k (t)&=\xi_{k+1}(t),~~k=1,\ldots, r-1,\\
\dot\xi_r(t)&=\sum_{k=1}^r R_k\xi_k(t)+S\eta (t) +\Gamma u(t),\\
\dot\eta(t)&=P\xi_1(t)+Q\eta (t),\\
y(t) &= \xi_1(t),
\end{aligned}
\end{equation}
where $\xi_k(\cdot )=CA^{k-1}x(\cdot )$, $k=1,\ldots,r$ and $\eta(\cdot)=Vx(\cdot)$.
This special structure -- wherein the embedded chain of integrators constitutes the first $r-1$ of its dynamic equations -- is known as a {\it Byrnes-Isidori form}.
We remark in passing (and without proof) that, whilst not
a \textit{canonical form}, a Byrnes-Isidori form is close to being so in the sense that if two such forms differ, then they do so only through the triple $(Q,P,S)$.  However, any two such triples (regarded as linear input-output systems) must be
obtainable from each other by a state space transformation; this means that the difference in two Byrnes-Isidori forms is resolved through coordinate transformation of the~$\eta$ variable. More precisely,
if $(\tilde A,\tilde B,\tilde C)$ and $(\hat A,\hat B,\hat C)$ (with associated triples $(\tilde Q,\tilde P,\tilde S)$ and $(\hat Q,\hat P,\hat S)$) are two Byrnes-Isidori
forms in the similarity orbit
of $(A,B,C)$, then $(\tilde Q,\tilde P, \tilde S)$ and $(\hat Q, \hat P, \hat S)$ are similar.  In this sense, a Byrnes-Isidori form is {\it essentially unique}.
Because of this property, the form is often called {\it Byrnes-Isidori normal form} in the literature.
For future reference, we record that, in the
context of the Byrnes-Isidori normal form, the system transfer function is given by
\begin{equation}\label{BIformG}
\hspace*{-2ex}    G(s) = - \left( \sum_{i=1}^r R_i s^{i-1} - s^rI+ S(sI-Q)^{-1} P\right)^{-1} \hspace*{-1ex} \Gamma,
\end{equation}

The above discussion assumes that~$r \geq 2$.  In the relative degree one case~$r=1$ we have~$\Gamma =CB$ and  the Byrnes-Isidori form simplifies to
\begin{align*}
\dot\xi(t)&=R\xi (t)+S\eta (t) +\Gamma u(t),\\
\dot\eta(t)&=P\xi (t)+Q\eta (t),\\
y(t)&= \xi(t).
\end{align*}
In all cases, the triple $(Q,P,S)$ of {\em internal loop matrices} (unique up to a state space transformation) corresponds to a linear $(n-mr)$-dimensional system with input $y$ and output $z$, given by
\begin{equation}\label{eq:QPS}
\dot\eta (t)=Q\eta(t)+Py(t),\quad z(t)=S\eta (t),
\end{equation}
and referred to as the {\it internal dynamics}.

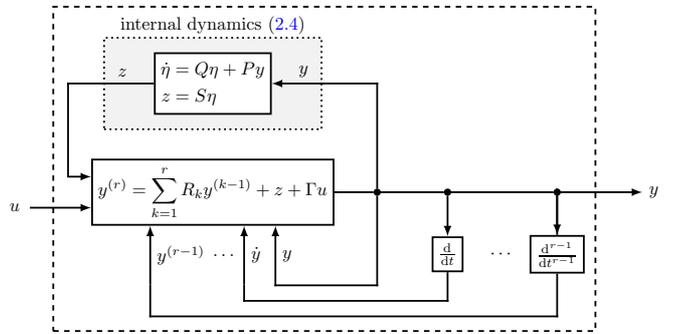
\begin{figure}[!t]
\begin{center}
\resizebox{\columnwidth}{!}{
  \begin{tikzpicture}[thick,node distance = 12ex, box/.style={fill=white,rectangle, draw=black}, blackdot/.style={inner sep = 0, minimum size=3pt,shape=circle,fill,draw=black},plus/.style={fill=white,circle,inner
sep = 0,thick,draw},metabox/.style={inner sep = 3ex,rectangle,draw,dashed,fill=gray!20!white}]
\tikzset{>=latex}
\tikzset{every path/.append style={line width=1pt}}

  \node (box1)    [box,minimum
size=8ex]{$\begin{aligned}y^{(r)}=\sum_{k=1}^{r }R_k y^{(k-1)}+z
+\Gamma u\end{aligned}$};
  \node (xrfork)[blackdot,right of = box1, xshift = 9ex] {};
  \node (x2fork)[blackdot,right of = xrfork, xshift = -3ex] {};
  \node (int1)    [box, below of = x2fork, yshift = 4ex]{$\frac{\rm
d}{\text{d} t}$};
  \node (dots1)    [right of = int1, xshift = -5ex, yshift = 0ex]{$\cdots$};
  \node (yfork)    [blackdot, right of = x2fork, xshift = 2ex] {};
  \node (int2)  [box,below of = yfork, yshift = 4ex]{$\frac{\text{d}^{r -1}}{\text{d} t^{r -1}}$};
  \node (box2)     [box, above of = box1, yshift =
2ex]{$\begin{aligned}\dot{\eta}&=Q\eta+Py\\ z&=S\eta\end{aligned}$};
  \node (end)    [right of = yfork, xshift = -1ex, minimum size=0pt,
inner sep = 0pt, xshift = 0ex] {};

  \node (ycorner0)    [above of = xrfork, minimum size=0pt, inner sep =
0pt, yshift = 2ex] {};
  \node (xrcorner1)    [below of = xrfork, minimum size=0pt, inner sep =
0pt, xshift = 0ex] {};
  \node (xrcorner2)    [below of = box1, minimum size=0pt, inner sep =
0pt, xshift = 8ex, yshift = 0ex] {};
  \node (x2corner1)    [below of = x2fork, minimum size=0pt, inner sep =
0pt, yshift = -2ex] {};
  \node (x2corner2)    [below of = box1, minimum size=0pt, inner sep =
0pt, xshift = 4ex, yshift = -2ex] {};
  \node (x2corner3)    [below of = box1, minimum size=0pt, inner sep =
0pt, xshift = 4ex] {};
  \node (ycorner1)    [below of = yfork, minimum size=0pt, inner sep =
0pt, yshift = -4ex] {};
  \node (ycorner2)    [below of = box1, minimum size=0pt, inner sep =
0pt, xshift = -8ex, yshift = -4ex] {};
  \node (ycorner3)    [below of = box1, minimum size=0pt, inner sep =
0pt, xshift = -8ex] {};
  \node (etacorner1)    [left of = box2, minimum size=0pt, inner sep =
0pt, xshift = -6.5ex] {};
  \node (etacorner2)    [left of = box1, minimum size=0pt, inner sep =
0pt, xshift = -6.5ex, yshift = 2ex] {};
  \node (begin)        [left of = box1, minimum size=0pt, inner sep =
0pt, xshift = -11.5ex, yshift = -2ex] {};

  \node(u_l)        [left of = box1, xshift = -4.5ex, yshift = 20ex]{};
  \node(l_r)        [right of = box1, xshift = 33ex, yshift = -14ex]{};
  \begin{pgfonlayer}{background}
       \node (system)  [metabox, fit = (u_l) (l_r), fill=lightgray!0] {};
  \end{pgfonlayer}

  \draw(box1) -- (xrfork)
node[pos=0.8,above]{$~$};
  \draw(xrfork) -- (x2fork)         node[midway,above]{};
  \draw[->](x2fork) -- (int1)
node[midway,above]{};
  \draw(x2fork) -- (yfork)        node[midway,above]{};
  \draw[->](yfork) -- (int2)
node[midway,above]{};
  \draw[->](yfork) -- (end)
node[pos=0.7,above]{} node[pos=1,right]{$y$};

  \draw(int1) -- (x2corner1.south)
node[midway,above]{};
  \draw[->](xrcorner2.south) -- (xrcorner2.south |- box1.south)
node[midway,right]{$y$};
node[midway,above]{};
  \draw(xrfork) -- (ycorner0.north);
  \draw[->](ycorner0.east) -- (box2)        node[midway,above]{$y\qquad ~$};
  \draw(box2) -- (etacorner1.west)node[midway,above]{$~\quad z$};
  \draw(etacorner1.north) -- (etacorner2.south);
  \draw[->](etacorner2.east) -- (etacorner2.east -| box1.west);
  \draw[->](begin) -- (begin -| box1.west)    node[midway,above]{} node[near start,left]{$u\quad$};
  \draw(xrfork) -- (xrcorner1.south);
  \draw(xrcorner1.east) -- (xrcorner2.west)    node[near start,above]{};
%

  \draw(x2corner1.east) -- (x2corner2.west);
  \draw(x2corner2.south) -- (x2corner3.north);
  \draw[->](x2corner3) -- (x2corner3 |- box1.south)
node(xi)[midway,right]{$\dot y$};
  \node (dots2)    [left of = xi, xshift = 5ex, yshift = 0ex ]{\qquad$~\ldots$};
  \draw(yfork.north) -- (int2);
  \draw(int2) -- (ycorner1.south);
  \draw(ycorner1.east) -- (ycorner2.west);
  \draw(ycorner2.south) -- (ycorner3.north);
  \draw[->](ycorner3) -- (ycorner3 |- box1.south)
node[midway,right]{$y^{(r-1)}$};

  \node(u_l_2)        [left of = box2, xshift = 2ex, yshift = 2ex]{};
  \node(l_r_2)        [right of = box2, xshift = -2ex, yshift = -2ex]{};

  \begin{pgfonlayer}{background}
       \node (system_2)  [metabox, dotted, fit = (u_l_2) (l_r_2), fill=lightgray!20] {};
  \end{pgfonlayer}
  \node (above_box_2) [above of = system_2, yshift = -4.5ex]{internal  dynamics \eqref{eq:QPS}};

  \end{tikzpicture}
}
\end{center}
            \vspace{-0.5cm}
  \caption{Byrnes-Isidori form}
\label{Fig:BIf}
\end{figure}
In summary, given a linear system $(A,B,C)$ of relative degree $r\geq 1$, we refer to its equivalent representation \eqref{eq:ABC_BIform} as its (essentially unique)
{\em Byrnes-Isidori form}.
The signal flow  for a system in Byrnes-Isidori form~\eqref{eq:ABC_BIform} is depicted in Fig.~\ref{Fig:BIf}.

\subsubsection{Zero dynamics}
\noindent
{Next, for the linear system~\eqref{eq:ABC},} we address the
following question: if the initial data and input are such that the output vanishes identically, what is the nature of the residual internal dynamic behaviour?
With this in mind, we proceed to define the {\it zero dynamics} $\mathcal{ZD}(A,B,C)$ of~\eqref{eq:ABC}.
Recall that
${\mathcal U}:=\mathcal L_{\text{loc}^\infty}(\R_{\ge 0},\R^m)$
 and,
for notational convenience, we write ${\mathcal X}:= {\mathcal{AC}}_{\rm loc}(\R_{\ge 0},\R^n)$.   Then
\[
\mathcal{ZD}(A,B,C) :=
 \setd{\! (x,u) \in  \mathcal{X}\times \mathcal{U} }{ \!\!\!\begin{array}{l}\dot x(t)=Ax(t)+Bu(t)\\ \text{a.e.},\ Cx(\cdot) = 0\end{array}\!\!\!\!}\!.
\]
Equivalently, the zero dynamics may be viewed as the solution space of the differential-algebraic system
\[
\frac{\dd}{\dd t}\begin{bmatrix}I & 0\\0&0\end{bmatrix}\begin{pmatrix}x(t)\\u(t)\end{pmatrix}= \begin{bmatrix} A & B\\C&0\end{bmatrix}\begin{pmatrix}x(t)\\u(t)\end{pmatrix}
\]
and so
\[
\mathcal{ZD}(A,B,C)= \text{ker}_{\mathcal{X} \times \mathcal{U}}\begin{bmatrix}A-\frac{\dd}{\dd t}I&B\\C&0\end{bmatrix}.
\]
The zero dynamics $\mathcal{ZD}(A,B,C)$ are said to be
\begin{itemize}
\item
{\it bounded}, if $(x,u)\in\cL^\infty(\R_{\ge 0},\R^n\times\R^m)$ for all $(x,u)\in {\mathcal {ZD}}(A,B,C)$;
\item
{\it asymptotically stable}, if for all $(x,u)\in {\mathcal {ZD}}(A,B,C)$ we have $x(t)\to 0$ as $t\to\infty$ and
$\text{ess\,sup}_{\tau\geq t}\|u(\tau)\| \to 0$ as $t\to\infty$.
\end{itemize}

Assume that system \eqref{eq:ABC} has relative degree~{$r\in\N$.}
Let $(Q,P,S)$ be the essentially unique representation of the internal dynamics. If $(x,u)\in{\mathcal ZD}(A,B,C)$, then, in view of the
Byrnes-Isidori form~\eqref{eq:ABC_BIform}, we may infer that $C_r x(\cdot ) = 0$, $\eta (\cdot ) = Vx(\cdot )$ satisfes $\dot \eta(\cdot ) =Q\eta(\cdot )$, $u(\cdot ) = -\Gamma^{-1}S\eta(\cdot )=-\Gamma^{-1}CA^rx(\cdot )$, and
\begin{align*}
&\mathcal{ZD}(A,B,C)\\
&=\setd{ \!\! (x,-\Gamma^{-1}CA^rx) }{\!\!\! \begin{array}{l} \dot x(t) =(I-B\Gamma^{-1}CA^{r-1})A x(t),\\
x(0)\in \bigcap\limits_{k=0}^{r  -1} \ker CA^k\end{array}\!\!\!\!}.
\end{align*}
From the Byrnes-Isidori form, we may also infer that
\[
\det\begin{bmatrix} A-sI & B\\C&0\end{bmatrix}=\det(\Gamma) \det(Q-sI) \in\R[s].
\]
Some immediate consequences of these inferences are recorded in the following.

\begin{prop}[\textbf{Relative degree and zero dynamics}]
\label{Prop:zero_dyn-srd}
Assume that system~\eqref{eq:ABC} has relative degree $r \in\N$.  Let $Q$ (unique up to similarity) be the internal loop matrix as in~\eqref{eq:QPS}.
Then  the zero dynamics $\mathcal{ZD}(A,B,C)$ are bounded if, and only if,
for all $\lambda\in\sigma(Q)$ we have $\Re \lambda \le 0$ and, if $\Re \lambda = 0$, then~$\lambda$ is semisimple. Moreover,  the following statements are equivalent:
\begin{enumerate}[$\bullet$]
\item the zero dynamics  $\mathcal{ZD}(A,B,C)$ are asymptotically stable;
\item $\sigma(Q){\subset} {\mathbb{C}}_{<0}$;
\item
$\forall\,\lambda\in{\mathbb{C}}_{\ge 0}:\ \det \begin{bmatrix}A-\lambda I & B \\ C & 0\end{bmatrix}\neq 0$.
\end{enumerate}
\end{prop}
\ \\
{We now introduce a second structural assumption.
\\[1ex]
\textbf{(SA2)} \quad The zero dynamics $\mathcal{ZD}(A,B,C)$ are asymptotically stable.
}

\subsubsection{High-gain stabilizability}
\noindent
A further  structural property exhibited by linear systems of the form \eqref{eq:ABC}~-- in the relative degree one case with asymptotically stable zero dynamics~--
is \emph{high-gain stabilizability} by output feedback.  In particular, if all eigenvalues of $CB$ have positive real part and the zero dynamics $\mathcal{ZD}(A,B,C)$ are asymptotically stable,
then there exists $k^*>0$ such that, for each fixed $k\ge k^*$, the output feedback $u(t)=-ky(t)$, renders the closed-loop system
$\dot x(t) =(A-kBC)x(t)$ asymptotically stable, i.e., $\sigma(A-kBC)\subseteq {\mathbb{C}}_{<0}$.
This is
the multivariable counterpart of the high-gain property for the scalar prototype of Section~\ref{hgas} and is~-- in different words~-- the content of the following lemma (see,~\cite[Lem.~2.2.7]{Ilch93}).
\begin{lem}[\textbf{High-Gain Lemma}]
\label{Lem:high-gain}
Consider a system~\eqref{eq:ABC} which satisfies~{(SA2)} and assume that $\sigma (CB)\subset{\mathbb{C}}_{>0}$. Then there exists $k^*> 0$ such that, for each fixed $k\ge k^*$, we have
\[
\sigma (A-kBC)  \subset {{\mathbb{C}}_{< 0}}\,.
\]
\end{lem}
Whilst Lemma~\ref{Lem:high-gain} does not play an explicit role in the ensuing exposition of funnel control, it implicitly
underpins much of the underlying intuition and early development of the funnel methodology.

\newcommand{\lincl}{\cL^{m,r} }
\subsubsection{Class $\mathbf{\cL}^{m,r}$ of linear systems amenable to funnel control}\label{Sec:Lmr}
\noindent
We summarize and close the above discussion with the following
description of a class of linear systems of form~\eqref{eq:ABC} which are amenable to control by funnel techniques in the sense that
he controllers developed in later sections are applicable.  This class comprises systems $(A,B,C)$ of form~\eqref{eq:ABC}
with known relative degree~$r$ (assumption~(SA1)), {with asymptotically stable zero dynamics (assumption~(SA2)), and which
satisfies our third structural assumption
(a higher-dimensional  counterpart of assumption~\eqref{eq:cbnot0}):
\\[1ex]
\textbf{(SA3)}  \quad
$ \forall\, v \in \R^m: \ \ v^\top \Gamma v =0~~\iff~~v=0$.
\\[1ex]
Assumption~(SA3) means that~$\Gamma$ is {\it sign-definite} and, stated otherwise, it is equivalent to the requirement that either $\Gamma + \Gamma^\top \succ 0$ or $-(\Gamma + \Gamma^\top) \succ 0$ (but which of these two possible polarities holds is not known to the controller).}
In particular, we {define the system class}
\begin{equation}\label{eq:Lmr}
\lincl \!:=\!
\setd{\!\!\!\begin{array}{l} (A,B,C)\\ \in\!\R^{n\times n}\!\times\!\R^{n\times m}\!\times\!\R^{m\times n}\end{array}\!\!\!}{\!\!\!
\begin{array}{l}
n\!\in\!\N, \text{(SA1),~(SA2)},\\
\text{and (SA3) hold}
\end{array}\!\!\!}.
\end{equation}

\subsection{Nonlinear functional differential systems}\label{Ssec:NonlFuncDiffSys}
\noindent
The notions of relative degree, control direction, and zero dynamics -- introduced in the context of finite-dimensional linear ODE systems -- when suitably
generalized underpin requisite structural assumptionsfor successful application of funnel control to more diverse classes of systems.

For the sake of motivation, consider again a linear system~\eqref{eq:ABC} with relative degree~$r\in\N$ in Byrnes-Isidori form~\eqref{eq:ABC_BIform}.
With  its internal  dynamics~\eqref{eq:QPS} we may associate a linear operator
\begin{equation}\label{eq:L}
L\colon  y(\cdot) \mapsto \left( t \mapsto   \int_0^t S {\rm e}^{Q(t-\tau)} Py(\tau)\, \dd \tau \right).
\end{equation}
With initial data $\eta(0)=\eta^0 := Vx^0$ and $d(\cdot) := S{\rm e}^{Q \cdot} \eta^0$, the output $z(\cdot )$ of~\eqref{eq:QPS} is
given by
\[
z(t) = d(t) +L(y)(t).
\]
Introducing the (linear) operator
\begin{equation}\label{eq:opT-lin}
\begin{aligned}
&\fT\colon \cC(\R_{\ge 0},\R^{rm})\to \cL^\infty_{\text{loc}}(\R_{\ge 0},\R^m),\\
&\zeta =(\zeta_1,\ldots,\zeta_r)
\mapsto
\left(  t\mapsto \sum_{k=1}^{r}  R_k  \zeta_k(t) + L(\zeta_1)(t) \right),
\end{aligned}
\end{equation}
it follows from~\eqref{eq:ABC_BIform}  that~\eqref{eq:ABC} is equivalent to the functional differential system
\begin{equation}\label{eq83:T}
\left.\begin{array}{l}y^{(r )}(t)  =  d(t)+\fT(y,\ldots,y^{(r-1)})(t)+\Gamma u(t)
\\[1ex]
y(0)=Cx^0, \ldots ,y^{(r-1)}(0)=CA^{r-1}x^0.
\end{array}\right\}
\end{equation}

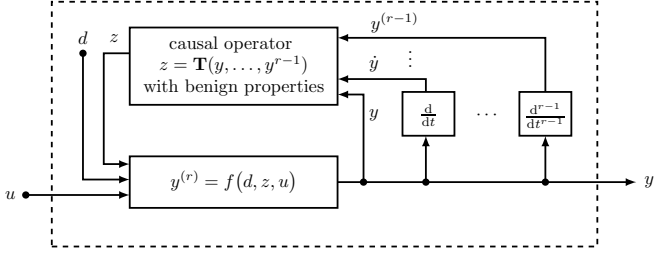
\begin{figure}[b]
\begin{center}
\resizebox{\columnwidth}{!}{
\begin{tikzpicture}[thick,node distance = 12ex, box/.style={fill=white,rectangle, draw=black}, blackdot/.style={inner sep = 0, minimum size=3pt,shape=circle,fill,draw=black},plus/.style={fill=white,circle,inner
sep = 0,thick,draw},metabox/.style={inner sep = 3ex,rectangle,draw,dashed,fill=gray!20!white}]
\tikzset{>=latex}
\tikzset{every path/.append style={line width=1pt}}
\draw(-1,1) rectangle (3,2.5);
\node at (1,1.75){$\begin{array}{c}
\text{causal operator }\\
z= \fT(y,\ldots, y^{r-1})\\
\text{with benign properties}\end{array}$};
\draw(-1,-1) rectangle (3,0);
\node at (1,-.5){$y^{(r)} = f\big(d,z,u\big)$};
\draw (4.25, .4) rectangle (5.25,1.25);
\node at (4.75,.8){$\frac{\text{d}}{\text{d}t}$};
\draw (6.5,.4) rectangle (7.5,1.25);
\node at (7,.8){$\frac{\text{d}^{r -1}}{\text{d} t^{r -1}}$};
\node at (5.9,.8){$\cdots$};
\draw[->] (3,-.5)--(8.75,-.5);
\node at (9,-.5){$y$};
\draw[->](-3,-.75)--(-1,-.75);
\draw[fill] (-3,-.75) circle[radius=1.5pt];
\node at (-3.3,-.75){$u$};
\draw[dashed] (-2.5,-1.75) rectangle (8,3);
\draw[->](-1,2)--(-1.5,2)--(-1.5,-.15)--(-1,-.15);
\draw[->](-1.9,2)--(-1.9,-.45)--(-1,-.45);
\draw[fill] (-1.9,2) circle[radius=1.5pt];
\node at (-1.9,2.3){$d$};
\node at (-1.3,2.3){$z$};
\draw[fill] (3.5,-.5) circle[radius=1.5pt];
\draw[->](3.5,-.5)--(3.5,1.2)--(3,1.2);
\draw[fill] (4.7,-.5) circle[radius=1.5pt];
\draw[->](4.7,-.5)--(4.7,.4);
\draw[->](4.7,1.25)--(4.7,1.5)--(3,1.5);
\draw[fill] (7,-.5) circle[radius=1.5pt];
\draw[->](7,-.5)--(7,.4);
\draw[->](7,1.25)--(7,2.3)--(3,2.3);
\node at (3.7,.8){$y$};
\node at (3.7,1.8){$\dot y$};
\node at (4.1,2.6){$y^{(r-1)}$};
\node at (4.4,2){$\vdots$};
\end{tikzpicture}
}
\end{center}
            \vspace{-0.5cm}
  \caption{Structure of nonlinear functional differential systems}
\label{Fig:gensys}
\end{figure}

Albeit a functional differential form, the advantage of~\eqref{eq83:T} is that it is an
$r$th-order functional differential equation in the variable~$y(\cdot)$ only.
This representation is the key to extending the results to more general
situations, in particular to nonlinear and infinite-dimensional systems with the structure depicted in Fig.~\ref{Fig:gensys}, with
appropriate hypotheses (to be elucidated in due course) on the causal operator~$\fT$ and the nonlinear function~$f$.

\subsubsection{Benign operators}
\noindent
Next, we make precise what we mean by a ``causal operator with benign properties".   Two fundamental requirements are causality and bounded-input, bounded-output behaviour of the operator.
Causality we impose without further comment (other than to say that, throughout,  we assume that the underlying systems are nonanticipative).  Bounded-input, bounded-output behaviour may be regarded as a counterpart of the assumption of asymptotically stable zero dynamics~(SA3).
Linearity of the operator is not required.  Instead, we impose only a local Lipschitz condition which plays a role in ensuring well-posedness of the
underlying system under feedback control.  In particular, we introduce the following class of operators.

\begin{defn}[\textbf{Operator class $\cTT$}]\label{Def:op-class}
For $n,q\in\N$ and $h \geq 0$ the set $\cTT$ denotes the class of operators $\fT\colon \cCC\to \cLL$ with the following properties.
\begin{enumerate}[\hspace{2pt}\textbf{(TP1)}]
\item[\textbf{(TP1)}]
 {\it Causality:} $\fT$ is causal, that is, for all $\zeta$, $\theta \in \cCC$ and all $t\ge 0$,
\[
\zeta|_{[-h,t]} =\theta|_{[-h,t]} ~~\implies~~ \fT(\zeta)|_{[0,t]}=\fT(\theta)|_{[0,t]}.
\]
\item[\textbf{(TP2)}]
 {\it Local Lipschitz property:} for each $t\ge 0$ and all $\xi\in\cC([-h,t],\R^{n})$, there exist positive constants $c_0, \delta, \tau >0$ such that, for all $\zeta_1,\zeta_2 \in \cCC$ with $\zeta_i|_{[-h,t]} = \xi$
and $\|\zeta_i(s)-\xi(t)\|<\delta$ for all $s\in[t,t+\tau]$ and $i=1,2$, we have
\begin{multline*}
\esup\nolimits_{s\in [t,t+\tau]}\|\fT(\zeta_1 )(s)-\fT(\zeta_2) (s)\| \\ \leq c_0 \sup\nolimits_{s\in [t,t+\tau]}\|\zeta_1(s)-\zeta_2(s)\|.
\end{multline*}
\item[\textbf{(TP3)}]
 {\it Bounded-input bounded-output (BIBO) property:}  for each $c_1 >0$, there exists $c_2 >0$ such that, for all $\zeta\in\cCC$,
\begin{multline*}
\sup\nolimits_{t\in[-h,\infty)}\|\zeta(t)\|< c_1 \\
\implies~~ \esup\nolimits_{t\ge 0}\|\fT(\zeta)(t)\| < c_2.
\end{multline*}
\end{enumerate}
\end{defn}
\noindent This formulation embraces nonlinear delay elements and hysteretic effects, as we shall briefly illustrate.
\\[1ex]
{\em Nonlinear delay elements.} \quad
For $i=0,\ldots,k$, let $\Psi_i\colon\R\times \R^m\to \R^q$ be measurable in its
first argument and locally
Lipschitz in its second argument, uniformly with respect to its first argument.  Precisely, 
 for each $\xi\in\R^m$, $\Psi_i(\cdot,\xi)$ is
measurable, and
for every compact $C\subset\R^m$, there exists~$c>0$ such that for a.a.\ $t\in\R$ we have
\[
\forall\, \xi_1,\xi_2\in C: \
\|\Psi_i(t,\xi_1)-\Psi_i(t,\xi_2)\|\leq c\|\xi_1-\xi_2\|.
\]
Let $h_i >0$, $i=0,\ldots,k$, and set $h:= \max_i h_i$.
For $y\in \cC([-h,\infty),\R^m)$ and $t\ge 0$ let
\[
\fT(y)(t):=
\int_{-{h_0}}^0 \Psi_0(s,y(t+s))\, \dd s
+  \sum_{i=1}^k \Psi_i (t,y(t-h_i)).
\]
The operator~$\fT$, so defined (which models distributed and point delays),
is of class~${{\mathbb T}^{m,q}_{h}}$; for details, see~\cite{RyanSang01}.
\\[1ex]
{\em Hysteresis.}  \quad
 A large class of nonlinear operators
$\fT:\cC(\R_{\ge 0},\R)\to \cC(\R_{\ge 0},\R)$,
which includes many physically-motivated
hysteretic effects, is defined in~\cite{LogeMawb00}.  These operators are contained in the class~${{\mathbb T}^{1,1}_{0}}$ of the present paper.
Specific examples include relay hysteresis, backlash hysteresis, elastic-plastic hysteresis, and Preisach operators.
For further details, see~\cite{IlchRyan02a}.

\subsubsection{Admissible nonlinearities}
\noindent
Next, and with reference to Figure \ref{Fig:gensys}, we proceed to make precise the admissible nonlinearities~$f$.

\begin{defn}[\textbf{Class of nonlinearities $\cNN$}]\label{adnon}
For $p,q,m\in\N$ the set $\cNN$ denotes the class of
functions $f\in\cC (\R^p\times\R^q \times\R^m, \R^m)$ with the following property.
\begin{enumerate}[\hspace{2pt}\textbf{(NP1)}]
\item[\textbf{(NP1)}]
  {There exists $v^*\in(0,1)$ such that, for every compact  $K_p\subset \R^p$ and compact $K_q\subset\R^q$ the continuous function $\chi\colon\R\to\R$ defined by
\[
\hspace*{-3mm} s \!\mapsto\! \min\setd{\! \langle v,f(\delta,z,-s v)\rangle }{\!\!\!\begin{array}{l} (\delta,z)\in K_p\times K_q,\\
v\in\R^m,~v^* \leq \|v\| \leq 1 \end{array}\!\!\!}
\]
is such that $\sup_{s\in\R} \chi(s)=\infty$.}
\end{enumerate}
\end{defn}

\noindent
 Property~{(NP1)}   may appear somewhat arcane.
 It becomes more transparent when interpreted in a linear context wherein the following holds~\cite{BergIlch21}.

 \begin{prop}\label{np1-equiv}
Let~$L_1\in\R^{m\times p}$, $L_2\in\R^{m\times q}$ and~$\Gamma\in\R^{m\times m}$.
Then the linear map $f:\R^p\times\R^q \times\R^m\to \R^m,\ (\delta,z,v)\mapsto L_1\delta+L_2 z+\Gamma v$,
satisfies:
\[   \text{$f$ has  property~{\rm (NP1)}}
 \iff
\text{$\Gamma$ \ is sign-definite.}
\]
\end{prop}
\noindent
Thus, {(NP1)} may be regarded as a nonlinear generalization of~(SA3).
If a function $f$ satisfies {(NP1)}, then, for {\em each} pair of compact sets $K_p\subset\R^p$ and $K_q\subset\R^q$, there exists $\sigma\in\{-1,+1\}$
such that
\[
\chi(\sigma s)\to\infty~~\text{as}~ s\to\infty.
\]
If $\sigma\in\{-1,+1\}$ is such that the above holds for {\em all} sets~$K_p$ and~$K_q$, then we refer to~$\sigma$ as the {\it control direction} (a natural
analogue of the term introduced in the context of the prototype linear system~\eqref{abc}.

\subsubsection{Class ${\mathcal N}^{m,r}$ of functional differential systems amenable to funnel control}\label{Sec:Nmr}
\noindent
We summarize the above discussion with the following characterization of a class of nonlinear functional differential systems which will be shown to be amenable to control by funnel techniques.
The system   representative
of this class, parametrized by~$m,r\in\N$, takes the form
\begin{equation}\label{eq:nonlSys}
y^{(r)}(t)= f\big(d(t), \fT(y,\dot{y},\dots,y^{(r-1)})(t), u(t)\big),
\end{equation}
 with initial data
\begin{equation}\label{eq:nonlSysIc}
\left.\begin{array}{ll}
y|_{[-h,0]}= y^0\in \cC^{r-1}([-h,0], \R^m), & \text{if~$h >0$,}
\\[1ex]
\big(y(0),\cdots, y^{(r-1)}(0)\big)=\big(y^0_1,\cdots,y^0_r\big)\in\R^{rm}, &\text{if~$h=0$,}
\end{array}
\!\!\right\}
\end{equation}
where~$h\ge 0$ quantifies the ``memory" in the system and,
for some~$p,q\in\N$, $f\in \cNN$,
$\fT\in {\mathbb T}^{rm,q}_h$, and~$d\in\cL^\infty(\R_{\ge 0},\R^p)$.
The representative system may be identified with a triple~$(d,f,\fT)$ and so we write
\[
{\mathcal N}^{m,r} \!:=\! \setd{(d,f,\fT) }{\!\!\!\begin{array}{l}  d\in\cL^\infty(\R_{\ge 0},\R^p),\ f\in\cNN,\\ \fT\in {\mathbb T}^{rm,q}_h,\ p,q\in\N,\ h\geq 0\end{array}\!\!\!}\!.
\]
We show that the class of linear systems~$\lincl$,
 as defined in~\eqref{eq:Lmr},
  is indeed contained in the class~$\cN^{m,r}$ for any~$m, r\in\N$, for which we recall that~\eqref{eq:ABC} is equivalent to~\eqref{eq83:T}.
\begin{lem}\label{Lem:lincl-Nmr}
  Let~$(A,B,C)\in \lincl$ for some~$m, r\in\N$, with associated
  Byrnes-Isidori   form~\eqref{eq:ABC_BIform}.
 Let the operator~$\fT$ be as in~\eqref{eq:opT-lin}.
 Define $f\colon \R^m\times\R^m\times\R^m\to\R^m,~(\delta, z, u)\mapsto \delta + z + \Gamma u$.
 Let $\eta^0\in\R^{n-mr}$ be arbitrary and define  $d(\cdot) := S{\rm e}^{Q\,\cdot} \eta^0$.
  Then $(d,f,\fT)\in \cN^{m,r}$.
\end{lem}
\begin{proof} Clearly, $d$ is bounded by~{(SA2)} and Proposition~\ref{Prop:zero_dyn-srd}. We show that $\fT \in {\mathbb T}^{rm,m}_0$. It is easy to see that the operator~$\fT$ satisfies properties~(TP1) and~(TP2) of the class~${\mathbb T}^{rm,m}_0$. The BIBO property~(TP3) is closely related to
property~{(SA2)} of the system $(A,B,C)$. First observe that the transfer function $C(sI-A)^{-1}B \in\R(s)^{m\times m}$ of $(A,B,C)$ is invertible over~$\R(s)$ by~\eqref{BIformG}, since $\Gamma$ is invertible. It is then shown in~\cite[Sec.~2.12]{BergIlch21} that $\fT$ satisfies~(TP3).

  Finally, to conclude that $f\in\fN^{m,m,m}$,
  it suffices to note that, by~{(SA3)} and Proposition~\ref{np1-equiv}, (NP1) holds.
\end{proof}
\subsubsection{Input nonlinearities}\label{Ssec:InputNonl}
\noindent
In addition to accommodating  the issue of (unknown)
control direction, the generic formulation~\eqref{eq:nonlSys} with associated property (NP1) encompasses a wide variety of input nonlinearities.  Consideration of a scalar
system of the form
\begin{equation}\label{eq:f1f2}
 \dot y(t)=f_1(y(t))+f_2(y(t)) \, \beta (u(t))
\end{equation}
with $f_1\in \cC(\R,\R)$, $f_2\in \cC(\R,\R\backslash\{0\})$ and $\beta\in \cC(\R,\R)$, will serve to illustrate this variety.  The assumption that $f_2$ is a non-zero-valued continuous
function ensures a well-defined control direction (unknown to the controller). Without loss of generality, we may assume that~$f_2\in \cC(\R,\R_{>0})$;
if~$f_2$ is negative-valued, then, in~\eqref{eq:f1f2}, simply replace~$f_2$ by~$-f_2$ and~$\beta$ by~$-\beta$.
We impose the following conditions on~$\beta\in \cC(\R,\R)$:
\begin{equation}\label{betaprops}
\beta {\text{ is surjective, with}}~~|\beta (\tau)| \to \infty ~~{\text{as}}~|\tau|\to\infty,
\end{equation}
which is equivalent to requiring that one of the following conditions hold:
\[
   \lim_{\tau\to\pm \infty} \beta(\tau) = \pm \infty\quad \text{or} \quad \lim_{\tau\to\pm \infty} \beta(\tau) = \mp \infty.
\]
Under these conditions, it is shown in~\cite[Sec.~2.3]{BergIlch21} that system~\eqref{eq:f1f2} has property~(NP1). Specific examples of functions~$\beta$ satisfying the above requirements are linear functions $\beta(u) = au +b$ with $a,b\in\R$ and $a\neq 0$, signed squares $\beta(u) = a \sgn(u) u^2$ with $a\neq 0$, dead-zone input nonlinearities (as discussed in the following subsection) and combinations thereof.

\subsubsection{Dead-zone input} \label{Ssec:DeadZone}
\noindent
An important example of a nonlinearity
$\beta=D$ with properties~\eqref{betaprops} is a so-called \emph{dead-zone input}  of the form
\[
   D : \R\to \R, \quad v\mapsto
    D(v) = \left\{
    \begin{array}{rcl} D_r(v), && v\ge b_r,\\ 0,&& b_l<v<b_r,\\ D_l(v), && v\le b_l\end{array}\right.
\]
with unknown deadband parameters~$b_l<0<b_r$ and unknown functions $D_l,D_r\in\cC(\R,\R)$ which satisfy,
 for unknown $\sigma\in\{-1,1\}$,
\begin{multline*}
D_l(b_l) = D_r(b_r) = 0 \\ \text{and} \qquad
\lim_{s\to\infty} \sigma D_r(s) = \infty, \quad \lim_{s\to-\infty} \sigma D_l(s) = -\infty.
\end{multline*}
Note that the above assumptions allow for a much larger class of
functions~$D_l, D_r$ compared to e.g.~\cite{Na13}, where assumptions on their derivatives are used.
In particular,~in the present context, $D_l$ and~$D_r$ need not be differentiable or monotone.

\subsubsection{A special subclass of ${\mathcal N}^{m,r}$}\label{Sssec:special-sc}
\noindent
For later use in the context of funnel control with non-derivative feedback, we introduce a subclass of ${\mathcal N}^{m,r}$.
The systems to be studied are affine in the control and are represented by
functional differential equations, with $\R^m$-valued input~$u$ and output~$y$, of the form
\[
y^{(r)}(t)=\hat f(d(t),\hat{\mathbf T}(y,\dot y,\ldots,y^{(r-1)})(t))+\Gamma u(t),
\]
where $\Gamma\in\Gl_m(\R)$,
$\hat{f}\in\cC(\R^{p}\times \R^{\hat q},\R^m)$ and
$\hat{ \mathbf{T}}\in {\mathbb T}_h^{rm,\hat q}$, $\hat q > rm$.
 We impose additional structural assumptions as follows.
First, it is assumed that $\hat{\mathbf T}$ is of the (highly structured) form given by
\[
\hat{\mathbf T}(\zeta_1,\ldots,\zeta_r)=\big(\zeta_1,\ldots,\zeta_r,{\mathbf T}(\zeta_1,\ldots,\zeta_r)\big)
\]
where ${\mathbf T}\in{\mathbb T}_h^{rm,q}$ ($q=\hat q -rm$) satisfies
\\[1ex]
{\bf (TP3')}\quad for all $c_1>0$ there exists $c_2>0$ such that for all\\
\hspace*{9ex} $\zeta_1,\ldots,\zeta_r \in \cC([-h,\infty),\R^{m}):$
      \[
        \sup_{t\in[-h,\infty)} \|\zeta_1(t)\| \le c_1\ \ \Longrightarrow\ \ \sup_{t\in[0,\infty)} \|\mathbf{T}(\zeta_1,\ldots,\zeta_r)(t)\| \le c_2.
      \]
Secondly, the function $\hat f\in {\mathcal C}(\R^{p}\times\R^{mr}\times\R^q,\R^m)$ is assumed to take the form
\[
\hat f  (d,\zeta,\eta)=\hat f(d,\zeta_1,\ldots,\zeta_r,\eta)=\sum_{i=1}^r  R_i\zeta_i+f(d,\eta),
\]
where $f\in\cC (\R^p\times\R^q,\R^m)$ and $R_i\in\R^{m\times m},\ i=1,\ldots,r.$ Thirdly, $\Gamma$ is assumed to be sign definite:  $|\langle v,\Gamma v\rangle|>0$ for all $v\neq 0$.
In summary, with $r\geq 2$, the generic system to be investigated is
\begin{multline}\label{eq:Sysmeth}
 y^{(r)}(t) = \sum_{i=1}^{r} R_i y^{(i-1)}(t)\\ + f\big(d(t), \mathbf{T}(y,\dot y,\ldots,y^{(r-1)})(t)\big) + \Gamma u(t),
 \end{multline}
 with initial data
 \begin{equation}\label{eq:Sysmethic}
 \left.
  \begin{array}{ll}
   \left.y\right|_{[-h,0]} = y^0\in \cC^{r-1}([-h,0],\R^m), &\!\text{if}~h >0,
   \\[1ex]
   (y(0),\dot y(0),\ldots,y^{r-1}(0))=(y_1^0,y_2^0,\ldots,y_{r-1}^0), &\!\text{if}~h=0,\end{array}\!\!\right\}
   \end{equation}
   where $\Gamma\in\Gl_m(\R)$ is sign definite, $R_1,\ldots,R_r\in\R^{m\times m}$, $f\in\cC(\R^p\times \R^q,\R^m)$, $\mathbf{T}\in {\mathbb T}_h^{rm,q}$ such that~{(TP3')} holds, and
   the disturbance $d$ is essentially bounded.

   \begin{rem}
   The assumption that the generic system is affine in the control can be weakened.  Assume instead that the input enters via a function $g\in\cC(\R^m,\R^m)$ and posit the existence of
   a sign-definite $\Gamma\in\Gl_m(\R)$ such that $v\mapsto g(v)-\Gamma v$ is bounded (which, for example, permits dead zone effects), then, for any  input $u(\cdot)$
   (of class $\cL^\infty_{\loc}$), the function
   $d_u\colon t\mapsto g(u(t))-\Gamma u(t)$ is essentially bounded and so the system with input operator $g$ is subsumed by the form ~\eqref{eq:Sysmeth} on
   replacing $f$ by the $\cC(\R^{p+m} \times \R^q,\R^m)$-function $\big((d_1,d_2),\eta\big)\mapsto f(d_1,\eta)+d_2$.
   \end{rem}

\subsection{Differential-algebraic systems}
\noindent
In the last decades the interest in control design for systems described by differential-algebraic equations (DAEs) steadily increased. In the simplest case, those equations are combinations of differential equations with algebraic constraints, restricting the dynamics to certain subspaces or submanifolds of the state space. However, in general the constraints are not obvious and may also impose restrictions on the possible choices of input functions or, at the other extreme, completely free variables are possible which may occur in the output. Therefore, a thorough analysis of~DAEs is necessary and we refer to the textbooks~\cite{BrenCamp89,KunkMehr06,LamoMarz13}, to name but a few.

\subsubsection{Linear differential-algebraic systems}\label{Sssec:DAEs-lin}
%
\noindent
Here, we focus on linear differential-algebraic system given by the equations
\begin{equation}\label{eq:EABC}
\begin{aligned}
  \ddt E x(t) &= A x(t) + B u(t),\\
  y(t) &= C x(t),
\end{aligned}
\end{equation}
where $E,A\in\R^{{n}\times n}, B\in\R^{{n}\times m}, C\in\R^{{m}\times n}$; we write $[E,A,B,C]\in\Sigma_{n,m}$.
We allow for singular~$E$. In the extreme case of~$E=0$,
\eqref{eq:EABC} consists only of algebraic equations.

Solutions~-- we define in due course what a solution is~--
 exhibit quite different features compared to linear ODE systems~\eqref{eq:ABC}. Consider the linear~DAE system {(in $\Sigma_{2,1}$)}
\begin{equation}\label{eq:DAE-ex-rd}
\begin{aligned}
   \frac{\dd}{\dd t} \begin{bmatrix} 0 & 0\\ 1 & 0\end{bmatrix} \begin{pmatrix} x_1(t) \\ x_2(t) \end{pmatrix} &= \begin{bmatrix} 1 & 0\\ 0 & 1\end{bmatrix} \begin{pmatrix} x_1(t) \\ x_2(t) \end{pmatrix} + \begin{bmatrix} -1 \\ 0\end{bmatrix} u(t),\\
    y(t) &= \begin{bmatrix} 0 & 1\end{bmatrix}  \begin{pmatrix} x_1(t) \\ x_2(t) \end{pmatrix},
\end{aligned}
\end{equation}
which can be reformulated as $y(t) = \dot u(t)$. Therefore,
this system does not have a free input,
the latter must be differentiable;
the state is not determined by~$u$ but
the derivative of~$u$ determines~$x_2$.

Moreover, it is necessary to {re-visit} the concept of relative degree given in Definition~\ref{Def:rel-deg}{: }for system~\eqref{eq:DAE-ex-rd}, a relative degree in the sense of Definition~\ref{Def:rel-deg} does not exist. First, we may observe that it is possible to extend the definition of a transfer functions to~DAE systems~\eqref{eq:EABC}, where the  so-called {\it matrix pencil} $sE-A\in\R[s]^{n\times n}$ is {\it regular}, i.e., $\det (sE-A) \in\R[s]\setminus\{0\}$. In this case, $sE-A$ is invertible over the quotient field $\R(s)$ and we may define the transfer function by
\[
    G(s) := C (sE-A)^{-1} B \in\R(s)^{{m}\times m}.
\]
For single-input, single-output systems (as discussed in Section~\ref{Ssec:rel-degree}), the relative degree equals the difference between the degrees of the denominator and numerator polynomials in the transfer function $G(s) = p(s)/q(s)$. For system~\eqref{eq:DAE-ex-rd}, the transfer function can be computed as
\[
    G(s) = \begin{bmatrix} 0 & 1\end{bmatrix} \begin{bmatrix} -1 & 0\\ s & -1\end{bmatrix} \begin{bmatrix} -1 \\ 0\end{bmatrix} = s,
\]
thus $p(s) = s$ and $q(s) = 1$ which yields a relative degree $r = \deg q(s) - \deg p(s) = -1$. In fact, for differential-algebraic systems a negative relative degree is quite common, which means that the underlying system contains a chain of differentiators (instead of integrators as for ordinary differential equations with positive relative degree). For general differential-algebraic systems, it is possible to extend the notion of relative degree to $r\in\Z$. Then again, this enables us to derive a decomposition of the system which exposes the underlying chains of integrators and differentiators as well as the zero dynamics; this generalizes the Byrnes-Isidori form, see Remark~\ref{Rem:ext-BIF} below.

In the current {\em linear}  context, the appropriate solution concept for differential-algebraic equations
is that of the \textit{behavioural approach},  introduced by Jan~C Willems~\cite{Will81} (see also \cite{PoldWill98,Will07}), wherein the {\em behaviour} of $[E,A,B,C]{\in \Sigma_{n,m}}$ is defined as
\begin{align*}
    &\fB_{[E,A,B,C]}\\
     &:=\! \setd{\!\!\!\!\begin{array}{l} (x,u,y)\\ \in\!\cL^1_{\loc}(\R_{\ge 0},\R^n\!\times\!\R^m\!\times\! \R^{m})\end{array}\!\!\!}{\!\!\!\begin{array}{l} Ex\in\cA\cC_{\loc}(\R_{\ge 0},\R^{n}),\\
     \text{\eqref{eq:EABC} holds for}\\ \text{a.a.\ $t\!\ge\! 0$}\end{array}\!\!\!\!}\!.
\end{align*}
The zero dynamics
$\mathcal{ZD}(E,A,B,C)$ of~\eqref{eq:EABC} are defined, similar to linear ODE systems, as those {elements $(x,u,y)$ of $\fB_{[E,A,B,C]}$
for which the output $y$ is (almost everywhere) zero:
\begin{align*}
&\mathcal{ZD}(E,A,B,C)\\
&:= \setdef{\!(x,u)\!\in\!\cL^1_{\loc}(\R_{\ge 0},\R^n\!\times\!\R^m )\!}{\!(x,u,0)\!\in\! \fB_{[E,A,B,C]}\!}.
\end{align*}
Analogous to the ODE case, the zero dynamics are said to be {\it bounded}, if
$(x,u)\in\cL^\infty(\R_{\ge 0},\R^n\times\R^m)$ for all $(x,u)\in {\mathcal {ZD}}(E,A,B,C)$  and are said to be {\it asymptotically stable}, if
$\text{ess\,sup}_{\tau\geq t}\|(x(\tau),u(\tau))\| \to 0$ as $t\to\infty$ for all $(x,u)\in {\mathcal {ZD}}(E,A,B,C)$.}
It is shown in~\cite[Lem.~3.11]{Berg16b} that the zero dynamics  ${\mathcal {ZD}}(E,A,B,C)$ are asymptotically stable if, and only if,
\[
    \forall\, \lambda\in{\mathbb{C}}_{\ge0} :\
   {\det \begin{bmatrix}A- \lambda E  & B\\ C & 0\end{bmatrix}\neq 0.}
\]
Another crucial system property, in particular for control purposes, is that every
 smooth function $\R_{\ge 0}\to\R^{m}$ can be generated as the output of the system for some appropriate input. This leads to the notion of right invertibility {(which has been introduced and analyzed for ODE systems e.g.\ in~\cite{Resp90,SannSabe87}, see also the textbook~\cite[Sec.~8.2]{TrenStoo01})}; we call $[E,A,B,C]\in{\Sigma_{n,m}}$ \textit{right invertible}, if
\begin{multline*}
    \forall\, y\in\cC^\infty(\R_{\ge 0},\R^{m})\ \exists\, (x,u)\in\cL^1_{\loc}(\R,\R^n\times \R^m):\\ (x,u,y)\in\fB_{[E,A,B,C]}.
\end{multline*}
For a right-invertible system $[E,A,B,C]\in\Sigma_{n,m}$ with asymptotically stable zero dynamics,
a distillation of results from \cite{Berg14a} (in particular, Lemma 4.2.5, Theorem 4.2.7, Proposition 4.2.12 \& Remark 4.3.10 therein; see also
\cite[Section 3]{Berg16b}) establishes that~\eqref{eq:EABC} is equivalent to
\begin{equation}\label{eq:ZDF-EABC}
\left.\begin{aligned}
x_2(t) &= \sum_{k=0}^{\nu-1} N^k E_{11}  x_1^{(k+1)}(t),\\
0 &= A_{21}x_1(t) - E_{21} {\dot x_1}(t) - E_{22}\dot x_2(t)\ \\
&\quad  +A_{23} {x_3}(t)+ u(t),\ \ \\
  \dot x_3(t) &= Qx_3(t) + A_{31} {x_1}(t),\\
  y(t) &= x_1(t),
  \end{aligned}\right\}
\end{equation}
where {$x_1(t)\in\R^{m}$, $x_2(t)\in\R^{n_2}$, $x_3(t)\in\R^{n_3}$ {with $n_2=\nu m$ and $n_3=n-(\nu+1)m$,
$N\in\R^{n_2\times n_2}$ is nilpotent with
index of nilpotency $\nu$, and all other matrices are of conforming size.  Moreover, $Q\in\R^{n_3\times n_3}$ is Hurwitz, that is, $\sigma(Q)\subset{\mathbb{C}}_{<0}$.
\begin{rem}\label{Rem:ext-BIF}
  The form~\eqref{eq:ZDF-EABC} is a generalization of the Byrnes-Isidori form~\eqref{eq:ABC_BIform} of linear systems $(A,B,C)$.
  More precisely, assume that $E$ in~\eqref{eq:EABC} is invertible (without loss of generality, we may assume that $E=I$) and so, in
   Byrnes-Isidori form, the system {(of relative degree $r$)} is given by \eqref{eq:ABC_BIform}.  Setting {$n_2=(r-1)m$} and writing
   \[
  N =\begin{bmatrix}0&0&\cdots&0&0
   \\
   I_m&0&\cdots&0 &0
   \\
   0& I_m&\cdots&0&0\\\vdots&\vdots&\ddots&\vdots&\vdots
   \\
   0&0&\cdots&I_m&0\end{bmatrix},
  \quad E_{11}=\begin{bmatrix} I_m\\0\\0\\\vdots\\0\end{bmatrix},
   \]
   (each being vacuous if $r=1$ and, for $r >1$, $N$ is nilpotent with index $\nu=r-1$), we have
   \[
  \sum_{k=0}^{\nu-1} N^k E_{11} { x_1}^{(k+1)}(t)=
  \begin{smallpmatrix} \dot y(t)\\\vdots\\y^{(r-1)}(t)\end{smallpmatrix}
  \]
  and so the first of relations \eqref{eq:ZDF-EABC} is simply a re-affirmation of the first set of $r-1$  equations in~\eqref{eq:ABC_BIform}.
  The second {set}  of equations \eqref{eq:ABC_BIform} can be re-written as
  \begin{multline*}
  0=\Gamma^{-1}\bigg(R_1y(t)+R_2\dot y(t)+\sum_{k=3}^r R_ky^{(k-1)}(t) \\
  -y^{(r)}(t)+S\eta(t)\bigg)+u(t)
  \end{multline*}
  which, on setting $A_{21}=\Gamma^{-1}R_1$, $E_{21}=-\Gamma^{-1}R_2$, $E_{22}=-\Gamma^{-1}
  [R_3, \ldots, R_r, -I_m]$ and $A_{23} =\Gamma^{-1}S$, coincides with the second of equations \eqref{eq:ZDF-EABC}.
  Finally, on setting $A_{31}=P$, the third of equations \eqref{eq:ABC_BIform} and \eqref{eq:ZDF-EABC} coincide.
  In summary, we have shown that, in the case of invertible~$E$, the form~\eqref{eq:ZDF-EABC} of system~\eqref{eq:EABC}
  {is equivalent to} its Byrnes-Isidori form~\eqref{eq:ABC_BIform}.
\end{rem}

Returning to the general case of {right-invertible} systems $[E,A,B,C]\in\Sigma_{{n,m}}$ {with} asymptotically stable zero dynamics,
and adopting the ``Byrnes-Isidori'' form~\eqref{eq:ZDF-EABC},  we see that, by nilpotency of $N$,  $(sN-I_{n_2})^{-1} = -\sum_{k=0}^{\nu-1} N^k s^{k}$.
Define
\begin{multline*}
-A_{21}+sE_{21} +\sum_{k=0}^{\nu-1} E_{22}N^k E_{11}s^{k+2}\\
-A_{23} (sI_{n_3}-Q)^{-1} A_{31} =: H(s)\in\R(s)^{m\times m}
\end{multline*}
and observe that (cf.\ also~\cite[Rem.~A.4]{Berg16b}), if $sE-A$ is regular with invertible transfer function $G(s) = C(sE-A)^{-1} B$, then $G(s)^{-1}=H(s)$.
We define the degree of a
vector of rational functions $h(s)=\big(p_1(s)/q_1(s) ,\ldots,p_m(s)/q_m(s)\big)^\top \in \R(s)^{m}$ by
\[
\deg h(s):= \max_{i=1,\ldots,m}\big(\deg p_i(s)-\deg q_i(s)\big).
\]
Let $h_i(s)$, $i=1,\ldots,m$, denote the columns of $H(s)$ and write
$ r_i := \max\{\deg h_i(s),0\}$, $i=1,\dots,m$: right-multiplication of $H(s)$ by a permutation matrix  $P\in\R^{m\times m}$
(corresponding to a re-ordering of the components of the system output, if necessary) ensures that,  without loss of generality,
we may assume the ordering $r_1\geq \cdots\geq r_m~(\geq 0)$.
Observe that the following are well defined: $\lim_{s\to\infty}s^{-r_i}h_i(s)=: \hat h_i\in\R^m$, $i=1,\ldots,m$.  Write
\begin{align}\notag
\Gamma_H&:= \lim_{s\to\infty} H(s)\text{diag}\big(s^{-r_1}, \ldots, s^{-r_m}\big)\\
&=\begin{bmatrix} \hat h_1, \ldots, \hat h_m\end{bmatrix}\in\R^{m\times m}\label{eq:Gamma}
\end{align}
Let $\ell \in \{1,\ldots,m\}$
be such that,
for all $i\in \{1,\ldots,m\}$,  $r_i=0$ implies $i> \ell $.  Define
\begin{equation}\label{eq:Gamma-l}
\Gamma_\ell  := \begin{bmatrix} \hat h_1 , \ldots, \hat h_\ell \end{bmatrix}\in\R^{m\times \ell }.
 \end{equation}
As introduced in~\cite{BergLe20}, the $m$-tuple~$(r_1,\ldots,r_m)$ is said to be the \textit{truncated vector relative degree} of~$[E,A,B,C]$,
 if~$\rk\Gamma_\ell  =\ell $.
\begin{rem}
At first glance, it might seem more natural to call the
$m$-tuple $(r_1,\ldots,r_m)$ the vector relative degree and to call the
$\ell $-tuple $(r_1,\ldots,r_\ell )$ the truncated vector relative degree.  However, a concept of ``vector relative degree" already exists for DAE systems (see Def.~2.7 in~\cite{BergLe20}) which differs from  $(r_1,\ldots,r_m)$ insofar as it may also contain negative entries: the terminology
``truncated" refers to the extant notion of vector relative degree with its negative terms excised.
\end{rem}

Although the situation of arbitrary truncated vector relative degree is extensively explored in~\cite{BergLe20},
for purposes of {exposition} we restrict ourselves {here} to the case of truncated \textit{strict} relative degree, that is,
we assume that there exists $r\in\N$ such that $r_1 = \ldots = r_\ell  = r$ and $r_{\ell +1} = \ldots = r_m = 0$; this relative degree is
denoted by the pair $(r,\ell )$.
Observe that, if $\rk \Gamma_\ell =\ell $, then (invoking a suitable re-ordering of the components of the system input if necessary),   we may assume,
without loss of generality, that $\Gamma_\ell $ takes the form
\begin{equation}\label{eq:Gamma2}
\Gamma_\ell  = \begin{bmatrix} \hat \Gamma\\ \tilde \Gamma \end{bmatrix}\quad\text{with}
\quad \hat \Gamma\in\Gl_\ell (\R).
 \end{equation}
\begin{rem}
The concept of truncated strict relative degree generalizes the concept of relative degree for linear systems $(A,B,C)$ introduced in Definition~\ref{Def:rel-deg}.
To see this, {let $E=I$ in \eqref{eq:EABC} and assume that $(A,B,C)$ has relative degree $r\in\N$, i.e., (SA1) holds.
Then $\Gamma = C A^{r-1} B \in \R^{m\times m}$ is invertible and for $F(s):= s^r G(s)\in\R(s)^{m\times m}$ we have that $F(s) = \Gamma + \tilde G(s)$,
where $\tilde G(s)$ is strictly proper, i.e., $\lim_{s\to\infty} \tilde G(s) = 0$
and so the degree of each of its elements is not greater than $-1$: $\deg \tilde G(s)_{ij} \le -1$, $i,j=1,\ldots,m$.  We show that $F(s)$ is invertible over $\R(s)$. Let {$\rho (s)=(\rho_1(s),\ldots,\rho_m(s))^\top \in\R(s)^m$} be such that $F(s) {\rho} (s) = 0$.
Let $J:=\setdef{j\in\{1,\ldots,m\}}{\rho_j\neq 0}$ and so $\rho_j(s)=p_j(s)/q_j(s)$, $p_j(s)\neq 0$, for all $j\in J$.
Seeking a contradiction, suppose that $J\neq\emptyset$.
Since ${\rho} (s) = -\Gamma^{-1} \tilde G(s) {\rho} (s)$ and $\deg \tilde G(s)_{ij}\leq -1$,  we have
\[
\deg{\rho} _i(s) \!=\! \deg \sum_{j\in J} \big( -\Gamma^{-1} \tilde G(s)\big)_{ij} {\rho}_j(s) \!\le\! -1+ \max_{j\in J} \deg {\rho}_j(s)
\]
for all $i\in J$, and so, for some $j\in J$,  we arrive at the contradiction
\[
\deg p_j(s)-\deg q_j(s) \leq -1 +\deg p_j(s)-\deg q_j(s).
\]
Therefore, $\rho(s)=0$ and so $F(s)^{-1}\in\R(s)^{m\times m}$. It follows that $G(s)$ is invertible and so, recalling~\eqref{BIformG},
\begin{align*}
    H(s) &= G(s)^{-1}\\
    &=-\Gamma^{-1}\left( \sum_{i=1}^r R_i s^{i-1} - s^rI+ S(sI-Q)^{-1} P\right).
\end{align*}
Clearly, each column $h_i(s) = H(s) e_i$ has degree $\deg h_i(s) = r$ for $i=1,\ldots,m$ and so $q=m$.  Moreover,  $\Gamma_\ell $ is invertible:
\[
    \Gamma_\ell  = \lim_{s\to\infty} s^{-r} H(s) = \Gamma^{-1}.
\]
Therefore, $[I,A,B,C]$ has truncated strict relative degree~$(r,m)$.
}
\end{rem}

Returning to the general context of differential-algebraic systems of form \eqref{eq:EABC}, we posit
the following {structural} assumptions:
\begin{enumerate}[\hspace{2pt}\textbf{(TP1)}]
\item[\textbf{(DA1)}]
  $[E,A,B,C]$ is right-invertible and has asymptotically stable zero dynamics,
\item[\textbf{(DA2)}]
 $[E,A,B,C]$ has a truncated strict relative degree $(r,\ell )$ which is known to the controller,
 \item[\textbf{(DA3)}]
 ${\hat\Gamma}$ is sign-definite.
\end{enumerate}
We now introduce a class of~DAEs, which will be shown to be amenable to funnel control,
\[
    \mathcal{LD}^{m,r,\ell }
     \!:=\! \setdef{\![E,A,B,C]\in\Sigma_{n,m}\!\!}{\!\!\!\begin{array}{l}{n\in\N},\ \text{(DA1),\,(DA2),}\\ \text{and (DA3) hold}\end{array}\!\!\!\!\!}\!.
\]

\begin{rem}
If $[I_n,A,B,C]\in\Sigma_{n,m}$, then it is readily verified that
Assumptions (DA1), (DA2), (DA3) all hold if, and only if, Assumptions (SA1), (SA2), (SA3) all hold.  Therefore,
\[
    \setdef{(A,B,C)}{[I_n,A,B,C]\in \mathcal{LD}^{m,r,\ell }} = \mathcal{L}^{m,r},
\]
where the latter is defined in~\eqref{eq:Lmr}.
\end{rem}

As shown in~\cite[Sec.~2.3]{BergLe20}, a system $[E,A,B,C]\in  \mathcal{LD}^{m,r,\ell }$ is equivalent to
\begin{equation}\label{eq:NF-EABC}
\begin{aligned}
 y_I^{(r)}(t) &= \sum_{k=1}^r R_{k,1} y_I^{(k-1)}(t) + P_1 y_{II}(t) + S_1 x_{3}(t)
 \\
 &\quad + \hat \Gamma u_I(t),\\
  0 &= \sum_{k=1}^r R_{k,2} y_I^{(k-1)}(t) + P_2 y_{II}(t) + S_2 x_{3}(t)  \\
  &\quad + \tilde \Gamma u_I(t) + u_{II}(t),
  \\
  \dot x_3(t) &= Qx_3(t) + A_{31} y(t),
\end{aligned}
\end{equation}
where\\
 $y_I = (y_1,\ldots,y_\ell )\in\R^\ell $,\quad
 $y_{II} = (y_{\ell +1},\ldots,y_m)\in\R^{m-\ell }$,\\
  $u_I = (u_1,\ldots,u_\ell )\in\R^\ell $,\quad
$u_{II} = (u_{\ell +1},\ldots,u_m)\in\R^{m-\ell }$.

\subsubsection{Nonlinear differential-algebraic systems}\label{Sssec:DAEs-nonl}
\noindent
Similar to the extension of the Byrnes-Isidori form~\eqref{eq:ABC_BIform} to the nonlinear functional differential systems~\eqref{eq:nonlSys}, the representation~\eqref{eq:NF-EABC} of linear~DAE systems can be extended to incorporate a class of nonlinear~DAE systems, cf.~\cite{BergIlch14,BergLe20}. For motivation, consider $[E,A,B,C]\in  \mathcal{LD}^{m,r,q}$
and assume that its representation is in form~\eqref{eq:NF-EABC}.  Analogous to \eqref{eq:QPS}, we  introduce the  linear operator
\begin{align*}
L\colon &\cC (\R_{\ge 0},\R^m)  \to  \cC(\R_{\ge0},\R^{m}),\\
&y \mapsto \left( t \mapsto   \int_0^t  {\rm e}^{Q(t-\tau)} A_{31}y(\tau)\, \dd \tau \right).
\end{align*}
Define operators
\begin{align*}
\fT_1\colon &\cC (\R_{\ge 0}, \R^{\ell }\times\cdots\times\R^\ell \times\R^{m-\ell })\to\cC(\R_{\ge 0},\R^m),
\\
&(\zeta_1,\ldots,\zeta_{r},\theta)\\
& \mapsto\left(t\mapsto \sum_{k=1}^r R_{k,1} \zeta_k(t) +S_1L(\zeta_1,\theta)(t)+P_1\theta (t)\right),\\
\fT_2\colon &\cC (\R_{\ge 0},\R^m)\to \cC (\R_{\ge 0},\R^{m-\ell }),~~ y\mapsto \big(t\mapsto S_2L(y)(t)\big)
\end{align*}
and set $d(\cdot):=e^{Q\cdot} x_3(0)$,  $d_1(\cdot ):=S_1d(\cdot)$ and $d_2(\cdot):=S_2 d(\cdot)$.
We may now identify \eqref{eq:NF-EABC} with the functional differential-algebraic system
\begin{equation}\label{eq:linDAE-op}
\begin{aligned}
  y_I^{(r)}(t) &= d_1(t) + \fT_1\big(y_I,\ldots,y_I^{(r-1)},y_{II}\big)(t) + \hat \Gamma u_I(t),\\
  0 &= \sum_{k=1}^r R_{k,2} y_I^{(k-1)}(t) \!+\! P_2y_{II}(t) \!+\! d_2(t) \!+\! \fT_2(y_I,y_{II})(t)\\
  &\quad+ \tilde \Gamma u_I(t) + u_{II}(t).
\end{aligned}
\end{equation}
Next, we extend this prototype to encompass nonlinear functional differential-algebraic equations (with memory quantified by $h \ge 0$) of the form
\begin{equation}\label{eq:DAE}
\begin{aligned}
 y_I^{(r)}(t) &= f_1\left(d_1(t), \fT_1(y_I,\ldots,y_I^{(r-1)},y_{II})(t), u_I(t)\right),\\
  0 &= f_2\big(y_I(t),\ldots,y_I^{(r-1)}(t)\big)+f_3\big(y_{II}(t)\big) \\
  &\quad + f_4\big(d_2(t), \fT_2(y_I,y_{II})(t)\big)\\
  &\quad +f_5(t) u_I(t)+f_6(t)u_{II}(t)
\end{aligned}
\end{equation}
with initial data
\begin{equation}\label{eq:nonlDAE-Ic}
\hspace*{-2ex}
\left.\begin{array}{lll}
y_I|_{[-h,0]}= y_I^0\in \cC^{r-1}([-h,0], \R^\ell ),
 \\
 y_{II}|_{[-h,0]}= y_{II}^0\in \cC([-h,0], \R^{m-\ell }),  & \text{if~$h >0$,}
\\[1ex]
\big(y_I(0),\cdots, y_I^{(r-1)}(0),y_{II}(0)\big)
\\
=\big(y^0_{I,1},\cdots,y^0_{I,r},y^0_{II}\big)\in\R^{m+(r-1)\ell},
&\text{if~$h=0$.}
\end{array}
\right\}
\end{equation}
We proceed to make precise the admissible operators and functions in the above extended formulation.  We first define a subclass of the
operator class of Definition \ref{Def:op-class}.
\begin{defn}[Operator class ${\mathbb T}^{n,q}_{h, \text{DAE}}$]\label{Def:operators}
For $h\ge 0$, $n,q\in\N$, the set ${\mathbb T}^{n,q}_{h, \text{DAE}}$ denotes the subclass of operators $\mathbf{\fT} :\cC([-h,\infty),\R^n)\to \cC^1(\R_{\ge 0},\R^q)$ such that $\fT\in {\mathbb T}^{n,q}_{h}$ and, in addition, there exist $g\in\cC(\R^n\times\R^q,\R^q)$ and $\tilde \fT\in {\mathbb T}^{n,q}_{h}$ such that
\[
    \forall\, \zeta\in\cC([-h,\infty),\R^n)\ \forall\,t\ge 0:\ \ddt (\fT\zeta)(t) = g\big(\zeta(t),\tilde\fT(\zeta)(t)\big).
\]
\end{defn}

We note that the additional assumption of the class~${\mathbb T}^{n,q}_{h, \text{DAE}}$ formulated above essentially requires that~$\mathbf{T} $ is the solution operator of a {functional} differential equation with input $\zeta$.

\begin{rem}
Recall that the operator $\fT_2$ in~\eqref{eq:linDAE-op} takes the form {$\fT_2\colon y\mapsto S_2 L(y)$.}
If $\sigma(Q)\subseteq{\mathbb{C}}_-$, then it is easy to see that $\fT_2\in {\mathbb T}^{m,q}_{0}$. Furthermore,
\[
    \ddt \mathbf{T}_2(y)(t) = S_2 A_{31} y(t)  + \tilde\fT (y)(t),
\]
where $\tilde\fT\colon y\mapsto S_2QL(y)$, and so $\fT_2\in{\mathbb T}^{m,m-q}_{0, \text{DAE}}$.
\end{rem}

Now, for $m,r\in\N$ and $\ell\in\{0,\ldots,m\}$, the representative nonlinear~DAE system~\eqref{eq:DAE} may be identified with the tuple
$(d_1,d_2,f_1,\ldots,f_6,\fT_1,\fT_2)$ on which we impose the following assumptions:  for some $\beta >0$ and $\ell,p\in\N$,
\begin{equation}\label{tuple}
\!\!\left.\begin{array}{l}
 d_1,d_2\in\cL^\infty(\R_{\ge 0},\R^{p}),~f_1\in\mathbf{N}^{p,q,\ell},\\
 f_2\in\cC^1(\R^{r\ell},\R^{m-\ell}), f_3\in\cC^1(\R^{m-\ell},\R^{m-\ell}),\\
 f_4\in\cC^1(\R^{p+q},\R^{m-\ell}), f_5\in(\cC^1\cap \cL^\infty)(\R_{\ge 0},\R^{(m-\ell)\times \ell}),\\
 f_6\in{(\cC^1\cap\cL^\infty)(\R_{\ge 0},\R)},\\
 \forall\,t\ge 0:\ |f_6(t)|\ge \beta,~\fT_1\in {\mathbb T}^{(r-1)\ell+m,q}_{h},\  \fT_2
 \in{\mathbb T}^{m,q}_{h, \text{DAE}}
 \end{array}\!\!\right\}
 \end{equation}
 where $\mathbf{N}^{p,q,\ell}$ is as in Definition~\ref{adnon}.  Thus, we are led to consideration of the following nonlinear functional differential-algebraic system class,
 parametrized by $m,r\in\N$ and $\ell\in\{0,\ldots,m\}$:
 \[
 \mathcal{ND}^{m,r,\ell}:=\setdef{\!\!\!\begin{array}{l} (d_1,d_2,f_1,\\
 \ldots,f_6,\fT_1,\fT_2)\end{array}\!\!\!}{\!\!\!\begin{array}{l}\eqref{tuple}~\text{holds for some},\\
 h\ge 0, \beta >0,\,q,p\in\N\end{array}\!\!\!\!}.
 \]
Recalling the equivalent representations \eqref{eq:NF-EABC} and \eqref{eq:linDAE-op} of any linear system $[E,A,B,C]\in
\mathcal{LD}^{m,r,\ell}$, we have the inclusion
$\mathcal{LD}^{m,r,q} \subset \mathcal{ND}^{m,r,\ell}$.
We also remark that, if~$\ell=m$, then~$y_{II}$ and the second of relations~\eqref{eq:DAE} are vacuous, in which case~\eqref{eq:nonlSys}
and~\eqref{eq:DAE}
are equivalent and so
$\mathcal{N}^{m,r}\equiv \mathcal{ND}^{m,r,m}$.

\subsection{Systems described by partial differential equations}\label{Sec:inf-dim}
\noindent
Early intimations on funnel control for infinite-dimensional systems modelled by partial differential equations (PDEs) may be found in \textit{Ilchmann, Ryan, and Sangwin} (2002)~\cite{IlchRyan02b}. However, in a general infinite-dimensional context, many open questions and challenges remain.
We briefly describe some recent findings in the following three subsections, which we preface with some basic facts pertaining to linear
infinite-dimensional systems in the abstract form
\begin{equation}\label{eq:PDE-wp}
  \dot z(t) = A z(t) + B \zeta(t),\ \ z(0)=z^0\in \cD (A),\ \
  \eta(t) = C z(t),
\end{equation}
where  $A$ is the generator of a strongly continuous semigroup of bounded linear operators on a real Hilbert space~$H$.  In what follows, for brevity,
technicalities are suppressed:  the reader is referred to the succinctly-written treatise~\cite{TucsWeis09} for full details; the survey article~\cite{TucsWeis14} is likewise
recommended.
Recall that a {\it semigroup}  $(T(t))_{t\ge0}$ on $H$ is a parameterized family of operators in  $\mathfrak{L}(H,H)$ satisfying $T(0)=I$ and $T(t+s)=T(t)T(s)$,
for all $s,t\geq0$, where $I$ denotes the identity operator. The semigroup is said to be strongly continuous if, for all $z\in H$,  $\|T(t)z-z\|\to 0$ as $t\to 0$.
The {\it growth bound} of the semigroup is
\[
\omega_T:= \inf \setdef{\omega\in\R}{\sup_{t\ge 0} \|e^{-\omega t} T(t)\| < \infty}
\]
and, for any $\omega >\omega_T$,  there exists a constant $c_\omega$ such that
\[
 \forall\, t\ge0:\ \|T(t)\|\leq c_\omega e^{\omega t}.
\]
The semigroup is {\it exponentially stable}, if $\omega_T<0$.

We assume that the (densely defined) operator $A$ has non-empty resolvent set $\varrho (A)$.   Introduce the (Hilbert) spaces $H_1$ and $H_{-1}$, where
$H_1=\cD (A)$ equipped with the graph norm and $H_{-1}$ is the completion of $H$ with respect to the norm given by
$\|z\|_{{-1}}=\|(\beta I-A)^{-1} z\|$, where $\beta$ is any element of $\varrho(A)$.  Then $H_1\subset H\subset H_{-1}$ with dense and continuous injections.
As a map $H_1\to H$, $A$ is bounded, that is, $A\in\fL (H_1,H)$, and has a unique extension $A_{-1}\in\fL (H,H_{-1})$.  Furthermore, the semigroup~$(T(t))_{t\ge 0}$
on~$H$ extends uniquely to a semigroup~$(T_{-1}(t))_{t\ge 0}$ with generator~$A_{-1}$.

We are now in a position to formulate assumptions on the
triple $(A,B,C)$, specifically tailored to our context of funnel control. First, we assume that~$\zeta$ and~$\eta$ are, respectively,
$\R^{\ell}$-valued and $\R^q$-valued functions.
Secondly, we assume that $(A,B,C)$ is a {\it regular well-posed} system, that is:
\begin{enumerate}[(i)]
  \item $A$ is the generator of a strongly continuous semigroup $(T(t))_{t\ge 0}$.
  \item $B$ is an {\it admissible control operator} (in the terminology coined by \textit{Curtain and Weiss}~\cite{CurtWeis89}); that is,
 $B\in\fL (\R^{\ell},H_{-1})$ and
\begin{multline*}
\Phi_{t}\colon~\zeta\mapsto \int_{0}^{t}{T}_{-1}(t-\tau)\, B\zeta (\tau)\, \text{d} \tau\\
\text{is in}~ {\mathfrak L}\big(\cL^2([0,t],\R^{\ell}),H\big)~\text{for all}~t\ge 0.
\end{multline*}
\item $C$ is an {\it admissible observation operator}; that is, $C\in\fL(H_1,\R^q)$ and, for all $t\ge 0$,
\[
   \Psi_t\colon~z \mapsto C T(\cdot)z \quad\text{is in}~ {\mathfrak L}\big((H_1,\|\cdot\|_H),\cL^2([0,t],\R^q)\big).
\]
\item
For some $\omega\in\R$,  there
exists an analytic function $\mathbf{G}:{\mathbb{C}}_{>\omega}\to \R^{q\times \ell}$ (referred to as a {\it transfer function}) which satisfies
\begin{equation}\label{eq:transfer1}
\forall\, s\in{\mathbb{C}}_{>\omega}:\ \mathbf{G}^\prime(s) = -C(sI-A)^{-2}B
\end{equation}
and $\lim_{\Re s\to\infty}\mathbf{G}(s)$ exists.
\end{enumerate}
The subtlety of assumption~(ii) is that $\Phi_t$ generates a $H$-valued function, even though the function $B\zeta(\cdot)$ takes its values in an the larger space $H_{-1}$; loosely speaking, the ``smoothing'' effect of the semigroup saves the day.
For $\zeta\in \cL^2_{\text{\rm loc}}(\R_{\ge 0},\R^\ell)$, the {\it mild solution} of the initialised differential
equation in~\eqref{eq:PDE-wp} is given by
\[
z(t)=T(t)z^0+ \Phi_t(\zeta|_{[0,t]}),\quad t\ge 0.
\]

\subsubsection{Infinite dimensional internal dynamics}\label{Ssec:InfIntDyn}
%
\noindent
Consider again system~\eqref{eq:PDE-wp} and assume that~$(A,B,C)$ is regular well-posed.
With this system, for every~$z_0\in H_1$  we may associate a map
\begin{multline*}
{\mathbf T}\colon\cC (\R_{\ge 0},\R^\ell)\to \cL^\infty_{\mathrm{loc}}(\R_{\ge 0},\R^q),\\
\zeta\mapsto \eta =\left(t\mapsto \big(C T(t) z^0 + C\Phi_t(\zeta|_{[0,t]})\big)\right)
\end{multline*}
for which, as shown in~\cite{BergPuch20a}, properties (TP1) and (TP2) of Definition \ref{Def:op-class} hold.  If, in addition, $(A,B,C)$ is bounded-input bounded-output stable, i.e., the inverse Laplace transform of each of the components of the transfer function $\mathbf{G}$ is a real-valued measure with bounded total variation,
 then property~(TP3) also holds and so ${\mathbf T}\in\mathbb T_0^{\ell,q}$; note that exponential stability of the semigroup $(T(t))_{t\ge 0}$ is sufficient for this property to hold.
  If $f\in {\mathbf N}^{p,q,m}$ (recall Definition~\ref{adnon}), $d\in \cL^\infty(\R_{\ge 0},\R^p)$ and setting $\ell=rm$, we may conclude that $(d,f,\fT)\in \cN^{m,r}$.
Note that the
class of operators~${\mathbf T}$ considered in~\cite{BergPuch20a} is considerably larger and also allows for certain nonlinear output operators associated with the differential equation in~\eqref{eq:PDE-wp}.

The application of funnel control to a particular member of the above described system class was considered in~\cite{BergPuch22}: the control of the horizontal movement of a water tank. The problem is modelled via the linearized Saint-Venant equations and subject to sloshing effects. It is shown that the overall system belongs to the above system class and hence tracking with prescribed transient behaviour can be achieved. We will return to this example in Section~\ref{Ssec:Appl-PDE}.

\subsubsection{Linear infinite-dimensional systems with integer-valued relative degree}\label{Ssec:IST}
\noindent
The following class of single-input, single-output, linear, infinite-dimensional systems $(A,b,c)$, coming from partial differential equations and of the general form~\eqref{eq:PDE-wp},
were considered by
\textit{Ilchmann, Selig, and Trunk} (2016) \cite{IlchSeli16}:
\begin{equation}\label{eq:Abc}
\dot{x}(t) =Ax(t)+ bu(t),\ \ x(0)=x^0\in\cD (A),\ \
 y(t) = \langle x(t),c\rangle \,,
\end{equation}
where
\begin{enumerate}[\hspace{2pt}\textbf{(A1)}]
\item[\textbf{(A1)}] \label{item:A}
$A \colon \cD(A)  \to H $ is  the generator of a strongly-continuous semigroup~$(T(t))_{t\geq0}$  of bounded linear operators on a real Hilbert space~$H$
with inner product $\Skdef$,
\end{enumerate}
and $b,c\in H$ with, for some  $r\in\N$,
\begin{enumerate}[\hspace{2pt}\textbf{(A1)}]
\item[\textbf{(A2)}] \label{item:bc}
 $b \in \cD(A^r)$ and $c \in \cD\big((A^*)^r\big)$,
\item[\textbf{(A3)}] \label{item:reldeg}
      $\langle A^{r-1}b,c\rangle \ne 0 \  \text{and}  \  \langle A^jb,c\rangle=0$ for all $j=0,1,\ldots,r-2$.
\end{enumerate}
For finite-dimensional systems (in which case, $H\simeq\R^n$ for some $n\in\N$) assumptions~{{(A1)}} and~{{(A2)}} are superfluous,
and assumption~{{(A3)}} is the  relative degree~$r$ property from Definition~\ref{Def:rel-deg}.
For infinite-dimensional systems, assumption~{{(A1)}} is ubiquitous in systems theory, see e.g.~\cite{CurtZwar95} and has already been discussed above;
assumption~{{(A2)}} is very restrictive from a practical point of view  (for example, if~$\Omega$ is the spatial domain of an underlying PDE, then control/observation
on the domain boundary and pointwise
control/observation concentrated at points in the interior of $\Omega$ are both excluded).  For $\omega > \omega_T$ (the growth bound of the semigroup),
the function $s\mapsto {\mathbf G}(s):=\langle c,(sI-A)^{-1}b\rangle$ is a
 transfer function on~${\mathbb C}_\omega$
 (recall that it is unique up to a constant).
Assumptions~(A2) and~(A3) imply, by~\cite[Lem.~2.9]{MorrReba07}, that
that the transfer function of the system satisfies
\begin{equation}\label{eq:rel_deg_freq}
\hspace*{-1ex}
\lim_{s\to\infty,\ s\in \R} s^r {\mathbf G}(s)\neq 0
\ \ \text{and}\
\lim_{s\to\infty,\ s\in \R} s^{r-1}{\mathbf G}(s)=0.
\end{equation}
It follows {\it a fortiori} that, under assumptions {\bf{(A1)}}-{\bf{(A3)}}, system~$(A,b,c)$ is regular well-posed.   In \cite{IlchSeli16}, it is shown that the class of such systems
allows for a Byrnes-Isidori form similar to that discussed in Section~\ref{Ssec:BIF} for finite-dimensional systems. The only difference is that the internal dynamics are described by a subsystem of the form~\eqref{eq:QPS}, where~$Q$ is the generator of a strongly continuous semigroup in a Hilbert space~$H_Q$ and $S:H_Q\to \R$, $P:\R\to H_Q$ are bounded linear operators. In particular, systems~\eqref{eq:QPS} with these properties are subclasses of the regular well-posed  infinite-dimensional systems~\eqref{eq:PDE-wp} as discussed above.
Therefore, assuming that $Q$ generates an exponentially stable and strongly continuous semigroup, the comments in Section~\ref{Ssec:InfIntDyn} apply to
conclude that~\eqref{eq:Abc} belongs to the class $\cN^{1,r}$.

In particular, systems~\eqref{eq:Abc} cover the heat equation with Neumann boundary conditions
modelled by
\begin{equation}\label{eq:heat_eq}
\hspace*{-2ex}
\begin{array}{l}
\partial_t x(\xi,t)=\partial_\xi^2 x(\xi,t)+u(t),~~(\xi,t)\in[0,1]\times\R_{>0},\\
x(\xi,0) = x^0(\xi),~~\xi\in [0,1],\\
\partial_{\xi} x(0,t) =0= \partial_{\xi} x(1,t),\\
  y(t) = \int_0^1 \cos^2(\pi\xi) x(\xi,t)\,\,{\rm d}\xi,~~ t> 0,
\end{array}
\end{equation}
where $x(\xi,t)$ represents the temperature at position~$\xi$ and time~$t$.  The initial
temperature profile is~$x^0(\cdot )$,
 and~$u(t)$ denotes  the heat input at time $t$.  Setting $H=\cL^2([0,1],\R)$, defining $b,c\in H$ by $b(\xi)=1$, $c(\xi)= \cos^2 \xi$, and with
 \begin{multline*}
 A\colon\cD(A)\to H,~f\mapsto f^{\prime\prime}\\
 \quad\text{with}~\cD(A):=\{f\in \cW^{1,2}([0,1],\R)|~ f^\prime(0)=0=f^\prime(1)\},
 \end{multline*}
 this example can be written as~\eqref{eq:Abc} satisfying~{{(A1)}}--{{(A3)}}. 

As already mentioned, a  limitation of the above system classes is that boundedness of the control and observation operators in~\eqref{eq:Abc} is assumed and hence
boundary control action is excluded. Moreover, if one introduces Dirichlet boundary conditions
instead of Neumann conditions,
then neither does it  satisfy~(A1)--(A3),
nor does it have a relative degree, nor does the Byrnes-Isidori form exist.

\subsubsection{Infinite-dimensional systems without well-defined relative degree}\label{Ssec:Inf-hard}
%
\noindent
Whilst the discussion in the previous subsection seems quite general, not
every linear, infinite-dimensional system has a well-defined (integer-valued) relative degree: in which case, results as in~\cite{BergLe18,BergPuch20a,IlchRyan02b,IlchSeli16}
are inapplicable. Instead, feasibility of funnel control has to be investigated on an {\it ad hoc} basis; thus, in this subsection, we directly refer to the respective results, although funnel control for the other system classes is discussed in Section~\ref{Sec:FC}. As an initial contribution in this regard, \textit{Reis and Selig} (2015)~\cite{ReisSeli15b}
considered a  boundary controlled heat equation  with Neumann boundary control and a~Dirichlet-like boundary observation,
\begin{equation}\label{eq:heateq}
\left.
\begin{array}{l}
\partial_t x(\xi,t)= \Delta_\xi x(\xi,t),~~(\xi,t)\in\Omega\times\R_{> 0},\\
u(t)=\partial_\nu x(\xi,t),~~(\xi,t)\in\partial\Omega\times\R_{> 0},\\
x(\xi,0)=x^0(\xi),~~\xi\in\Omega,\\
y(t)=\int_{\partial\Omega} x(\xi,t)\,{\rm d}\sigma_\xi,~~(\xi,t)\in\partial\Omega\times\R_{>0},
\end{array}
\right\}
\end{equation}
where   $\Omega\subseteq\R^d$ denotes a~bounded domain with uniformly $\mathcal{C}^2$-boundary~$\partial\Omega$. This example is considerably different from the finite dimensional case and
from~\eqref{eq:heat_eq}.
Although it can be formulated as an infinite-dimensional linear system of the form~\eqref{eq:Abc},
the operators~$b$ and~$c$ are now unbounded;
$b$ maps to the space $\cD(A^*)' \supseteq H = \cL^2(\Omega,\R)$
and~$c$ is defined on a proper subset of~$H$.
Therefore, a Byrnes-Isidori form cannot be expected,
and the product~$cb$, which indicates the relative degree, does not exist.

Nevertheless, feasibility of funnel control is shown in~\cite[Thm.\,4.2]{ReisSeli15b}.
The proof is based on modal approximation of the input-output map by finite-dimensional linear systems with asymptotically stable zero dynamics and relative degree one.
It is shown that funnel control is feasible for these truncated systems and that the sequence of solutions to the
closed-loop truncated systems contains a~convergent subsequence.
The limit of this subsequence will solve a~nonlinear Volterra equation
that represents the input-output behaviour of the heat equation system~\eqref{eq:heateq}
 under funnel control~\eqref{eq:FC}. This solution results in a~well-defined input signal $u\in \cL^2_{\rm loc}(\R_{>0},\R)$. Inserting this signal into the heat equation~\eqref{eq:heateq} yields a~solution to the funnel controlled heat equation in the sense of well-posed linear systems.
 It is then shown that this solution~$x$ solves the partial differential equation formed by~\eqref{eq:heateq},~\eqref{eq:FC} in a~stronger sense and that it has additional regularity and boundedness properties.

Essentially, it is also possible to reformulate~\eqref{eq:heateq} as a regular well-posed system of the form~\eqref{eq:PDE-wp} with the help of Section~5.2 in \textit{Staffans}~\cite{Staf05}. However, this would require a high level of technicalities and it is easier to analyze the system in the boundary control formulation~\eqref{eq:heateq}.
As an extension of those results,
 \textit{Puche, Reis and Schwenninger}~(2021)~\cite{PuchReis21}
consider a fairly general class of boundary control systems of the form
\begin{equation}\label{eq:ABC-infdim}
\begin{aligned}
	\dot{x}(t)&=\mathfrak{A}x(t),\ \ x(0)=x^0,\\
	u(t)&=\mathfrak{B} x(t),\ \
	y(t)=\mathfrak{C} x(t),
\end{aligned}
\end{equation}
where~$\mathfrak{A}$, $\mathfrak{B}$, $\mathfrak{C}$ are linear operators
and the~$\R^m$-valued  functions~$u$
and~$y$ are interpreted as the input and the measured output~$y$, resp., whereas~$x$ is called the state of the system. Typically, $\mathfrak{A}$ is a differential operator on a Hilbert
space~$H$ and~$\mathfrak{B}$, $\mathfrak{C}$ are boundary control and observation operators, resp.
The system class is specified by the following assumptions:

\begin{enumerate}[(i)]
\item\label{ass:1}
The system is \emph{(generalized) impedance passive}, that is,
		\begin{multline*}\label{eq:diss}
		\exists\,\alpha\in\R\ \forall\,x\in\cD(\mathfrak{A}) \colon\\
\text{Re}\, \langle \mathfrak{A} x, x\rangle_H\leq
		\text{Re}\, (\mathfrak{B} x)^\top (\mathfrak{C} x) +\alpha\|x\|_H^2.
		\end{multline*}
\item\label{ass:3}
 {There exists $\beta\ge\alpha$, such that the operator $\mathfrak{A}|_{\ker\mathfrak{C}}$ (i.e., the restriction of $\mathfrak{A}$ to $\ker\mathfrak{C}\subset{\mathcal D}(\mathfrak{A})$) satisfies $\ran (\mathfrak{A}|_{\ker\mathfrak{C}}-\beta I)=H$.}
\item\label{ass:CBonto} The operator
$		\begin{bmatrix}
		\mathfrak{B}\\
		\mathfrak{C}
		\end{bmatrix}:\cD(\mathfrak{A})\to {\R^m\times \R^m}\label{eq:CBop}
$
		is onto.
\end{enumerate}

Under these assumptions, the zero dynamics of system~\eqref{eq:ABC-infdim} are described by a strongly continuous semigroup, which is generated by the restriction of~$\mathfrak{A}$ to the kernel of~$\mathfrak{C}$.
Furthermore, it follows from the
Lumer–Phillips-Theorem that the semigroup is
exponentially stable, if~$\alpha <0$. This property resembles that of asymptotic stability of the zero dynamics in the finite dimensional case.

Feasibility of funnel control can be shown for the class~\eqref{eq:ABC-infdim}, under assumptions~\eqref{ass:1}--\eqref{ass:CBonto} with $\alpha<0$,
by invoking  $m$-dissipative operators and a ``clever'' change of
coordinates. This class encompasses  hyperbolic boundary control systems  in one spatial variable
 (e.g., the lossy transmission line),
hyperbolic systems in several spatial variables
(e.g., the wave equation in two spatial dimensions),
and parabolic systems with Neumann boundary control (e.g., the heat equation). Further classes of boundary controlled port-Hamiltonian systems are discussed in the recent works~\cite{PhilReis23,ReisScha23}, which are amenable to funnel control in the case of co-located input-output structures (i.e., actuators and sensors are placed at the same position) and finite dimensional input and output spaces~-- but this has not been proved yet. Specific examples which belong to this class are Maxwell's equations, Oseen's equations (linearized incompressible flow), and advection-diffusion equations.

Furthermore, in the context of infinite-dimensional systems which do not have a well-defined relative degree, feasibility of funnel control has also been investigated for the
FitzHugh-Nagumo monodomain model (which represents defibrillation processes of the human heart)~\cite{BergBrei21} and the Fokker-Planck equation for a multidimensional Ornstein-Uhlenbeck process~\cite{Berg21}.

\subsection{An overview of the system classes and their relations}
\noindent
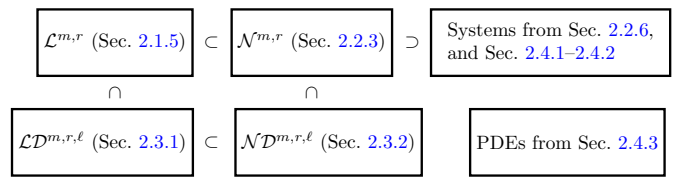
\begin{figure}[h!]
  \centering
\resizebox{\columnwidth}{!}{
  \begin{tikzpicture}[very thick,node distance = 12ex, box/.style={fill=white,rectangle, draw=black}, blackdot/.style={inner sep = 0, minimum size=3pt,shape=circle,fill,draw=black},plus/.style={fill=white,circle,inner sep = 0, minimum size=5pt,thick,draw},metabox/.style={inner sep = 3ex,rectangle,draw,dotted,fill=gray!20!white}]

    \node (box0)		[box,left of=box1,minimum size=8ex,xshift=-11.5ex]{$\cL^{m,r}$\ (Sec.~\ref{Sec:Lmr})};
  \node (box1)		[box,minimum size=8ex]{$\cN^{m,r}$\ (Sec.~\ref{Sec:Nmr})};
  \node [left of=box1,xshift=0ex] {$\subset$};
  \node [below of=box1,yshift=6ex] {\protect\rotatebox[origin=c]{-90}{$\subset$}};
   \node [below of=box0,yshift=6ex] {\protect\rotatebox[origin=c]{-90}{$\subset$}};
  \node (box2)		[box,minimum size=8ex,below of=box1,yshift=0ex,xshift=2ex]{$\cN\cD^{m,r,\ell}$\ (Sec.~\ref{Sssec:DAEs-nonl})};
   \node [left of=box2,xshift=-2ex] {$\subset$};
   \node (box3)		[box,left of=box2,minimum size=8ex,xshift=-15ex]{$\cL\cD^{m,r,\ell}$\ (Sec.~\ref{Sssec:DAEs-lin})};
    \node [right of=box1,xshift=0ex] {$\supset$};
    \node (box4)		[box,right of=box1,minimum size=8ex,xshift=16.5ex]{$\begin{array}{l} \text{Systems from Sec.~\ref{Sssec:special-sc},} \\  \text{and Sec.~\ref{Ssec:InfIntDyn}--\ref{Ssec:IST}} \end{array}$};
    \node (box5)		[box,minimum size=8ex,below of=box4,yshift=0ex,xshift=2ex]{PDEs from Sec.~\ref{Ssec:Inf-hard}};
  \end{tikzpicture}
}
\caption{The system classes and how they are related.}
\label{Fig:sys-classes}
\end{figure}

\section{Funnel control}\label{Sec:FC}
\noindent
Before discussing the different variants of funnel controllers (see also Fig.~\ref{Fig:FC-variants}), we like to emphasize that essentially there is only \textit{one} funnel controller, its structure depending on the relative degree~$r$ of the considered system class. Those system classes, as discussed in the previous section, can also contain differential-algebraic systems or PDE systems or both (as e.g.\ the class $\cN\cD^{m,r,\ell}$). The specific shape of the controller can be adjusted by the choice of certain controller parameters. For DAE systems, the funnel controller is combined with a relative-degree-one controller component for the algebraic part. Once the controller is fixed, it works for \textit{every} member of the considered system class; in particular, a fixed controller may work for finite-dimensional ODE systems \textit{and} for infinite-dimensional PDE systems, without further adjustments.

\begin{figure*}[h!]
  \centering
\resizebox{0.8\textwidth}{!}{
  \begin{tikzpicture}[very thick,node distance = 12ex, box/.style={fill=white,rectangle, draw=black}, blackdot/.style={inner sep = 0, minimum size=3pt,shape=circle,fill,draw=black},plus/.style={fill=white,circle,inner sep = 0, minimum size=5pt,thick,draw},metabox/.style={inner sep = 3ex,rectangle,draw,dotted,fill=gray!20!white}]

    \node (box0)		[box,left of=box1,minimum size=8ex,xshift=-35ex]{$\begin{array}{l} \text{Relative degree one (Sec.~\ref{Sec:reldeg1}):}\\[2mm] u(t)=(N\circ\alpha)\big(\|w(t)\|^2\big) w(t),\\ w(t) = \varphi(t) e(t)\\[2mm] \text{(works for systems from $\cN^{m,1}$)}\end{array}$};
  \node (box1)		[box,minimum size=8ex,xshift=2ex]{$\begin{array}{l} \text{Relative degree~$r$ (Sec.~\ref{Ssec:dervi-fb}):}\\[2mm] u(t)=(N\circ\alpha)\big(\|w(t)\|^2\big) w(t),\\ w(t) = \rho_r\big(\varphi(t) {\mathbf e(t)}\big)\\[2mm] \text{(works for systems from $\cN^{m,r}$)}\end{array}$};
  \node (box2)		[box,below of=box1,minimum size=8ex,yshift=-20ex]{$\begin{array}{l} \text{Controller for DAEs (Secs.~\ref{Sec:reldeg1}--\ref{Ssec:dervi-fb}):}\\[2mm] u(t)=\begin{pmatrix} u_I(t)\\ \hat k\, u_{II}(t)\end{pmatrix}\\[4mm] \text{(works for systems from $\cN\cD^{m,r,\ell}$)}\end{array}$};

  \node (plus) [circle,inner sep = 0,very thick,draw,above of = box1,yshift=5ex,xshift=-17.9ex] {$+$};
  \node (box3)		[box,left of=plus,minimum size=8ex,xshift=-15ex]{Funnel pre-compensator (Sec.~\ref{Ssec:precomp})};
  \node (box4)		[box,right of=plus,minimum size=8ex,xshift=10ex]{$\begin{array}{l} \text{Funnel control with}\\ \text{non-derivative feedback}\\[1mm] \text{(works for the system}\\ \text{class from Sec.~\ref{Sssec:special-sc})}\end{array}$};

  \node (box5)		[box,below of=box0,minimum size=8ex,yshift=-10ex]{$\begin{array}{l} \text{Saturated funnel}\\ \text{controller (Sec.~\ref{Ssec:FC-sat})}\end{array}$};
  \node (box6)		[box,below of=box5,minimum size=8ex,yshift=-2ex,xshift=2ex]{$\begin{array}{l} \dot \psi(t) = -\alpha \psi(t) + \beta + \psi(t) \frac{\kappa(v(t))}{\|e(t)\|},\\ \kappa(v(t)) = \|v(t)-\satu(v(t))\|\end{array}$};
  \node (box7)		[box,below of=box6,minimum size=8ex,yshift=-5ex]{$\begin{array}{l} \text{Input-constrained funnel controller (Sec.~\ref{Ssec:ICFC})}\\ \text{(works for modified system classes)}\end{array}$};

  \draw[->] (box0)--(box1) node[midway,above]{extension};
  \draw[->] (box1)--(box2) node[midway,right]{$\begin{array}{l} e=e_I\\ \varphi = \varphi_I\\ u = u_I\end{array}$};
  \draw[->] (box0.south east)--(box2.north west) node[pos=0.5,right]{\hspace*{-2mm}$\begin{array}{l} e=e_{II}\\ \quad  \varphi = \varphi_{II}\\ \qquad u = u_{II}\end{array}$};
  \draw[->] (box1.north west)--(plus);
  \draw[->] (box3)--(plus);
  \draw[->] (plus)--(box4);
  \draw[->] (box0)--(box5) node[midway,right]{$\satu(u)$};
  \draw[bend right,->]  (box0.west) to node [pos=0.8,left] {$v=\frac{u}{\varphi}$} (box6.west);
  \draw[bend right,<-]  (box0.south west) to node [right] {$\varphi=\frac{1}{\psi}$} (box6.north west);
  \draw[->] (box6)--(box7) node[midway,right]{$\satu(u)$};
  \end{tikzpicture}
}
\caption{Variants of the funnel controller.}
\label{Fig:FC-variants}
\end{figure*}
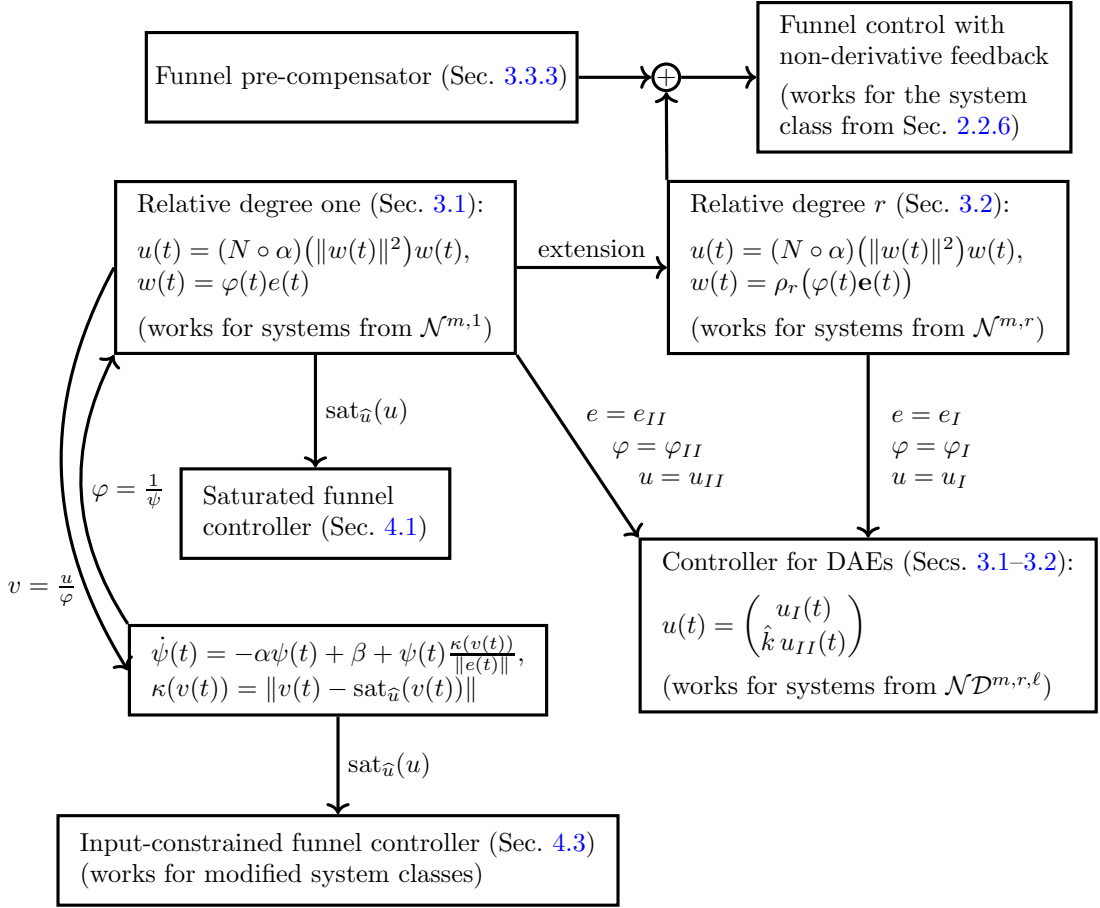

\subsection{The relative-degree-one case}\label{Sec:reldeg1}
\subsubsection{Systems of class $\cN^{m,1}$}
\noindent
Here, as an expository precursor to a result for systems of arbitrary (but known) relative degree (found in Section~\ref{Ssec:dervi-fb}), we focus attention on relative-degree-one of class~$\cN^{m,1}$, described by
systems
\begin{equation}\label{eq:Rd1}
\hspace*{-2ex}
\begin{aligned}
&\dot y(t)= f\big(d(t), \fT(y)(t), u(t)\big),\\
&\text{with}~~\left\{\begin{array}{ll}
y|_{[-h,0]}= y^0\in \cC([-h,0], \R^m), & \text{if $h >0$,}
\\[1ex]
y(0)=y^0\in\R^{m}, &\text{if $h=0$,}
\end{array}
\right.
\end{aligned}
\end{equation}
and $(d,f,\fT)\in\cN^{m,1}$.   Choose (as control design parameters) $\varphi\in\Phi$, a surjection $N\in \cC(\R_{\ge 0},\R)$, and a bijection
$\alpha\in \cC^1 ([0,1),[1,\infty))$.  For example, $N\colon s\mapsto s\sin s$ and $\alpha\colon s\mapsto 1/(1-s)$ suffice. {Let $y_{\rm ref}\in \cW^{1,\infty}(\R_{\ge 0},\R^m)$.
The funnel control is given (formally) as
\begin{equation}\label{FcRd1}
\begin{aligned}
u(t)&=(N\circ\alpha)\big(\|w(t)\|^2\big) w(t),\\ w(t) &= \varphi(t)(y(t)-y_{\rm ref}(t)).
\end{aligned}
\end{equation}
We like to note that the above controller is a more general version of~\eqref{eq:FC}, which is recovered by the choice $\alpha(s) = 1/(1-s)$ and $N(s) = -s$. Although the latter is not a surjection from $\R_{\ge 0}$ to $\R$, its choice is valid in cases discussed in Remark~\ref{Rem:N(s)=-s}.

\begin{thm}\label{Thm:FunCon-Nonl-Rd1}
Consider system~\eqref{eq:Rd1} with $(d,f,\fT)\in \cN^{m,1}$, $m\in\N$. Choose $\varphi\in\Phi$,
a surjection $N\in \cC(\R_{\ge 0},\R)$, and a bijection $\alpha\in \cC^1 ([0,1),[1,\infty))$.
Let $y_{\mathrm{ref}}\in \cW^{1,\infty}(\R_{\ge 0},\R^m)$ be arbitrary and assume that
\begin{equation}\label{ic}
{\varphi(0) \|y(0)-y_{\mathrm{ref}}(0)\| < 1.}
\end{equation}
Then the funnel control~\eqref{FcRd1} applied to~\eqref{eq:Rd1} yields an initial-value problem which has a solution {(in the sense of Carath\'{e}odory)}, every solution can be maximally extended and every maximal solution
$y:\left[-h,\omega\right)\rightarrow \R^m$ has the properties:
\begin{enumerate}[(i)]
\item \label{item-main-1}
$\omega=\infty$ (global existence);
\item \label{item-main-2}
$u\in\cL^\infty (\R_{\ge0},\R^m)$, $y\in\cW^{1,\infty}([-h,\infty),\R^m)$;
\item \label{item-main-3}
the tracking error $e\colon \R_{\ge 0}\to \R^m, ~t\mapsto y(t)-y_{\mathrm{ref}}(t)$ evolves strictly inside the funnel~$\mathcal{F}_{\varphi}$
in the sense that there exists $\varepsilon \in (0,1)$ such that
$\varphi(t)\|e(t)\|\leq \varepsilon$ for all $t\geq 0$.
\end{enumerate}
\end{thm}
This result is a special case of a more general result in Theorem \ref{Thm:FunCon-Nonl} below.

\subsubsection{{Systems of class $\mathcal{ND}^{m,1,\ell}$}}\label{Ssec:DAE-FC}
\noindent
Funnel control has been shown for a couple of subclasses of systems~\eqref{eq:EABC} in~\cite{Berg14a}, see also~\cite{BergIlch12b,BergIlch12c,Berg16b}. The general case has been considered recently in~\cite{BergLe20}. Although a slightly different approach (with a stronger assumption on~$f_1$) has been considered in~\cite{BergLe20}, in view of~\cite{BergIlch21} it is straightforward to extend the results to the following framework.

Again, choose $\varphi_I\in\Phi$, a surjection $N\in \cC(\R_{\ge 0},\R)$, and a bijection
$\alpha\in \cC^1 ([0,1),[1,\infty))$.  Let $y_{\rm ref}\in \cW^{1,\infty}(\R_{\ge 0},\R^m)$ with $y_{\rm ref} =(y_{{\rm ref},I},y_{{\rm ref},II})$, where $y_{{\rm ref},I} = (y_{{\rm ref},1},\ldots,y_{{\rm ref},\ell})$ and $y_{{\rm ref},II} = (y_{{\rm ref},\ell+1},\ldots,y_{{\rm ref},m})$. The first component of the funnel control is given (formally) as
\begin{equation}\label{FcDAERd1-1}
\begin{aligned}
u_I(t)&=(N\circ\alpha)\big(\|w(t)\|^2\big) w(t),\\
w(t) &= \varphi_I(t)(y_I(t)-y_{{\rm ref},I}(t)).
 \end{aligned}
\end{equation}
Next, we define the second control component~$u_{II}$. Since~$\Gamma$ as in~\eqref{eq:Gamma} plays the role of the \textit{inverse} of the high-frequency gain matrix, cf.~\cite[Rem.~5.3.9\,(iv)]{Berg14a}, but is not assumed invertible, the non-invertible part induces algebraic constraints in the control law. In order to guarantee feasibility of funnel control, these constraints need to be resolved, which is possible when the initial gain is chosen large enough, see also~\cite[Rem.~5.2.1]{Berg14a}. Choosing $\varphi_{II}\in\Phi\cap \cW^{1,\infty}(\R_{\ge 0},\R)$, this leads to a modification of the funnel controller~\eqref{FcRd1} of the form
\begin{equation}\label{FcDAERd1-2}
\begin{aligned}
  u_{II}(t) &= -\hat k\, \alpha\big(\|v(t)\|^2\big) v(t),\\
  v(t) &= \varphi_{II}(t)(y_{II}(t)-y_{{\rm ref},II}(t)),
 \end{aligned}
\end{equation}
where the initial gain $\hat k>0$ {is required} to satisfy
\begin{equation}\label{eq:init-gain}
{\hat k > \frac{1}{\beta}\,\text{ess\,sup}_{t\ge 0} \|f_3(t)\|,}
\end{equation}
{$\beta$ being} the lower bound for $|f_{6}|$ from the definition of~$\mathcal{ND}^{m,1,\ell}$.  In the case of
systems in the subclass $\mathcal{LD}^{m,1,\ell}$, with representative \eqref{eq:linDAE-op}, the latter condition reduces to~$\hat k > \|P_2\|$.

\begin{rem}\label{Rem:cons-init-val}
Since the second equation in~\eqref{eq:DAE} is an algebraic equation we need to guarantee that it is initially satisfied for a solution to exist. In essence, this is the issue of {\it consistency} or {\it well-posedness} of the closed-loop system.
Since $\fT_2\in{\mathbb T}^{m,q}_{h, \text{DAE}}$ is causal it ``localizes'', in a natural way, to an operator $\hat \fT_2: \cC([-h,\omega]\to\R^m)\to \cC^1([0,\omega]\to\R^q)$, cf.~\cite[Rem.~2.2]{IlchRyan09}. With some abuse of notation, we will henceforth not distinguish between~$\fT_2$ and its ``localization'' $\hat \fT_2$.
Then, in the case of relative degree $r=1$, an initial condition $y^0=(y_I^0,y_{II}^0)$ as in~\eqref{eq:nonlDAE-Ic} (for~$h>0$) is called {\it consistent} for the closed-loop system~\eqref{eq:DAE},~\eqref{FcDAERd1-1},~\eqref{FcDAERd1-2}, if
\begin{multline}\label{eq:consistent-IV}
f_2\left(y_I^0(0)\right) +f_3(y_{II}^0(0)) +f_4\big(d_2(0),\fT_2(y^0)(0)\big)\\
+f_5(0)u_I(0)+f_6(0)u_{II}(0)=0,
\end{multline}
where $u_I(0), u_{II}(0)$ are defined by~\eqref{FcDAERd1-1} and~\eqref{FcDAERd1-2}, respectively.
If $h=0$, then the initial values are adjusted accordingly as in~\eqref{eq:nonlDAE-Ic}.\\
Regarding \eqref{eq:DAE} as a model of some real-world dynamical process, it is reasonable to assume consistency in the absence of feedback --
otherwise, the integrity of the model is suspect.
In the context of DAEs, and invoking the behavioural approach~\cite{PoldWill98,Will07}, a clear distinction between inputs, states, and outputs is often not possible during the modeling procedure. The interpretation of variables should be done after the analysis of the model reveals the free variables, which ``can be viewed as unexplained by the model
and imposed on the system by the environment''~\cite{PoldWill98}. In this way the physical meaning of the system variables is respected.
In the presence of feedback, the input variables $u_I$, $u_{II}$ should be part of any consistency condition, as they are constituents of the model.
\end{rem}

Feasibility of the controller~\eqref{FcDAERd1-1},~\eqref{FcDAERd1-2} for~DAE systems ${(d_1,d_2,f_1,\ldots,f_6,\fT_1,\fT_2)}
\in\mathcal{ND}^{m,1,\ell}$ is shown in~\cite[Thm.~4.3]{BergLe20} for $\alpha(s) = 1/(1-s)$; the extension to general $\alpha$ is straightforward.

\begin{thm}\label{Thm:DAE-funnel}
Consider system~\eqref{eq:DAE} with $(d_1,\ldots,d_4,f_1,\ldots,f_5,\fT_1,\fT_2)\in\mathcal{ND}^{m,1,\ell}$, $m\in\N$, $\ell\in\{0,\ldots,m\}$. Choose $\varphi_{I}\in\Phi$, $\varphi_{II}\in\Phi\cap \cW^{1,\infty}(\R_{\ge 0},\R)$,
a surjection $N\in \cC(\R_{\ge 0},\R)$, a bijection $\alpha\in \cC^1 ([0,1),[1,\infty))$, and $\hat k>0$ such that~\eqref{eq:init-gain} holds. Let $y_{\mathrm{ref}}\in \cW^{1,\infty}(\R_{\ge 0},\R^m)$ be arbitrary and assume that
{the initial data is consistent, in the sense that \eqref{eq:consistent-IV} holds, and }
\begin{multline}\label{ic-DAE}
\varphi_I(0) \|y_I(0)-y_{{\rm ref},I}(0)\| < 1\\
\text{and}\quad \varphi_{II}(0) \|y_{II}(0)-y_{{\rm ref},{II}}(0)\| < 1.
\end{multline}
Then the funnel control~\eqref{FcDAERd1-1},~\eqref{FcDAERd1-2} applied to~\eqref{eq:DAE} with $r=1$ yields an initial-value problem which has a solution (in the sense of Carath\'{e}odory), every solution can be maximally extended and every maximal solution
$y:\left[-h,\omega\right)\rightarrow \R^m$ has the properties:
\begin{enumerate}[(i)]
\item
$\omega=\infty$ (global existence);
\item
$u\in\cL^\infty (\R_{\ge0},\R^m)$, $y\in\cW^{1,\infty}([-h,\infty),\R^m)$, $k\in\cL^\infty (\R_{\ge0},\R)$;
\item
the tracking errors $e_I(t)=y_I(t)-y_{{\rm ref},I}(t)$ and $e_{II}(t) = y_{II}(t)-y_{{\rm ref},{II}}(t)$ evolve strictly inside the funnels~$\mathcal{F}_{\varphi_I}$ and~$\mathcal{F}_{\varphi_{II}}$, resp.,  
in the sense that there exist ${\varepsilon \in (0,1)}$ such that
\[
    \forall\, t\ge0:\ \varphi_I(t)\|e_I(t)\|\leq {\varepsilon}~\text{and}~\varphi_{II}(t)\|e_{II}(t)\|\leq {\varepsilon}.
\]
\end{enumerate}
\end{thm}
This result is a consequence of~\cite[Thm.~4.3]{BergLe20} with straightforward modifications accounting for the controller part~\eqref{FcDAERd1-1}, which follows from Theorem~\ref{Thm:FunCon-Nonl-Rd1}.

\subsection{The higher-relative-degree case: derivative feedback}\label{Ssec:dervi-fb}
\noindent
Approaches to funnel control of systems of relative degree greater than one separate into two categories according to the information available
for feedback to the controller.  Throughout, it is (reasonably) assumed that the instantaneous values of the system output and reference signal are
available.  However, in cases of relative degree greater than one, the derivatives of the output and reference signals play a role.
In applications, such  derivatives
may or may not be available for feedback: we
distinguish  these two scenarios via the terminology {\it derivative} and {\it non-derivative feedback},
respectively; the latter is discussed in Section~\ref{Ssec:non-deriv-fb}.

In the context of the first scenario, it might be argued that the control problem is reducible to that of the relative-degree-one case.
For example, consider the relative-degree-two system
\centerline{$\ddot y(t)=f(d(t),y(t),\dot y(t),u(t))$, \  $y(0)=y^0$, \ $\dot y(0)=v^0$,}
 and assume that the output
derivative~$\dot y$ is available for feedback.  Introducing the surrogate output $z(t)=y(t)+\dot y(t)$, the system may be expressed as
\begin{align*}
\dot y (t)&=-y(t)+z(t),\\
\dot z(t)&=-y(t)+z(t) + f(d(t),y(t),\dot y(t),u(t))
\end{align*}
which, on defining ${\mathbf T}_0\colon  \cC (\R_{\ge 0},\R^m)\to\cC (\R_{\ge 0},\R^m)$ by ${\mathbf T}_0(z)(t):=\int_0^t e^{-(t-s)} z(s)\dd s$ and
writing $d_0\colon t\mapsto e^{-t}y^0$, takes the form
\begin{align*}
&\dot z(t)=\tilde f(\tilde d(t), \tilde{\mathbf T}(z)(t),u(t)),~~z(0)=z^0=y^0+v^0,\\
&\tilde d(\cdot )=(d(\cdot),d_0(\cdot)),\quad \tilde{\mathbf T}\colon z\mapsto (z,{\mathbf T}_0(z))
\end{align*}
with $\tilde f\colon (\delta, w, v)=\big((\delta_1,\delta_2),(w_1,w_2),v\big)\mapsto -(w_2+d_2)+w_1 +f(\delta_1,w_2+d_2,w_1-w_2-d_2,v)$.   This is a
system of relative degree one amenable to funnel control through application of Theorem~\ref{Thm:FunCon-Nonl-Rd1}.
However, this simple observation is somewhat misleading.
Application of Theorem~\ref{Thm:FunCon-Nonl-Rd1} ensures prescribed transient and asymptotic behaviour of the {\it surrogate} output~$z(\cdot)$  but
the true objective of causing that the {\it actual} output~$y(\cdot)$ to evolve in a prescribed funnel is not guaranteed.    Attainment of the true
objective using derivative feedback is the subject of Theorem~\ref{Thm:FunCon-Nonl} below.

\subsubsection{Functional differential and nonlinear differential-algebraic systems}
\noindent
We present a recent result on funnel control for the class~$\cN^{m,r}$, which generalizes an earlier contribution from~\cite{BergLe18}, see Section~\ref{Ssec:non-backst}. It also sheds some new light on systems with unknown control directions, which remains an active research area, see e.g.~\cite{ChoiYoo16,Liu16,LiuLim17,GuoXu16,WuLi17,YinGao17,ZhanYang17b}. We stress that several of the classes discussed in those papers
(albeit with some restrictions if necessary) are contained in the class of nonlinear systems~\eqref{eq:nonlSys}. What the aforementioned approaches also have in common is a level of complexity greater than that of the funnel controller that we describe below.
\\[1ex]
{\it Information available for feedback.}\ \
Throughout, it is assumed that the instantaneous value of the output $y(t)$ and its first $r-1$ derivatives $\dot y(t),\ldots,y^{(r-1)}(t)$
are available for feedback.
Admissible reference signals are functions $y_{\textrm{ref}} \in \cW^{r,\infty}(\R_{\ge 0},\R^m)$.
The instantaneous reference value~$y_{\textrm{ref}}(t)$ is assumed to be accessible to the controller
and, if $r\geq 2$, then, for some $\hat r\in\{1,\ldots,r\}$, the derivatives $\dot y_{\textrm{ref}}(t),\ldots,y^{(\hat r-1)}_{\textrm{ref}}(t)$ (a vacuous list if $\hat r=1$) are also
accessible for feedback.  In summary,
 for some $\hat r\in \{1,\ldots,r\}$, the following
instantaneous vector is available
for feedback purposes:
\begin{equation}\label{eq:fback-quantities}
\begin{array}{l}
{\mathbf e(t)}=\big(e^{(0)}(t),\ldots, e^{(\hat r-1)}(t),
y^{(\hat r)}(t),\ldots,y^{(r-1)}(t)\big)\!\in\!\R^{rm},
\\
{e(t):=y(t)-y_{\textrm{ref}}(t),}\end{array}
\end{equation}
with the convention that $e^{(0)}\equiv e$ and ${\mathbf e(t)}=\big(e^{(0)}(t),\ldots, e^{(r-1)}(t)\big)$ if $\hat r=r$.

\noindent
{\it Feedback strategy.}\ \
As before, primary ingredients in the feedback construction, are the funnel control design parameters:
\begin{equation}\label{eq:fcts-FC}
\left.
\begin{array}{l}
{\varphi\in\Phi,~~\text{bounded if \ $\hat r < r$},}\\
N\in \cC(\R_{\ge 0},\R),~~\text{a surjection,}
\\[.6ex]
\alpha\in \cC^1( [0,1), [1,\infty)),~~\text{a bijection.}
\end{array}\right\}
\end{equation}
{These functions are open to choice.} For notational convenience, define
\begin{multline}\label{eq:fcts-FC2}
\cB :=\setdef{w\in\R^m }{ \|w \| < 1}\\
\text{and}~~\gamma\colon\cB\to\R^m, \ w\mapsto  \alpha (\|w\|^2)\, w.
\end{multline}
Next, we introduce continuous maps $\rho_k\colon\cD_k\to \cB$, $k=1,\ldots,r$, recursively as follows:
\begin{equation}\label{eq:fcts-FC3}
\begin{array}{l}
\cD_1\!:=\!\cB,\quad \rho_1\colon\cD_1\to\cB,~\eta_1\mapsto\eta_1,
\\[0.6ex]
\cD_k\!:=\! \setdef{\!\!\!(\eta_1,\ldots,\eta_k)\!\in\!\R^{km}\!\!\!}{\!\!\!\!\!\begin{array}{l} (\eta_1,\ldots,\eta_{k-1})\in\cD_{k-1},\\
 \eta_k\!+\!\gamma(\rho_{k-1}(\eta_1,\ldots,\eta_{k-1}))\!\in\!\cB\end{array}\!\!\!\!\!\!}\!,
\\[1.9ex]
\rho_k\colon\cD_k\!\to\!\cB,\ (\eta_1,\ldots,\eta_k)\mapsto \eta_k +\gamma (\rho_{k-1}(t,\eta_1,\ldots,\eta_{k-1})),
\end{array}
\end{equation}
see also Fig.~\ref{Fig:rho-k}. Note that each of the sets~$\cD_k$ is non-empty and open.
With~$\mathbf{e}$ and~$\rho_r$ defined by~\eqref{eq:fback-quantities}
and~\eqref{eq:fcts-FC3},
the \textit{funnel controller} is given  by
\begin{equation}\label{eq:FC-}
\framebox{$
u(t)=\big(N\circ \alpha\big)(\|w(t)\|^2) \, w(t),\quad w(t)=\rho_r\big(\varphi (t)\mathbf{e}(t)\big),$
}
\end{equation}
which, in the relative degree one case $r=1$, corresponds to \eqref{FcRd1}; see Fig.~\ref{Fig:funnel-controller} for an illustration of its construction.

\begin{figure}[!t]
\begin{center}
 \begin{tikzpicture}[thick,node distance = 12ex, box/.style={fill=white,rectangle, draw=black}, blackdot/.style={inner sep = 0, minimum size=3pt,shape=circle,fill,draw=black},plus/.style={fill=white,circle,inner sep = 0, minimum size=5pt,thick,draw},metabox/.style={inner sep = 3ex,rectangle,draw,dotted,fill=gray!20!white}]
  xshift=-10ex;
  \node (u4) [shape = rectangle, rounded corners = 3mm, fill=white, draw = black]
{$\begin{array}{l}
\rho_{k-1}\colon \cD_{k-1}\subset\R^{(k-1)m} \to\cB,
\\
(\eta_1,\ldots,\eta_{k-1})=:\boldsymbol{\eta}_{k-1}\mapsto \rho_{k-1}(\boldsymbol{\eta}_{k-1})
\end{array}
$};
\node (u5) [shape = rectangle, rounded corners = 3mm, fill=white, below of = u4, draw = black, yshift = -4ex]
{$
\begin{array}{l}
\boldsymbol{\eta}_k :=(\boldsymbol{\eta}_{k-1},\eta_k)\in\R^{(k-1)m}\times\R^m
\\
\cD_k :=\{\boldsymbol{\eta}_k\mid \boldsymbol{\eta}_{k-1}\in\cD_{k-1}, ~\eta_k+\gamma(\rho_{k-1}(\boldsymbol{\eta}_{k-1}))\in\cB \}
\\
\rho_k\colon\cD_k\to\cB,\ \boldsymbol{\eta}_k\mapsto \eta_k+\gamma(\rho_{k-1}(\boldsymbol{\eta}_{k-1}))
\end{array}
$};
 \draw[->,arrows=-Latex] (u4) -- (u5) node[midway,right]{$(k-1)\to k$};
\node (u6) [shape = rectangle, rounded corners = 3mm, fill=white, above of = u4, draw = black, yshift = -2ex]
{\ Recursion initialized by $\rho_1\colon \cD_1:=\cB\to\cB,\ \eta_1\mapsto\eta_1$\ \ };
 \end{tikzpicture}
 \caption{Illustration of the recursive definition of $\rho_k$ in~\eqref{eq:fcts-FC3}.} \label{Fig:rho-k}
 \end{center}
 \end{figure}

The efficacy of funnel control for systems~\eqref{eq:nonlSys} belonging to the class~$ \cN^{m,r}$ was established in~\cite{BergIlch21}:
we restate this result here.

\captionsetup[subfloat]{labelformat=empty}
\begin{figure*}[h!t]
\centering
   \begin{tikzpicture}[very thick,scale=0.7,node distance = 9ex, box/.style={fill=white,rectangle, draw=black}, blackdot/.style={inner sep = 0, minimum size=3pt,shape=circle,fill,draw=black},blackdotsmall/.style={inner sep = 0, minimum size=0.1pt,shape=circle,fill,draw=black},plus/.style={fill=white,circle,inner sep = 0,very thick,draw},metabox/.style={inner sep = 3ex,rectangle,draw,dotted,fill=gray!20!white}]
 \begin{scope}[scale=0.5]
    \node (sys) [box,minimum size=7ex]  {$y^{(r)}(t)= f\big(d(t), \fT(y,\dot{y},\dots,y^{(r-1)})(t), u(t)\big)$};
    \node [minimum size=0pt, inner sep = 0pt,  below of = sys, yshift=3ex] {System $(d,f,\fT)\in  \cN^{m,r}$};
    \node(fork1) [minimum size=0pt, inner sep = 0pt,  right of = sys, xshift=20ex] {};
    \node(end1)  [minimum size=0pt, inner sep = 0pt,  right of = fork1, xshift=18ex] {$\big(y,\ldots,y^{(r-1)}\big)$};

   \draw[->] (sys) -- (end1) node[pos=0.4,above] {};

  \node(FC0) [box, below of = fork1,yshift=0ex,minimum size=7ex] {{${\mathbf e}(t)$ as in~\eqref{eq:fback-quantities}}};
    \node(FC1) [box, below of = FC0,yshift=-5ex,minimum size=7ex] {$w(t)=\rho_r\big(\varphi(t){\mathbf e}(t)\big)$};
    \node(phi) [minimum size=0pt, inner sep = 0pt,  right of = FC0, xshift=18ex] {$\big(y_{\rm ref},\ldots,y_{\rm ref}^{(\hat r-1)}\big)$};
    \draw[->] (fork1) -- (FC0) {};
    \draw[->] (phi) -- (FC0) node[midway,above] {};
    \draw[->] (FC0) -- (FC1) node[midway,right] {${\mathbf e}$};
    \node(FC2) [box, left of = FC1,xshift=-30ex,minimum size=7ex] {$u(t) =\big(N\circ \alpha\big)(\|w(t)\|^2) \, w(t)$};
   \node(alpN) [minimum size=0pt, inner sep = 0pt,  below of = FC2, yshift=-5ex] {};
   \draw[->] (FC1) -- (FC2) node[midway,above] {$w$};
   \node (DP)[box, right of = alpN,xshift=10ex,yshift=2ex,minimum size=5ex]{Design parameters as in~\eqref{eq:fcts-FC}};
   \draw[->] (DP) -| (FC2) node[pos=0.8,right] {$\alpha$, $N$};
   \draw[->] (DP) -| (FC1) node[pos=0.8,right] {$\varphi$};
   \node(fork2) [minimum size=0pt, inner sep = 0pt,  left of = sys, xshift=-20ex] {};
   \draw (FC2) -| (fork2.north) {};
   \draw[->] (fork2.west) -- (sys) node[pos=0.7,above] {$u$};

   \node [minimum size=0pt, inner sep = 0pt,  below of = alpN, yshift=5ex, xshift=-5.5ex] {Funnel controller~\eqref{eq:FC-}};
\end{scope}
\begin{pgfonlayer}{background}
      \fill[lightgray!20] (-6,-3) rectangle (9,-9.2);
      \draw[dotted] (-6,-3) -- (9,-3) -- (9,-9.2) -- (-6,-9.2) -- (-6,-3);
  \end{pgfonlayer}
  \end{tikzpicture}
\caption{Construction of the funnel controller~\eqref{eq:FC-} with design parameters~\eqref{eq:fcts-FC}.}
\label{Fig:funnel-controller}
\end{figure*}

\begin{thm}\label{Thm:FunCon-Nonl}
Consider system~\eqref{eq:nonlSys} with $(d,f,\fT)\in \cN^{m,r}$, $m,r\in\N$, and initial data as in~\eqref{eq:nonlSysIc}.
Choose the triple $(\alpha,N,\varphi)$ of funnel control design parameters as in~\eqref{eq:fcts-FC}
and let $y_{\mathrm{ref}}\in \cW^{r,\infty}(\R_{\ge 0},\R^m)$ be arbitrary.
Assume that, for some $\hat r\in\{1,\ldots,r\}$,
the instantaneous vector ${\mathbf{e}}(t)$, given by~\eqref{eq:fback-quantities}, is available for feedback and the following holds:
\begin{equation}\label{ic-}
\varphi(0){\mathbf{e}}(0)\in \cD_r,
\end{equation}
(trivially satisfied if $\varphi(0)=0$).
Then the funnel control~\eqref{eq:FC-} applied to~\eqref{eq:nonlSys} yields an initial-value problem which has a solution (in the sense of Carath\'{e}odory), every solution can be maximally extended and every maximal solution
$y:\left[-h,\omega\right)\rightarrow \R^m$ has the properties:
\begin{enumerate}[(i)]
\item \label{item-main-1-}
$\omega=\infty$ (global existence);
\item \label{item-main-2-}
$u\in\cL^\infty (\R_{\ge0},\R^m)$, $y\in\cW^{r,\infty}([-h,\infty),\R^m)$;
\item \label{item-main-3-}
the tracking error $e\colon \R_{\ge 0}\to \R^m$ as in~\eqref{eq:fback-quantities} evolves strictly inside the funnel~$\mathcal{F}_{\varphi}$
in the sense that there exists $\varepsilon \in (0,1)$ such that
$\varphi(t)\|e(t)\|\leq \varepsilon$ for all $t\geq 0$.
\item \label{item-main-4}
If $\hat r> 1$ and~$\varphi$ is unbounded, then $e^{(k)}(t)\to 0$ as $t\to\infty$, \  $k=0,\ldots,\hat r-1$.
\end{enumerate}
\end{thm}

\begin{rem}\label{Rem:IMP}
The above result presents a possible anomaly:  performance of funnel control might seem to  contradict the {\it internal model principle} which asserts
that ``a regulator is
structurally stable only if the controller~[\ldots]   incorporates~[\ldots]
 a suitably reduplicated model of the dynamic structure of the exogenous signals which
the regulator is required to process''~\cite[p.\,210]{Wonh79}.   However, as discussed in \cite{BergIlch21}, this potential anomaly may be resolved.
The internal model principle applies in the context of {\it exact} asymptotic tracking of reference signals.  In the case of a {\it bounded} funnel function~$\varphi$,
only {\it approximate} tracking, with non-zero prescribed asymptotic accuracy, is assured and so the anomaly evaporates.
But what of the case of an unbounded funnel function~$\varphi$, which is permissible whenever~$\hat r =r$?    In this case, exact asymptotic tracking is achieved.
Returning to the control-theoretic origins of the internal model principle, summarised in~\cite[p.\,210]{Wonh79} as ``every good regulator must incorporate a
model of the outside world'', we regard the term ``good regulator'' as  most pertinent.
A good regulator should exhibit robustness with respect to sufficiently small disturbances/noise/modelling inaccuracies.
However, the case of unbounded $\varphi$ inevitably
contains (easily constructed) examples of $L^2$ disturbances of positive, but arbitrarily small norm, that cause the controlled process to violate the requirement of
strict evolution within the funnel.
 Whilst of theoretical interest, the case of unbounded~$\varphi$ is of limited practical utility.
\end{rem}

\begin{rem}\label{Rem:N(s)=-s}
Note that a ``switching function'' $N$ is used in the controller~\eqref{eq:FC-} to encompass the case of unknown control direction.  If the control direction is known,
then a simpler design can be used. More precisely, if, for fixed $\sigma\in\{-1,+1\}$ known to the controller,  the function~$\chi:\R\to\R$
 in~(NP1) is such that $\chi(\sigma s)\to\infty$ as $s\to\infty$
for all compact $K_p\subset \R^p$ and compact $K_q\subset \R^q$, then, setting $N:s\mapsto \sigma s$, the assertions of  Theorem~\ref{Thm:FunCon-Nonl}
remain valid.
\end{rem}

\begin{rem}
We comment on the possible choices of the funnel control design parameters in~\eqref{eq:fcts-FC} and provide some guidance for practical applications. First of all, the function~$\varphi\in \Phi$ defines the performance funnel for the tracking error, hence its choice should reflect the safety specifications of the considered application as well as physical limitations. Some choices are discussed in Section~\ref{Sec:intro}, another typical choice is $\varphi(t) = (a e^{-bt} + c)^{-1}$ for appropriate parameters $a,b,c>0$ taking the aforementioned criteria into account;  here the funnel boundary tends, for~$t\to\infty$, to the constant funnel with boundary $c^{-1}$. For the function~$N$ it is recommended to use the simple versions mentioned in Remark~\ref{Rem:N(s)=-s}, if the control direction is known. If it is unknown, then $N(s) = s\sin s$ is a standard choice. For the bijection $\alpha$ it usually suffices to consider the choice $\alpha(s) = 1/(1-s)$.
\end{rem}

\begin{rem}
 Other versions of the funnel controller, that are not encompassed by the design~\eqref{eq:FC-} can be found in the literature, see e.g.~\cite{BergLe18,IlchRyan02b}. These modifications do not change the qualitative behaviour of the controller and will not be discussed here in detail. Depending on the application, one of the modifications might be preferred over the controller~\eqref{eq:FC-}.
\end{rem}
For~DAE systems~\eqref{eq:DAE} the controller~\eqref{eq:FC-} needs to be adjusted appropriately, that is for $y_{{\rm ref},I}\in \cW^{r,\infty}(\R_{\ge 0},\R^m)$ we define the signal
\begin{multline*}
    {\mathbf e_I(t)}\!=\!\big(e_I^{(0)}(t),\ldots, e_I^{(\hat r-1)},
y_I^{(\hat r)}(t),\ldots,y_I^{(r-1)}(t)\big)\!\in\!\R^{rq},
\\
{e_I(t)\!=\!y_I(t)-y_{{\rm ref},I}(t),}
\end{multline*}
and for $\varphi_I \in\Phi$ we set
\begin{equation}\label{eq:FC-DAE}
u_I(t)=\big(N\circ \alpha\big)(\|w(t)\|^2) \, w(t),\quad w(t)=\rho_r\big(\varphi_I (t){\mathbf e_I(t)}\big),
\end{equation}
which is combined with the controller~\eqref{FcDAERd1-2} for the algebraic part that stays unchanged. Furthermore, we extend the notion of consistent initial values from~\eqref{eq:consistent-IV} to arbitrary relative degree, i.e., to the condition
\begin{multline}\label{eq:consistent-IV-r>1}
f_2\left(y_I^0(0),\ldots,\big(y_I^0\big)^{(r-1)}(0)\right) +f_3(y_{II}^0(0))\\
 +f_4\big(d_2(0),\fT_2(y_I^0,y_{II}^0)\big) +f_5(0)u_I(0)+f_6(0)u_{II}(0)=0.
\end{multline}

Funnel control for~DAE systems with arbitrary relative degree has been discussed in~\cite{Berg14a,BergIlch12c}, but for system classes smaller than~\eqref{eq:DAE}. The result given below is a consequence of~\cite[Thm.~4.3]{BergLe20}, with slight modification. In fact, it is a straightforward combination of Theorems~\ref{Thm:FunCon-Nonl} and~\ref{Thm:DAE-funnel}, since the controller~\eqref{FcDAERd1-2} of the algebraic part does not change.

\begin{thm}\label{Thm:FunCon-Nonl-DAE}
Consider the DAE~system~\eqref{eq:DAE} with
\begin{multline*}
(d_1,\ldots,d_4,f_1,\ldots,f_5,\fT_1,\fT_2)\in\mathcal{ND}^{m,r,\ell},\\
m,r\in\N,~~\ell\in\{0,\ldots,m\}.
\end{multline*}
Choose a triple $(\alpha,N,\varphi_I)$ of funnel control design parameters as in~\eqref{eq:fcts-FC}, $\varphi_{II}\in\Phi\cap \cW^{1,\infty}(\R_{\ge 0},\R)$ and $\hat k>0$ such that~\eqref{eq:init-gain} holds. Let $y_{\rm ref}$ be such that
$y_{{\rm ref},I}\in \cW^{r,\infty}(\R_{\ge 0},\R^{\ell})$
and
$y_{{\rm ref},{II}}\in \cW^{1,\infty}(\R_{\ge 0},\R^{m-\ell})$,
assume that the initial data is consistent, in the sense that \eqref{eq:consistent-IV-r>1} holds, and
\begin{multline}\label{ic-DAE-}
\varphi_I(0) {\mathbf e_I(0)}\! \in\! \cD_r\ \text{(as in~\eqref{eq:fcts-FC3}
with $m=\ell$)}\\ \text{and}\ \ \varphi_{II}(0) \|y_{II}(0)-y_{{\rm ref},{II}}(0)\| \!<\! 1.
\end{multline}
Then the funnel control~\eqref{eq:FC-DAE},~\eqref{FcDAERd1-2} applied to~\eqref{eq:DAE} with $r=1$ yields an initial-value problem which has a solution (in the sense of Carath\'{e}odory), every solution can be maximally extended and every maximal solution
$y:\left[-h,\omega\right)\rightarrow \R^m$ has the properties:
\begin{enumerate}[(i)]
\item
$\omega=\infty$ (global existence);
\item
$u\in\cL^\infty (\R_{\ge0},\R^m)$, $y\in\cW^{r,\infty}([-h,\infty),\R^m)$, $k\in\cL^\infty (\R_{\ge0},\R)$;
\item
the tracking errors $e_I(t)=y_I(t)-y_{{\rm ref},I}(t)$ and $e_{II}(t) = y_{II}(t)-y_{{\rm ref},{II}}(t)$ evolve strictly inside the funnels~$\mathcal{F}_{\varphi_I}$ and~$\mathcal{F}_{\varphi_{II}}$, resp.,
in the sense that there exists ${\varepsilon \in (0,1)}$ such that
$\varphi_I(t)\|e_I(t)\|\leq {\varepsilon}$ and $\varphi_{II}(t)\|e_{II}(t)\|\leq {\varepsilon}$ for all $t\ge 0$.
\end{enumerate}
\end{thm}

\subsubsection{Antecedent approaches}\label{Ssec:non-backst}
\noindent
\paragraph{A relative degree two funnel controller}
In the case of single-input, single-output,  nonlinear systems with relative degree two and asymptotically stable zero dynamics,
a funnel controller has been proposed by \textit{Hackl, Hopfe, Ilchmann,  Mueller,  Trenn} (2013)~\cite{HackHopf13} (see also the modification in~\cite{Hack11}). The aim in the control design was to avoid the backstepping procedure from~\cite{IlchRyan07} (see Section~\ref{Ssec:backstep}) by using a linear combination of the output and its derivative instead.

The systems which are considered in~\cite{HackHopf13} are of the form~\eqref{eq:nonlSys} with $m=1$, $r=2$, $g=0$ and $f(\delta,\eta,u) = f_1(\delta,\eta) + f_2(\delta,\eta) u$ for suitable functions~$f_1$ and~$f_2$. It is assumed that $f_2(\delta,\eta) > 0$ everywhere. The work~\cite{HackHopf13} introduces a~funnel controller which feeds back the error~$e$ and its derivative. Compared to Theorem~\ref{Thm:FunCon-Nonl},
 it is possible to directly prescribe the evolution of the error derivative. The controller reads
\begin{equation}\label{eq:PDfunnel}
\boxed{
\begin{array}{l}
u(t)=-k_0^2(t)e(t)-k_1(t)\dot{e}(t),\\[2mm]
k_0(t)=\frac{\varphi_0(t)}{1-\varphi_0(t)|e(t)|},\quad
k_1(t)=\frac{\varphi_1(t)}{1-\varphi_1(t)|\dot{e}(t)|}.
\end{array}
}
\end{equation}
The funnel functions $\varphi_0$ for the error and $\varphi_1$ for the derivative of the error have to satisfy $(\varphi_0,\varphi_1)\in \Phi^2$; the latter class is defined by
\[
    \Phi^2:=\setd{\!\!(\varphi_0,\varphi_1)\in\Phi\times \Phi}{\!\!\!\begin{array}{l} \exists\, \delta> 0\ \text{for a.a.}\ t> 0:\\      (1/\varphi_1)(t) + \ddt (1/\varphi_0)(t) \geq \delta\end{array}\!\!\!},
\]
where $\Phi$ is as in~\eqref{Phi}. The motivation for the definition of~$\Phi^2$ is that the derivative funnel $\mathcal{F}_{\varphi_1}$ must be large enough to allow the error
to follow the funnel boundaries; for more details see~\cite{HackHopf13}. Feasibility of the control~\eqref{eq:PDfunnel} is shown in~\cite[Thm.~3.1]{HackHopf13}.

As shown in~\cite{Hack11,Hack12}, see also~\cite[Sec.~9.4.4]{Hack17}, the equation for~$u(t)$ in the controller~\eqref{eq:PDfunnel} can be modified such that
\begin{equation}\label{eq:PDfunnel-mod}
     u(t) = -k_0(t)^2 e(t) - k_0(t) k_1(t) \dot e(t)
\end{equation}
and feasibility of the control is still guaranteed;
 in~\cite{Hack11,Hack12} this is shown for a certain class of linear systems, but the extension to nonlinear systems~\eqref{eq:nonlSys} as discussed above is straightforward.

The modification~\eqref{eq:PDfunnel-mod} is advantageous compared to~\eqref{eq:PDfunnel}, since the latter yields a badly damped closed-loop system
response and may lead to admissibility problems in applications since speed measurement is usually very noisy. The controller~\eqref{eq:PDfunnel} (and its modification~\eqref{eq:PDfunnel-mod}) is simple and its practicability has been verified experimentally. Its advantage is that the performance of both~$e$ and~$\dot e$ may be prescribed. However, there is no straightforward extension to systems with relative degree larger than two.

\paragraph{Non-backstepping feedback for higher relative degree} A funnel controller for systems with arbitrary relative degree $r\in\N$ was introduced
by \textit{Berger, Hoang, and Reis} (2018)~\cite{BergLe18} for systems of the form~\eqref{eq:nonlSys} with~$g=0$ and $f(\delta,\eta,u) = f_1(\delta,\eta) + f_2(\delta,\eta) u$ for suitable functions~$f_1$ and~$f_2$ such that $f_2(\delta,\eta) + f_2(\delta,\eta)^\top \succ 0$ everywhere.
This controller, which does not involve any backstepping procedure, is an output error feedback of the form
$u(t) = F(t,e(t),\dot{e}(t),\ldots,e^{(r-1)}(t))$, where $e(t)= y(t) - y_{\rm \scriptsize ref}(t)$
evolves within the performance funnel~$\mathcal{F}_{\varphi}$ which is determined by a function~$\varphi$ belonging to
\begin{equation}
\Phi_r :=
\left\{
\varphi\in  \cC^r(\R_{\ge 0},\R)
\left|\!\!\!
\begin{array}{l}
\text{ $\varphi, \dot \varphi,\ldots,\varphi^{(r)}$ are bounded,}\\
\text{ $\varphi (\tau)> 0$ for all $\tau> 0$,}\\
 \text{ and }  \liminf_{\tau\rightarrow \infty} \varphi(\tau) > 0
\end{array}\!\!\!
\right.
\right\}.
\label{eq:Phir}\end{equation}
The controller is of the form
\begin{equation}\label{eq:fun-con}
\boxed{\begin{aligned}
e_0(t)&=e(t) = y(t) - y_{\rm \scriptsize ref}(t),\\
e_1(t)&=\dot{e}_0(t)+k_0(t)\,e_0(t),\\
&\ \vdots \\
e_{r-1}(t)&=\dot{e}_{r-2}(t)+k_{r-2}(t)\,e_{r-2}(t),\\
k_i(t)&=\frac{1}{1-\varphi_i(t)^2\|e_i(t)\|^2},\quad i=0,\dots,r-1, \\
 u(t) &= -k_{r-1}(t)\,e_{r-1}(t),
\end{aligned}
}
\end{equation}
where the reference signal and funnel functions satisfy:
\begin{multline}\label{eq:fun-con-ass}
y_{\rm \scriptsize ref}\in\, \mathcal{W}^{r,\infty}(\R_{\ge 0},\R^m),\\
\varphi_0\in\, \Phi_r,\;\;
\varphi_1\in \Phi_{r-1},\;\ldots,\;\;
\varphi_{r-1}\in \Phi_{1}.
\end{multline}

\noindent
We stress that $\dot e_0,\ldots,\dot e_{r-2}$ in~\eqref{eq:fun-con} merely serve as short-hand notations and may be resolved in terms of~$e^{(i)}, k_i$ and $\varphi_i$, $i=0,\ldots,r-1$, where $e^{(i)}$ is assumed to be available to the controller. Therefore, the control law may be reformulated accordingly; in the following we determine the funnel controllers explicitly for the cases $r=1$ and $r=2$.
\begin{enumerate}[\hspace{2pt}$r=1$:]
\item[$r=1$:] The control law~\eqref{eq:fun-con} reduces to the ``classical'' funnel controller~\eqref{eq:FC}.
\item[$r=2$:] We obtain the controller
\[
\begin{aligned}
u(t) &=-k_1(t)(\dot{e}(t)+k_0(t)e(t)),\\
k_0(t)&=\frac{1}{1-\varphi_0^2(t)\|e(t)\|^2},\\
k_1(t)&=\frac{1}{1-\varphi_1^2(t)\|\dot{e}(t)+k_0(t)e(t)\|^2}.
\end{aligned}
\]
We stress that this controller is different from both the relative degree two funnel controller~\eqref{eq:PDfunnel} and its modification~\eqref{eq:PDfunnel-mod}.
\end{enumerate}

Feasibility of the control~\eqref{eq:fun-con} is shown in~\cite[Thm.~3.1]{BergLe18}.
We emphasize that, compared to the bang-bang funnel controller (which is  another antecedent approach discussed in detail in  Section~\ref{Ssec:bang-bang}) and the relative degree two funnel controller~\eqref{eq:PDfunnel}, the funnel functions $\varphi_0,\ldots,\varphi_{r-1}$ in the controller~\eqref{eq:fun-con} do not have to satisfy any compatibility condition. However, the control design~\eqref{eq:fun-con} involves successive derivatives of the auxiliary error variables~$e_i$, which exhibit an increasing complexity for higher relative degree, which is also illustrated by the explicit control law for the cases~$r=2$ and~$r=3$ presented above. The simple funnel control design~\eqref{eq:FC-} helps to resolve these issues.
\subsubsection{Prescribed performance control}\label{Sssec:PPC}
\noindent
An alternative approach to funnel control has been developed by \textit{Bechlioulis and Rovithakis (2008)}~\cite{BechRovi08}, which is called \textit{prescribed performance control}. In the first contributions,
feedback linearizable systems~\cite{BechRovi08}, strict feedback systems~\cite{BechRovi09} and general multi-input, multi-output systems which are affine in the control~\cite{BechRovi10} have been considered. An extension to systems with dead-zone input and time-delays is presented by \textit{Na (2013)} in~\cite{Na13} and further explored by \textit{Theodorakopoulos and Rovithakis (2015)} in~\cite{TheoRovi15}. Using so called \textit{performance functions}, which are special funnel boundaries, and a transformation that incorporates these performance functions, the original controlled system is transformed into a new one for which boundedness of the states, via the prescribed performance control input, can be proved. Therefore, the tracking error evolves in the funnel defined by the performance functions.

However, strictly speaking the controllers presented in~\cite{BechRovi08,BechRovi09,BechRovi10} are no funnel controllers since they are not of high-gain type. They have in common that neural networks are used to approximate the unknown nonlinearities of the system, which contrasts the classical funnel control methodology where parameter estimators
are not used. Problems of the approximation may be that disturbances or small errors in the approximation cause the tracking error to leave the performance funnel. Although a certain level of robustness is ensured, the controllers are not inherently robust since they are not of high-gain type. Furthermore, the controllers are prone to common challenges for approximation-based control schemes, both with the design and implementation, in particular the selection of the size of the neural network and the number of network parameters as well as the high order of the dynamics of the resulting controller because of the neural weight adaptive laws. Moreover, some parameters of the neural network must be chosen large enough, but it is not known a priori how large and suitable values must be identified by several simulations.

These drawbacks have been resolved by \textit{Bechlioulis and Rovithakis (2011)}~\cite{BechRovi11}, where the neural networks are avoided in the control design for single-input, single-output strict feedback systems. However, the controller is dynamic and incorporates~$r$ differential equations, where~$r$ is the relative degree of the system; this is due to the compensation of possibly unknown control directions and the controller is static in case of known directions. The dynamic component can be viewed as a filter, and it is needed in addition to the derivatives of the output. Finally, this filter is avoided in \textit{Bechlioulis and Rovithakis (2014)}~\cite{BechRovi14}, as the control directions are assumed to be known, and the complexity of the controller is further reduced; also, a feature of this controller is that no derivatives of the reference signal are needed. The class of systems considered in~\cite{BechRovi14} are so called pure feedback systems, which are of the form
\begin{equation}\label{eq:nonl-NF}
\begin{aligned}
  \dot x_k(t) &= f_k\big(x_1(t),\ldots, x_{k+1}(t)\big),~~k=1,\ldots,r-1,\\[1ex]
  \dot x_r(t) &= f_r\big(d(t), x_1(t),\ldots, x_r(t), \eta(t),u(t)\big),\\[1ex]
  \dot \eta(t) &= g\big(d(t), x_1(t),\ldots, x_r(t),\eta (t)\big),\\[1ex]
  y(t) &= x_1(t)
  \end{aligned}
\end{equation}
and initial data
\begin{multline}\label{eq:nonl-NFIc}
\big(x_1(0),\cdots,x_r(0),\eta(0)\big)=(x_1^0,\cdots,x_r^0,\eta^0\big)\\
\in\R^m\times \cdots\times \R^m\times\R^q.
\end{multline}
The considerations in~\cite{BechRovi14} are restricted to the case of no disturbances ($d=0$) and trivial internal dynamics ($q=0$); further, the partial derivatives $\frac{\partial f_i}{\partial x_{i+1}}$ and $\frac{\partial f_r}{\partial u}$ are assumed to be uniformly positive definite. We stress that in this system class no internal dynamics and no uncertainties or disturbances are allowed; the influence of the latter is discussed in~\cite{TheoRovi16}. Compared to~\cite{BechRovi14}, in the system class considered in~\cite{BechRovi11} internal dynamics of a certain hierarchical structure are allowed; these dynamics are called ``dynamics uncertainty'' there.

The \textit{prescribed performance controller} for the above described system class as introduced in~\cite{BechRovi14} is of the following form: First, a performance function~$\rho$ is chosen, which is usually of the form
\[
    \rho(t) = (\rho_0-\rho_\infty) e^{-\ell t} + \rho_\infty,\quad t\ge 0,
\]
where $\rho_0>\rho_\infty>0$, $\ell > 0$.
Clearly, $\varphi(t):=\rho(t)^{-1}$ defines a finite performance funnel with $\varphi\in\Phi$ for $\Phi$ as in~\eqref{Phi}. For $i=1,\ldots,r$ choose performance functions~$\rho_i(t) = \varphi_i(t)^{-1}$ and constants~$k_i>0$ and let
\begin{multline*}
    T_f:(-1,1)^m\to\R^m,\ (s_1,\ldots,s_m) \\
    \mapsto \left(\ln\left(\frac{1+s_1}{1-s_1}\right),\ldots, \ln\left(\frac{1+s_m}{1-s_m}\right)\right);
\end{multline*}
other choices for $T_f$ are possible (as long as it is continuously differentiable and bijective), but the above function is the standard choice in the literature. The prescribed performance controller is then given by
\begin{equation}\label{eq:PPC}
\boxed{\begin{aligned}
a_1(t)&=-k_1 T_f\Big(\varphi_1(t) \big(x_1(t) - y_{\rm ref}(t)\big)\Big),\\
a_2(t)&=-k_2 T_f\Big(\varphi_2(t) \big(x_2(t) - a_1(t)\big)\Big),\\
&\ \vdots \\
a_r(t)&=-k_r T_f\Big(\varphi_r(t) \big(x_r(t) - a_{r-1}(t)\big)\Big),\\
 u(t) &= a_r(t),
\end{aligned}
}
\end{equation}
where the performance functions must be such that for all $j=1,\ldots,m$ and $i=2,\ldots,r$ we have $\varphi_1(0)|x_{1,j}(0) - y_{{\rm ref},j}(0)|<1$ and $\varphi_i(0)|x_{i,j}(0) - a_{i-1,j}(0)|<1$.

It is shown in~\cite[Thm.~2]{BechRovi14} that the controller~\eqref{eq:PPC} applied to a system~\eqref{eq:nonl-NF} satisfying the conditions mentioned above, leads to a closed-loop system which has a solution and every maximal solution is global and bounded. Furthermore, each component $e_i(t) = x_{1,i}(t) - y_{{\rm ref},i}(t)$ of the tracking error evolves in the performance funnel $\cF_{\varphi_1}$, with  $\varphi_1(t)=\rho_1(t)^{-1}$.

Although funnel control and prescribed performance control achieve the same control objective and look similar in their controller structure, the two system classes~\eqref{eq:nonlSys} (amenable to funnel control) and~\eqref{eq:nonl-NF} (amenable to prescribed performance control) are different and a thorough comparison of the two approaches is still missing.

\subsubsection{Other approaches}\label{Sssec:FC-other}
\noindent
Further specialized approaches to funnel control include, but are not limited to:\footnote{We thank the anonymous reviewers for pointing out most of these references to us.} adaptive fuzzy funnel control~\cite{LiuWang21,LiuWang17}, fault-tolerant funnel control~\cite{Berg21,ZhanDing25}, neural-network-based funnel control~\cite{HanLee14,WangYu20}, observer-based funnel control~\cite{ChengRen22,ChowKhal19}, PI funnel control~\cite{Hack13,ZhanChai24}, and transition process-based funnel control~\cite{ZhanRen24}.

\subsection{Non-derivative feedback  via two methodologies: filtering and pre-compensation}\label{Ssec:non-deriv-fb}
\noindent
\noindent
Now we turn attention to the second scenario wherein derivative information on the output and reference signal are not available to the
controller, e.g.\ due to issues of accuracy and sensitivity to ``noise". In this scenario, a dynamic component (which we will label either a filter or a pre-compensator\footnote{We use these
terms loosely: they are intended to indicate a rationale that seeks
 to compensate for
the unavailability of output derivatives through (dynamic) operations on available input and output signals.  The terms ``filter"  and
 ``pre-compensator" are adopted solely to distinguish the two distinct methodologies.}),
  operating on available system input and output error data,
 is incorporated in the  control design in order to generate a vector of ``surrogate" variables $\boldsymbol{\xi}$ which deputises for the
 (unavailable) output derivatives in some appropriate sense, and which is used in a feedback
 $u(t)=U(t,e(t),\boldsymbol{\xi}(t))$ based only the available instantaneous information $(t,e(t),\boldsymbol{\xi}(t))$.
 We illustrate the  main features by means of a simple example.
\begin{figure}[h!]
  \centering
  \resizebox{\columnwidth}{!}{
  \begin{tikzpicture}[very thick,node distance = 12ex, box/.style={fill=white,rectangle, draw=black}, blackdot/.style={inner sep = 0, minimum size=3pt,shape=circle,fill,draw=black},plus/.style={fill=white,circle,inner sep = 0, minimum size=5pt,thick,draw},metabox/.style={inner sep = 3ex,rectangle,draw,dotted,fill=gray!20!white}]

\node (box0)		[box,left of=box1,minimum size=8ex,xshift=-24ex]{$u(t)=U(t,e(t),\boldsymbol{\xi}(t))$};
  \node (box1)		[box,minimum size=8ex]{$\begin{array}{c}\text{System:}\\\text{input $u$, output $y$,}\\\text{output error $e=y-y_{\text{ref}}$}\end{array}$};
    \node (box3)		[box,below of = box1,minimum size=8ex,yshift=-4ex]{$\begin{array}{c}\text{Dynamic control}\\\text{component}\end{array}$};
  \node (l1)   [above of = box0, minimum size=0pt, inner sep = 0pt, xshift=0ex, yshift=-4ex] {};

  \node (r3)   [below of = box0, minimum size=0pt, inner sep = 0pt, xshift=0ex, yshift=-4ex] {};
  \node (l2) [right of =  box1, minimum size=0pt, inner sep = 0pt, xshift=12ex, yshift=-0ex] {};

\draw(box3.west) -- (r3.west) node[midway,above]{$\text{surrogate}$} node[midway,below]{$\boldsymbol{\xi}(t)$};
\draw[->](r3.south) -- (box0.south);
\draw[->](box1) -- (box3) node[midway,right]{$(u(t),e(t))$};
 \draw (box1) |-(l1.east) node[pos=0.75,above]{$e(t)$};
  \draw[->](l1.north) -- (box0);
\draw[->](box1) -- (l2) node[right]{$y(t)$};
  \draw[->](box0) -- (box1) node[midway,above]{$u(t)$};
  \end{tikzpicture}
}
\caption{General structure.}
\vspace{-2mm}
\label{Fig:intercon-1}
\end{figure}
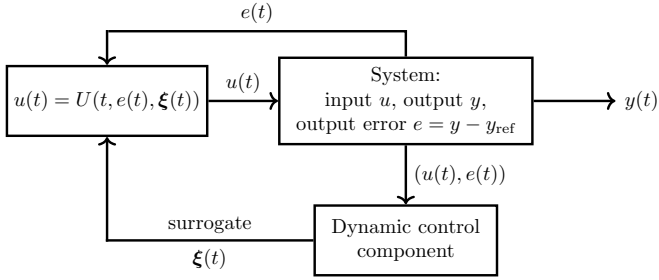

\subsubsection{Motivating example: the double integrator}\label{Sssec:LinSISO-RD2}
\noindent
For purposes of illustration, consider the simplest scalar system of relative degree two:
\begin{equation}\label{dint}
\ddot y(t)=g \, u(t),~~g > 0,~~ (y(0),\dot y(0))=(y^0,v^0)\in\R^2
\end{equation}
and, for  ease of exposition, assume that~$y_{\mathrm{ref}}=0$.   Assume furthermore that the funnel parameter~$\varphi$ is of
class $\Phi\cap\cW^{1,\infty}(\R_{\ge 0},\R)$ with~$\varphi(0)=0$.   As before, let $\alpha\in C^1 ([0,1),[1,\infty))$ be a bijection and define
$\gamma$ as in~\eqref{eq:fcts-FC2}. By Theorem~\ref{Thm:FunCon-Nonl} and Remark~\ref{Rem:N(s)=-s},  we know that the feedback control
\[
u(t)=-\gamma\big(\varphi (t)\dot y(t)+\gamma (\varphi(t)y(t))\big)
\]
ensures that the maximal solution (unique by standard arguments) of \eqref{dint} is global,  bounded and~$y$ evolves in the prescribed performance
funnel~$\cF_\varphi$.   However, this result assumes availability of the ``velocity''~$\dot y(t)$ for feedback.
But what if the velocity is inaccessible?  We highlight two approaches which address this question.

\paragraph{{\bf I. Filtering}}
Augment the double integrator with a ``filter'' driven by $u$:
\[
 \dot\xi(t)=-\xi(t)+u(t),~~
 \xi(0)=0.
\]
Solely for simplicity of exposition, we have adopted the filter initial condition $\xi(0)=0$.
Introducing the variable $z(t):=\dot y(t)-y(t)-g\,\xi(t)$,  the augmented system takes the form
\begin{equation}\label{augsys1}
\left.
\begin{array}{ll}
\dot y(t)=y(t)+z(t)+g\,\xi(t),&y(0)=y^0
\\[1ex]
\dot z(t)=-z(t)-y(t),&z(0)=z^0:=v^0-y^0
\\[1ex]
\dot\xi(t)=-\xi(t)+u(t),&\xi(0)=0.
\end{array}
\right\}
\end{equation}
Viewing the first two of the above equations as an independent system -- with input~$\xi$, output~$y$ and initial data $(y(0),z(0))=(y^0,z^0)$ --  we have
\begin{align*}
&\begin{pmatrix}\dot y(t)\\\dot z(t)\end{pmatrix}=A \begin{pmatrix}y(t)\\z(t)\end{pmatrix}+b\,\xi(t),~~
y(t)
=c \begin{pmatrix}y(t)\\z(t)\end{pmatrix},\\
&A :=\begin{bmatrix}~1&~1\\-1&-1\end{bmatrix},~~
b:=\begin{pmatrix}g\\0\end{pmatrix},~~
c:=\begin{bmatrix}1~~0\end{bmatrix}.
\end{align*}
Observe that $ \Gamma = cb  =g \neq 0$ and
\[
\forall\,\lambda\in {\mathbb{C}}_{\ge 0}:\ \det \begin{bmatrix} A -\lambda I & b\\c & 0\end{bmatrix}=  (1+\lambda)g
 \neq 0.
\]
Thus,  this (independently viewed) system is of relative degree~$r=1$ has asymptotically stable zero dynamics ${\mathcal Z}{\mathcal D}( A,b,c)$ and satisfies (SA1)--(SA3). Therefore,
$( A,b,c) \in {\mathcal L}^{1,1}$. In this illustrative context, the operator~$\fT$ given by~\eqref{eq:opT-lin} has the form
\begin{align*}
\fT&\colon \cC (\R_{\ge 0},\R)\to \cC(\R_{\ge 0},\R),~y\mapsto  y-L(y),\\
 L&\colon y\mapsto \left(t\mapsto \textstyle{\int_0^t} e^{-(t-s)}y(s)\dd s\right).
\end{align*}
Defining $d_0\in\cL^\infty (\R_{\ge 0},\R)$ by $d_0(t):= e^{-t}z^0$ writing  $f\colon (\delta,\zeta,v)\mapsto \delta +\zeta+gv$, we have
\begin{equation}\label{eq:ydot-filter}
\dot y(t)= f(d_0(t),(\fT y)(t),\xi(t)),~~y(0)=y^0.
\end{equation}
By Lemma~\ref{Lem:lincl-Nmr}, $(d_0,f,\fT)\in{\mathcal N}^{1,1}$ and so (in view of by Theorem~\ref{Thm:FunCon-Nonl-Rd1}, Remark~\ref{Rem:N(s)=-s}, and setting
$\gamma \colon v\mapsto -\alpha(v^2)v$ with the special choice $\alpha(s) = 1/(1-s)$) the strategy $\xi(t):= \gamma\big(\varphi(t)y(t)\big)$
ensures that the global solution of~\eqref{eq:ydot-filter} is bounded and~$y$ evolves in the performance
funnel~$\cF_\varphi$. However, this observation is predicated on the premise that~$\xi$ is a variable open to choice.  But this is not
the case:~$\xi$ must lie in the solution set of the filter
\[
\mathcal S
:=
 \setd{\xi\in\cA\cC(\R_{\ge 0},\R) }{ \xi=L(u), ~u\in \cL_{\loc}^\infty (\R_{\ge 0},\R)}.
\]
Writing $\theta\colon t\mapsto  (Lu)(t)-\gamma\big(\varphi (t)y(t)\big)$ and $d\colon t\mapsto d_0(t)+g\theta (t)$, system~\eqref{eq:ydot-filter} may be expressed as
\begin{equation}\label{eq:ydot-filter-2}
\dot y(t)=
f(d(t),(\fT y)(t),\gamma(\varphi(t)y(t))),\quad y(0)=y^0.
\end{equation}
Therefore, if $u\in \cL_{\loc}^\infty(\R_{\ge 0},\R)$  can be chosen such that~$\theta$ (and so, $d$) is bounded, then $(d,f,\fT)\in{\mathcal N}^{1,1}$ and, again invoking Theorem~\ref{Thm:FunCon-Nonl} and Remark~\ref{Rem:N(s)=-s},
it follows that every maximal solution is bounded (and so has domain~$\R_{\ge 0}$)
and~$y$ evolves in the performance funnel~$\cF_\varphi$. Consequently, the issue to be addressed is the design of a feedback strategy,
based only on the available instantaneous information triple~$(t,y(t),\xi(t))$, which ensures boundedness of~$\theta$.
This is precisely the issue resolved in the general setting of~\cite{IlchRyan07}  and summarized in Theorem~\ref{method1} below, which, when
applied to the current  illustrative ``double-integrator" setting, establishes the strategy
\[
u(t)=\gamma_2\big(k(t),\varphi(t)y(t),\xi(t)\big),\quad k(t)=\alpha\big(\varphi^2(t)y^2(t)\big)
\]
where
\begin{align*}
&\gamma_2\colon [1,\infty)\times (-1,1)\times\R\to\R,\\
&(\kappa,\eta,\zeta)\mapsto -\kappa\eta -(\eta^2+\kappa^2)\big(\kappa^2(1+|\zeta|)\big)^2\big(\zeta-\gamma_1(\kappa,\eta)\big),
\end{align*}
achieves the performance objective.  In the context of Figure \ref{Fig:intercon-1}, the surrogate ${\boldsymbol{\xi}}(t)$ is simply the filter state $\xi(t)$.

\paragraph{{\bf II. Pre-compensation}}
Augment the double integrator with a ``pre-compensator" driven by the input~$u$ and output~$y$:
\begin{equation}\label{eq:ddoty-FPC}
\left.
\begin{aligned}
 & \dot \xi_1(t) = \xi_2(t) + (q_1 + p_1 k(t)) (y(t) - \xi_1(t)),\\
 & \dot \xi_2(t) = \tilde g\, u(t) + (q_2 + p_2 k(t)) (y(t) - \xi_1(t)),\\
 & (\xi_1(0),\xi_2(0))=(0,0),
  \\
  & k(t) = \frac{1}{1- \big(\varphi_1 (t)(y(t)-\xi_1(t))\big)^2}
\end{aligned}\right\}
\end{equation}
with $\tilde g, q_i, p_i>0$ (design parameters open to choice) and~$\varphi_1:=2\varphi$.  Analogous to the filtering case, solely for simplicity of exposition, we have adopted the
pre-compensator initial condition $(\xi_1(0),\xi_2(0))=(0,0)$.
The above structure resembles a high-gain observer~\cite{EsfaKhal87,KhalSabe87} with time-varying gain function, however they serve a different purpose. In contrast to high-gain observer theory, the variable~$\xi_2$ is not used to approximate the derivative~$\dot y$ of the output.  Instead,~$\xi_1$ serves as a ``surrogate output'' which is close to the true output~$y$ in the sense that the difference~$y(\cdot )-\xi_1(\cdot )$
evolves within a prescribed performance funnel.
The derivative~$\dot \xi_1$ of the surrogate output is known and so is available for control purposes.  Viewed as a system with input~$u$ and output~$\xi_1$
(with $\dot\xi_1$ also available for feedback),
we seek to apply the funnel controller~\eqref{eq:FC-} in the context of the pre-compensated double integrator given by the conjunction of \eqref{dint} and \eqref{eq:ddoty-FPC}.

To ensure feasibility of the above approach, we need to show that the augmented system~\eqref{dint}-\eqref{eq:ddoty-FPC} satisfies the assumptions of
Theorem~\ref{Thm:FunCon-Nonl}.  To this end, we first proceed to show that the augmented system may be expressed in the form \eqref{eq:nonlSys}.
For simplicity of exposition only, choose $\tilde g=q_1=q_2=p_1=1$
(leaving the design parameter~$p_2 >0$ to be determined). Introducing the variables $z_1 := y-\xi_1$, $z_2 := \dot y - g \xi_2$,
we arrive at a representation of the augmented system with input $u$ and output $\xi_1$:
\begin{equation}\label{aug-sys}
\left.
\begin{array}{ll}
  \ddot \xi_1(t) = (1+ p_2 k(t))z_1(t) + \ddt \big((1 + k(t)) z_1(t)\big) +  u(t),\\[1ex]
  \dot z_1(t) =   z_2(t) - g(1 + k(t)) z_1(t) + \left(g-1\right) \dot \xi_1(t),\\[1ex]
  \dot z_2(t) = - g(1 + p_2 k(t)) z_1(t),~~ k(t) = \frac{1}{1- (\varphi_1(t) z_1(t))^2},
  \\[2mm]
  (\xi_1(0),\dot\xi_1(0),z_1(0),z_2(0))=(0,(1+p_2)y^0,y^0,v^0),
\end{array}
\right\}
\end{equation}
Temporarily replacing~$\dot \xi_1$ by arbitrary $\zeta\in\cC(\R_{\ge 0},\R)$, consider the second and third subsystems of \eqref{aug-sys}
as an initial-value problem with input $\zeta$ and underlying domain $\cD := \{(t,\theta)=(t,\theta_1,\theta_2)\in\R_{\ge 0}\times\R^2|~\varphi_1 (t)|\theta_1| <1\}$.
\begin{equation}\label{eq:pre-comp-ivp}
\left.
\begin{array}{l}
\dot z(t)=Qz(t) -\left(\frac{g\ z_1(t)}{1-(\varphi_1 (t)z_1(t))^2}\right){\mathbf p} + (1-g)\begin{pmatrix}1\\1\end{pmatrix}z_1(t)\\
\qquad\quad + (g-1)\begin{pmatrix}1\\0\end{pmatrix}\zeta(t),\\[3ex]
z(t)=\begin{pmatrix}z_1(t)\\z_2(t)\end{pmatrix},~z(0)=z^0=\begin{pmatrix}z_1^0\\z_2^0\end{pmatrix},\\
Q=\begin{bmatrix} -1 & 1\\-1 & 0\end{bmatrix},
~~{\mathbf p}=\begin{pmatrix}1\\p_2\end{pmatrix}
\end{array}
\right\}
\end{equation}
By the standard theory of differential equations this initial-value problem
has, for all  $(z^0,\zeta)\in\R^2\times \cC(\R_{\ge 0},\R)$,  a~unique maximal solution
$z\colon [0,\omega)\to\R^2$ and $\text{graph}(z)\subset \cD$: we write $z(\cdot )=\varrho (\cdot,z^0,\zeta)$.   Moreover, noting that $Q$ is Hurwitz, a straightforward, if tedious, calculation establishes that $\omega=\infty$ (and so $|\varphi_1(t)z_1(t)|<1$ for all $t\ge 0$).    Therefore, we may define the following causal
operator (more precisely, the generic member of a family $\{{\mathbf T}_{z^0}| ~z^0\in\R^2\}$ of operators parameterized by $z^0$: for notational simplicity we
suppress the dependence on $z^0$)
\begin{equation}\label{causal-op}
\left.
\begin{array}{ll}
   {\mathbf T}\colon &\cC (\R_{\ge 0},\R)\to \cL^\infty_{\text{loc}}(\R_{\ge 0},\R^4),\ \zeta\mapsto (z,k,{\zeta}),
    \\[1ex]
   &z(\cdot )=\varrho(\cdot,z^0,\zeta)=(z_1(\cdot),z_2(\cdot)),\\[1ex]
    &k\colon t\mapsto 1/(1-(\varphi_1(t)z_1(t))^2).
\end{array}\right\}
\end{equation}
Defining $f\in \cC(\R^2\times \R^4\times \R,\R)$ by
\begin{align*}
f\colon &(d,\eta,u)=\big((d_1,d_2),(\eta_1,\ldots,\eta_4),u\big)\\
&\mapsto (1+p_2 \eta_3) \eta_1 + 2 \eta_3^2 d_1 \eta_1^2 \big(d_2 \eta_1 + \\
&\qquad d_1 (\eta_2 - g(1+\eta_3)\eta_1+(g-1)\eta_4)\big)\\
    & \qquad + (1+\eta_3) (\eta_2 - g(1+\eta_3)\eta_1+(g-1)\eta_4) + u,
\end{align*}
it is readily verified that \eqref{aug-sys} may be expressed in the form of the functional differential equation
\begin{multline*}
\ddot \xi_1(t)=f(d(t), {\mathbf T}(\dot\xi_1)(t),u(t)),~~ d(t)=(\varphi_1(t),\dot\varphi_1(t)),\\
(\xi(0),\dot \xi(0))=\big(0,\tfrac 43 y^0\big),
\end{multline*}
where ${\mathbf T}$ is the operator, associated with the initial data $z^0=(y^0,v^0)$, given by~\eqref{causal-op}.
Moreover, both $\xi_1(t)$ and its derivative $\dot\xi_1(t)$ are available for feedback.  If it can be shown that the triple~$(d,f,\fT)$ is of class~$\cN^{1,2}$
(and so is amenable to funnel control), then, applying Theorem \ref{Thm:FunCon-Nonl}  in this context and adopting the performance funnel ${\mathcal F}_{\varphi_1}$ with $\varphi_1:=2\varphi$ (recall that ${\mathcal F}_\varphi$
is the performance funnel stipulated {\em ab initio} for the double integrator plant), the control
\[
u(t)=-\gamma\big(\varphi_1 (t)\dot \xi_1(t)+\gamma (\varphi_1(t)\xi_1(t))\big)
\]
ensures that, for some $\varepsilon_1 \in (0,1)$, ~$\varphi_1(t)|\xi_1(t)| \le \eps_1$ for all~$t\geq 0$.   We also know that $\varphi_1(t) |y(t)-\xi_1(t)|=\varphi_1|z_1(t)| < 1$ for all
$t\geq 0$, and so, setting $\varepsilon := \half (1+\varepsilon_1)$, we have
\begin{multline*}
    \varphi(t) |y(t)| =\half\varphi_1|y(t)|\\
    \le \half \left( \varphi_1(t) |y(t)-\xi_1(t)| + \varphi_1(t)|\xi_1(t)|\right)
    <  \eps < 1.
    \end{multline*}
Therefore, the performance objective is achieved by the dynamic component
\begin{align*}
  \dot \xi_1(t) &= \xi_2(t) + (1 +  k(t)) (y(t) - \xi_1(t)),\\
  \dot \xi_2(t) &= u(t) + (1 +{\third} k(t)) (y(t) - \xi_1(t)), \\
  k(t) &= \frac{1}{1- \varphi_1(t)^2 (y(t)-\xi_1(t))^2},\\
 & (\xi_1(0),\xi_2(0)=(0,0),
  \end{align*}
  in conjunction with the feedback
  \begin{multline*}
  u(t)= -\gamma\big(\varphi_1 (t)\big(\xi_2(t) + (1+ k(t)) (y(t) - \xi_1(t))\big)\\
  +\gamma (\varphi_1(t)\xi_1(t))\big)
\end{multline*}
which requires only the available instantaneous information quadruple $(t,y(t),\xi_1(t),\xi_2(t))$.  In the context of Fig.\,\ref{Fig:intercon-1},
we have $\boldsymbol{\xi}(\cdot)=(\xi_1(\cdot),\xi_2(\cdot))$.

What remains at issue is the question: does the triple~$(d,f,\fT)$ belong to the class~$\cN^{1,2}$?   This question is answered in the affirmative if it can be
shown that the operator ${\mathbf T}$ is of class ${\mathbb T}_0^{1,4}$.
This is essentially the issue resolved in the general setting of \cite{Lanz22}  and summarized in Theorem \ref{Thm:LinSys-MP} below.

\subsubsection{Funnel control with filtering}\label{Ssec:backstep}
\noindent
Having highlighted their main ingredients via the simplest of relative-degree-two systems, we now describe the above two methodologies in the broad context of systems~\eqref{eq:Sysmeth}--\eqref{eq:Sysmethic}. First, we consider funnel control with filtering.
Let $N\in{\cC}^r(\R_{\ge 0},\R)$ be surjective (for example, $N\colon \kappa\mapsto \kappa\sin \kappa$ suffices) and
let $\alpha\colon [0,1)\to [1,\infty)$ be a $r$-times continuously differentiable bijection such that $\alpha^\prime = a\circ\alpha$ for some function $a\colon [1,\infty)\to \R_{\ge 0}$
(for example, $\alpha\colon s\mapsto (1-s)^{-\beta}$, $\beta >0$, suffices).   Again, let $\cB$ denote the open unit ball centred at $0$ in $\R^m$.  Define
\begin{align*}
\gamma&\colon \cB\to\R^m,~v\mapsto (N\circ\alpha )\big(\|v\|^2\big)v,\\
 \gamma_1&\colon [1,\infty)\times\cB\to\R^m~, (\kappa,v)\mapsto N(\kappa)v,
\end{align*}
and projections
\[
\pi_i\colon \R^{(r-1)m}\to\R^{im},~~\xi =(\xi_1,\ldots,\xi_{r-1})\mapsto (\xi_1,\ldots,\xi_i)
\]
for $i=1,\ldots,r-1$. Fix $\mu >0$ (a design parameter) and
define $\gamma_i\colon[1,\infty)\times  \cB\times\R^{(i-1)m}\to\R^m$, $i=2,\ldots,r$, by the recursion
\begin{multline}\label{meth1-gammai}
\gamma_i(\kappa,v,\pi_{i-1}\xi):= \gamma_{i-1}(\kappa,v,\pi_{i-2}\xi)
\\
-\Big(a(\kappa)(1+\|\pi_{i-1}\xi\|)\|(D\gamma_{i-1})(\kappa,v,\pi_{i-2}\xi)\|\Big)^2
\\\times (\mu^{2-i}\xi_{i-1}-\gamma_{i-1}(\kappa,v,\pi_{i-2}\xi))
\end{multline}
wherein $D$ denotes the differentiation operator, $D\gamma_{i-1}$ being the Jacobian of $\gamma_{i-1}$ with
\begin{multline*}
\|D\gamma_{i-1}(\cdot,\cdot,\cdot)\|^2=\|\partial_1 \gamma_{i-1}(\cdot,\cdot,\cdot)\|^2\\
+\|\partial_2\gamma_{i-1}(\cdot,\cdot,\cdot)\|^2+\|\partial_3\gamma_{i-1}(\cdot,\cdot,\cdot)\|^2,
\end{multline*}
where $\partial_j$  denotes differentiation with respect to the $j$-th argument.
We adopt the convention $(\kappa,v,\pi_0\xi)\equiv (\kappa,v)$, in other words, the symbol $\pi_0$ is vacuous.  In particular, we record that
$\|D\gamma_1(\kappa,v,\pi_0\xi)\|^2 = N^\prime (\kappa)^2\|v\|^2 + N(\kappa)^2$.

Augment the system~\eqref{eq:Sysmeth} by a linear input ``filter'' of the form
\begin{equation}\label{meth1-filter}
\begin{aligned}
\dot{\xi}_i(t)&=-\mu\xi_i(t)+\xi_{i+1}(t),\quad i=1,\dots,r-2,\\
\dot{\xi}_{r-1}(t)&=-\mu\xi_{r-1}(t)+u(t),
\end{aligned}
\end{equation}
with $\xi_i(t)\in \R^m$ and arbitrary initial data $\xi_i(0)=\xi_i^0\in\R^m$, $i=1,\ldots,r-1$. The augmented system takes the form
\begin{multline}\label{meth1-sys}
\begin{pmatrix}\dot {\mathbf y}(t)\\\dot\xi(t)\end{pmatrix}=\begin{bmatrix}A&0\\0&F\end{bmatrix}\begin{pmatrix}{\mathbf y}(t)\\\xi(t)\end{pmatrix} +\begin{bmatrix} B\\0\end{bmatrix}
\Gamma^{-1}f(d(t),
{\mathbf T}({\mathbf y})(t))\\
+\begin{bmatrix}B\\G\end{bmatrix}u(t),
\end{multline}
with output $\begin{pmatrix}C{\mathbf y}(t)\\\xi(t)\end{pmatrix}$, where
\[
    C\!=\!\big[I, 0,\cdots, 0\big],\ \ {\mathbf y}(t)\!=\!\begin{pmatrix}y(t)\\\dot y(t)\\\vdots\\y^{(r-1)} (t)\end{pmatrix},\ \ \xi(t)\!=\!\begin{pmatrix}\xi_1(t)\\\xi_2(t)\\\vdots\\\xi_{r-1}(t)\end{pmatrix},
\]
\begin{multline*}
A=\begin{bmatrix}0&I&\cdots&0\\\vdots&\vdots&\ddots&\vdots\\0&0&\cdots&I\\R_1&R_2&\cdots&R_r\end{bmatrix},\quad B=\begin{bmatrix}0\\\vdots\\0\\\Gamma\end{bmatrix},\\
F=\begin{bmatrix}-\mu I&I&\cdots &0\\\vdots&\vdots&\ddots&\vdots\\
0&0&\cdots&I\\0&0&\cdots &-\mu I\end{bmatrix},~~{\text{and}}~~G=\begin{bmatrix}0\\\vdots\\0\\I\end{bmatrix}.
\end{multline*}
Let $y_{\mathrm{ref}}\in \cW^{r,\infty}(\R_{\ge 0},\R^m)$ be arbitrary and write $e(\cdot)=y(\cdot)-y_{\mathrm{ref}}(\cdot)$.
We introduce the control
\begin{equation}\label{meth1-con}
u(t)=\gamma_r(k(t),\varphi(t)e(t),\xi(t)),~~
k(t)=\alpha\big(\varphi^2(t)\|e(t)\|^2\big),
\end{equation}
which will ensure  attainment of the performance objectives of boundedness of all signals and evolution of the tracking error in the performance funnel.

 Note that,  if we set  $r=1$ in \eqref{meth1-con}, then
\begin{multline*}
u(t)=\gamma_1(k(t),\varphi(t)e(t))\\
=\gamma (\varphi(t)e(t))=(N\circ\alpha)(\varphi(t)^2\|e(t)\|^2)\varphi(t)e(t)
\end{multline*}
and so, as is to be expected, we recover the (non-dynamic) controller \eqref{FcRd1}.   In the case of relative degree $r=2$
and $\mu=1$, we have the dynamic controller
\begin{align*}
&\dot \xi(t) = - \xi(t) + u(t),\\
&u(t)=\gamma(\varphi(t)e(t))-\big(a(k(t))\,(1+\|\xi(t)\|)\big)^2 \\
&\cdot\Big(\big(N^\prime (k(t))\varphi(t)\|e(t)\|\big)^2\! +\!N(k(t))^2\Big)\big(\xi(t)\!-\!\gamma(\varphi(t)e(t))\big),
\end{align*}
with $k(t):=\alpha(\varphi^2(t)\|e(t)\|^2)$.

In the general case $r\geq 2$, the efficacy of the control~\eqref{meth1-con} was established in~\cite{IlchRyan07}. We restate this result here, tailored to the present framework.

\begin{thm}\label{method1}
Consider the initial-value problem~\eqref{eq:Sysmeth}--\eqref{eq:Sysmethic}. Choose $(\alpha,N,\varphi)$ such that $\varphi\in\Phi$, $N\in\cC^r(\R_{\ge 0},\R)$ is surjective,   and
$\alpha\in \cC^r([0,1),[1,\infty))$ is bijective with $\alpha^\prime = a\circ \alpha$ for some function $a\colon [1,\infty)\to \R_{\ge 0}$. Let $y_{\mathrm{ref}}\in \cW^{r,\infty}(\R_{\ge 0},\R^m)$ be such that $\varphi(0)\|y(0)-y_{\rm ref}(0)\| <1$
(trivially satisfied if $\varphi(0)=0$).
Then the control~\eqref{meth1-con} applied to the augmented system~\eqref{meth1-sys}, with initial data given by~\eqref{eq:Sysmethic} and
the initial condition $\xi(s)=\xi^0\in\R^{(r-1)m}$ for all $s\in [-h,0]$,  yields an initial-value problem which has a solution (in the sense of Carath\'{e}odory), every solution can be maximally extended and every maximal solution
$({\mathbf y},\xi)\colon \left[-h,\omega\right)\rightarrow \R^{rm} \times \R^{(r-1)m}$ has the properties:
\begin{enumerate}[(i)]
\item
$\omega=\infty$ (global existence);
\item
$u\in\cL^\infty (\R_{\ge0},\R^m)$, $\xi\in\cL^\infty (\R_{\ge0},\R^{(r-1)m})$,  $y\in\cW^{r,\infty}([-h,\infty),\R^m)$ where $y=C{\mathbf y}$;
\item
the tracking error $e=y-y_{\mathrm{ref}}\colon \R_{\ge 0}\to \R^m$ evolves in the funnel~$\mathcal{F}_{\varphi}$  and there exists $\varepsilon \in (0,1)$ such that
$\varphi(t)\|e(t)\|\leq \varepsilon$ for all $t\geq 0$.
\end{enumerate}
\end{thm}

\begin{rem}\label{rem:disc-filter}
The recursive procedure in \eqref{meth1-gammai}~-- generating the feedback function $\gamma_r$ in the control \eqref{meth1-con}~-- is a form of backward induction
structurally reminiscent of the “back-stepping” procedure developed in the 1990s by Kotokovic and others~\cite{Koko92,LozaBrog92}
in a different context of feedback stabilization of nonlinear systems.
Such procedures risk falling victim to the ``curse of dimensionality", a phrase coined by Bellman \cite[Preface]{Bell57} in the development of Dynamic Programming,
and indeed \eqref{meth1-gammai} is not exempt from this risk.  The
``curse'' refers to adverse features  that arise with increasing dimension.   In the present setting, dimension equates to relative degree $r$.
For example, set $\alpha\colon s\mapsto (1-s)^{-1}$ and consider the case wherein~$\Gamma$ is known to be positive definite (and so $N\colon \kappa\mapsto -\kappa$ may be chosen).
As before, write $k(\cdot)=\alpha(\varphi^2(\cdot )\|e(\cdot)\|^2)$, which, if $r=1$,  enters as a simple multiplier or gain in the feedback control, viz.
$u(t)=-k(t)\varphi(t)e(t)$.   However, for $r\ge 2$, the recursive procedure in \eqref{meth1-gammai} generates multipliers (embedded in the feedback control) of the form $k(t)^p$, the exponent $p$ of which may become impractically
large even for moderately low values of~$r$.   Funnel control
with pre-compensation (detailed below) seeks to circumvent this drawback, but not without paying a cost:
as shall be seen, the dynamic order of the pre-compensator is $r(r-1)$, whereas the dynamic order of the filter is $r-1$.
\end{rem}
\subsubsection{Funnel control with pre-compensation}\label{Ssec:precomp}
\noindent
In this section we describe  a recent approach to funnel control with non-derivative feedback which avoids the  backstepping procedure. A straightforward idea to do this was the use of a high-gain observer; see the classical works~\cite{EsfaKhal87,KhalSabe87,SabeSann90,Torn88} and the survey~\cite{KhalPral14}. One advantage of high-gain observers is that they can be used to estimate the system states without knowing the exact parameters (in contrast to observer synthesis, see e.g.~\cite{ChoRaja97,EmelKoro04} and the references therein); only some structural assumptions, such as a known relative degree, are necessary. Furthermore, they are robust with respect to input noise. The drawback is that in most cases it is not known a priori how large the high-gain parameter~$k$ in the observer must be chosen and appropriate values must be identified by offline simulations. If~$k$ is chosen unnecessarily large, the sensitivity to measurement noise increases dramatically. High-gain observers with time-varying gain functions~$k(\cdot)$ and corresponding adaptation laws are proposed in~\cite{BullIlch98,SanfPral11}. However, they are not able to influence the transient behaviour of the observation error.

The combination of the adaptive high-gain observer from~\cite{BullIlch98} with a $\lambda$-tracker has been successfully developed by
\textit{Bullinger and Allg\"ower} (2005)~\cite{BullAllg05}. In the recent  paper by \textit{Chowdhury and Khalil} (2019)~\cite{ChowKhal19} the funnel controller from~\cite{IlchRyan02b} is combined with a high-gain observer (for a similar result on prescribed performance control, discussed in Section~\ref{Sssec:PPC}, see~\cite{DimaBech20}). For
single-input, single-output  systems with higher relative degree a virtual (weighted) output is defined such that the system has relative degree one with respect to this virtual output. Then funnel control is feasible and it is shown that (ignoring the additional use of a high-gain observer) for sufficiently small weighting parameter in the virtual output, the original tracking error evolves in a prescribed performance funnel. However, tuning of the weighting parameter has to be done a posteriori and hence depends on the system parameters and the chosen reference trajectory. Therefore, this approach is not model-free like standard funnel control approaches and the controller is not robust, since small perturbations of the reference signal may cause the tracking error to leave the performance funnel.

\noindent
\textit{Berger and Reis} (2018)~\cite{BergReis18a} presented a controller which uses only dynamic output feedback (and no derivatives of the output), avoids the backstepping procedure, and guarantees evolution of the tracking error within a prescribed performance funnel for the class of linear systems with relative degree two.
 This controller is based on the combination of the relative degree two funnel controller~\eqref{eq:PDfunnel} with a funnel pre-compensator~\eqref{eq:ddoty-FPC}. The funnel pre-compensator for systems with arbitrary degree was developed in~\cite{BergReis18b}. Combinations of the funnel pre-compensator with the funnel controller~\eqref{eq:fun-con} are discussed in~\cite{BergOtto19} with applications to underactuated multibody systems.
The general funnel pre-compensator, with $\R^{rm}$-valued state $(\xi_1(\cdot),\cdots,\xi_r(\cdot))$, is defined as follows:
\begin{equation}\label{eq:fun-precomp}
\boxed{
\begin{array}{l}
      \dot \xi_i(t) = \xi_{i+1}(t) + \big(q_i + p_i k(t)\big)(y(t) - \xi_1(t)),\\[1ex]
      \qquad\qquad\qquad\qquad\qquad\qquad i=1,\ldots,r-1,\\[1ex]
      \dot \xi_{r-1}(t) = \xi_r(t) +\big(q_{r-1} + p_{r-1} k(t)\big) (y(t) - \xi_1(t)),\hspace*{-3ex}
      \\[1ex]
      \dot \xi_r(t) =  \big(q_r + p_r k(t)\big)(y(t) - \xi_1(t)) +\;\widetilde \Gamma u(t),
      \\[1ex]
      (\xi_1(0),\ldots,\xi_r(0))=(\xi_1^0,\ldots,\xi_r^0)\in\R^{m}\times\cdots\times\R^m,\hspace*{-3ex}
      \\[3mm]
      k(t) = {\displaystyle\frac{1}{1-\varphi(t)^2 \|y(t)-\xi_1(t)\|^2}},
    \end{array}
    \hspace*{-1ex}
}
\end{equation}
with design parameters $p_i> 0$, $q_i> 0$,  $\widetilde \Gamma\in\Gl_m(\R)$ and funnel function $\varphi\in\Phi$.
We write
\[
\mathbf p = \begin{pmatrix} p_1\\\vdots\\ p_r\end{pmatrix}\quad\text{and}\quad \mathbf q = \begin{pmatrix} q_1\\\vdots\\ q_r\end{pmatrix}.
\]
The adaptation scheme for $k(t)$ in~\eqref{eq:fun-precomp} is non-dynamic and non-monotone, and it guarantees
prescribed transient behaviour of the difference $y(\cdot)-\xi_1(\cdot)$, which we refer to as the {\it compensator error}. Another advantage of the funnel pre-compensator~\eqref{eq:fun-precomp} is that no higher powers of the gain function~$k$ are involved in~\eqref{eq:fun-precomp} (cf.\ the discussion in Remark~\ref{rem:disc-filter}).
 Moreover, the pre-compensator obviates the need for estimates of the underlying model as required in the context of high-gain observers, see~\cite{AstoMarc15,Khal16}.

In contrast to other approaches, the signals~$u$ and~$y$ given to the funnel pre-compensator~\eqref{eq:fun-precomp} are not necessarily the input and output corresponding to some system or plant. We only assume that they are signals belonging to the following set parameterized by $r\in\N$:
\[
\cP_r\!:=\!
\left\{\!\!\!\begin{array}{r}
(u,y)\!\in\!  \cL^\infty_{\loc}(\R_{\ge 0},\R^m)\\
\times \cW^{r,\infty}_{\loc}(\R_{\ge 0},\R^m)\end{array}\!\!\!
\left|
\!\! \begin{array}{l}
y^{(r-1)}\in \cL^\infty(\R_{\ge 0},\R^m),
\\
 y^{(r)}\!-\!\Gamma u\!\in\!\cL^\infty(\R_{\ge 0},\R^m),\\
\Gamma\in \Gl_m(\R)
\end{array}
\right.\!\!\!\!
\right\}
\]
The vector
$\mathbf q =(q_1, \ldots,q_r)^\top$
 is chosen such that the matrix
\begin{equation}\label{eq:Q}
    Q = \begin{bmatrix} -q_1 & 1 &\ldots& 0\\ \vdots &\vdots& \ddots & \vdots\\ -q_{r-1} &0&\ldots& 1\\ -q_r &0&\ldots& 0\end{bmatrix}\in\R^{r\times r}
\end{equation}
(with characteristic polynomial
$s^r + q_r s^{r-1}+\cdots+q_1$) is Hurwitz, i.e., $\sigma(Q)\subset\mathbb{C}_{< 0}$.  Let $R=R^\top\succ0$ and
\[
    P = \begin{bmatrix} P_1 & P_2\\ P_2^\top & P_4\end{bmatrix},\ \ P_1\!\in\!\R,\ P_2\!\in\!\R^{1\times (r-1)},\ P_4\!\in\!\R^{(r-1)\times (r-1)}
\]
be such that
\begin{equation}\label{eq:lyap}
    Q^\top P + PQ + R = 0,\quad P=P^\top\succ0.
\end{equation}
The vector $\mathbf p$ is uniquely determined by $\mathbf q$ and $R$ via the following construction:
\begin{equation}\label{eq:rel-piqi}
   \mathbf p= \begin{pmatrix} p_1\\ \vdots\\ p_r\end{pmatrix} := P^{-1} \begin{pmatrix} P_1 - P_2 P_4^{-1} P_2^\top \\ 0\\ \vdots\\ 0\end{pmatrix} = \begin{pmatrix} 1\\ - P_4^{-1} P_2^\top \end{pmatrix}.
\end{equation}

The pre-compensator~\eqref{eq:fun-precomp} is a nonlinear and time-varying system, yet it is simple in its structure and its dimension depends only on the ``relative degree''~$r$ given by~$\cP_r$. The set~$\cP_r$ of signals~$(u,y)$ ensures error evolution within the funnel.
For a schematic of the construction of the funnel pre-compensator~\eqref{eq:fun-precomp} see also Fig.~\ref{Fig:construction}.

\captionsetup[subfloat]{labelformat=empty}
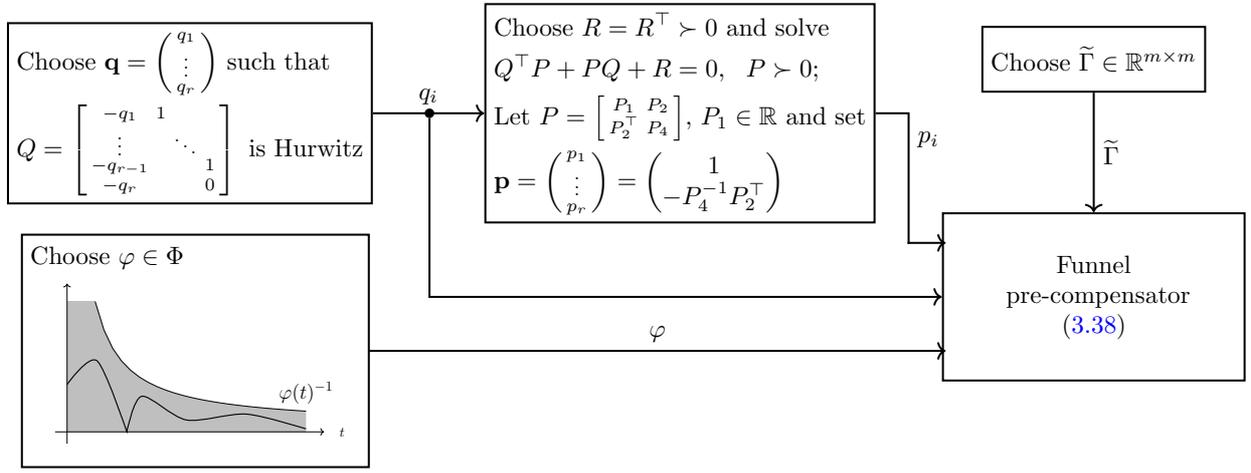
\begin{figure*}[h!t]
  \centering
  \resizebox{0.9\textwidth}{!}{
\begin{tikzpicture}[thick,node distance = 12ex, box/.style={fill=white,rectangle, draw=black}, blackdot/.style={inner sep = 0, minimum size=3pt,shape=circle,fill,draw=black},plus/.style={fill=white,circle,inner sep = 0, minimum size=5pt,thick,draw},metabox/.style={inner sep = 3ex,rectangle,draw,dotted,fill=gray!20!white}]

  \node (box1)	[box, minimum size=6ex]{$\begin{aligned}
    & \text{Choose $\mathbf q =\begin{smallpmatrix}q_1\\\vdots\\q_r\end{smallpmatrix}$ such that}\\ & Q = \begin{smallbmatrix} -q_1 & 1 && \\ \vdots && \ddots & \\ -q_{r-1} &&& 1\\ -q_r &&& 0\end{smallbmatrix}\ \text{is Hurwitz} \end{aligned}$};	

  \node (box2) [box, right of = box1, minimum size=6ex, xshift=33ex]{$\begin{aligned} & \text{Choose $R=R^\top\succ0$ and solve}\\ & Q^\top P + PQ + R = 0,\ \  P\succ0;\\ & \text{Let $P = \begin{smallbmatrix} P_1 & P_2\\ P_2^\top & P_4\end{smallbmatrix}$, $P_1\in\R$ and set}\\ & \mathbf p =\begin{smallpmatrix} p_1\\ \vdots\\ p_r\end{smallpmatrix} =  \begin{pmatrix} 1\\ - P_4^{-1} P_2^\top \end{pmatrix}\end{aligned}$};

 \node (box3) [box, below of = box1, minimum size=6ex, yshift=-10ex, xshift=0.5ex]{$\begin{aligned} &\text{Choose $\varphi\in\Phi$}\\[-5mm] &  \text{\resizebox{0.25\textwidth}{!}{%
	\begin{tikzpicture}[domain=0.001:4,scale=1.5]
						\fill[color=black!25,domain=0.47:4] (0,0)-- plot (\x,{min(2.2,1/\x+2*exp(-3))})--(4,0)-- (0,0);
            \fill[color=black!25,] (0,0) -- (0,2.2) -- (0.47,2.2) -- (0.47,0) -- (0,0);
			\draw[->] (-0.2,0) -- (4.3,0) node[right] {$t$};
			\draw[->] (0,-0.2) -- (0,2.5) node[above] {};
			\draw[color=black,domain=0.47:4] plot (\x,{min(2.2,1/\x+2*exp(-3))}) node[above] {\Large$\varphi(t)^{-1}$};
			\draw[smooth,color=black,style=thick] (0,0.8) node[left] {}
			plot coordinates{(0,0.8)(0.5,1.2)(1,0)}--
			plot coordinates{(1,0)(1.25,0.6)(2,0.2)(3,0.3)(4,0.05)} ;
	\end{tikzpicture}}}\end{aligned}$};

  \node (box4) [box, below of = box2, minimum size=6ex, yshift=-5ex, xshift=38ex]{
  $\begin{array}{c}~\\\text{Funnel}\\\quad~~\text{pre-compensator}\quad~\\\eqref{eq:fun-precomp}\\~\end{array}$
  };

  \node (box5) [box, above of = box4, minimum size=6ex, yshift=10ex]{Choose $\widetilde \Gamma\in\R^{m\times m}$};

  \node (fork)  [blackdot,right of = box1, xshift = 10ex] {};
  \node (p1)  [left of = box4, minimum size=0pt, inner sep = 0pt, yshift = 5ex, xshift=-5ex] {};
  \node (p2)  [left of = box4, minimum size=0pt, inner sep = 0pt, yshift = 5ex, xshift=-1.5ex] {};
  \node (p3)  [left of = box4, minimum size=0pt, inner sep = 0pt, yshift = -5ex, xshift=-1.5ex] {};

  \draw[->](box1) -- (box2) node[midway,above]{$q_i$};
  \draw[->](fork) |- (box4) node[midway,above]{};
  \draw[->](box3) -- (p3) node[midway,above]{$\varphi$};

  \draw[-](box2) -| (p1.south) node[pos=0.6,right]{$p_i$};
  \draw[->](p1.west) -- (p2) node[midway,right]{};

  \draw[->](box5) -- (box4) node[midway,right]{$\widetilde \Gamma$};

  \end{tikzpicture}
    }
\caption{Construction of the funnel pre-compensator~\eqref{eq:fun-precomp} depending on its design parameters; taken from~\cite{BergReis18b}.}
\label{Fig:construction}
\end{figure*}

It is shown in~\cite{BergReis18b} that for signals $(u,y)\in\cP_r$ with $r\ge 2$, the funnel pre-compensator~\eqref{eq:fun-precomp} has
a unique maximal solution $(\xi_1,\ldots,\xi_r)$: moreover, the (absolutely continuous) solution is bounded (and so has interval of existence $\R_{\ge 0}$) and
\begin{equation*}\label{eq:Intro-fun}
    \exists\,\varepsilon>0\ \forall\, t>0:\ \|y(t) - \xi_1(t)\| < \varphi(t)^{-1} - \varepsilon.
\end{equation*}
 Thus, with each admissible quadruple $(\mathbf p,\mathbf q,\widetilde\Gamma,\varphi)$, we may associate a funnel pre-compensator operator
$\text{FP} (\mathbf p,\mathbf q,\widetilde\Gamma,\varphi)\colon\cP_r\to\cL^\infty (\R_{\ge 0},\R^{m})$,
$(u,y)\mapsto \xi_1$
(or, more precisely, a family of such operators parameterized by the initial data: for notational simplicity, we suppress the dependency on this
arbitrary data.)

While the funnel pre-compensator is able to achieve prescribed transient behaviour of the compensator error $e_1=y-\xi_1$, we like to stress that no transient behaviour can be prescribed for the errors $e_i =y^{(i-1)}-\xi_i$ for $i=2,\ldots,r-1$ and $e_r=\widetilde \Gamma  \Gamma^{-1} y^{(r-1)}- \xi_r$, since $\dot y,\ldots, y^{(r-1)}$ are not known. Therefore, the variables $\xi_2,\ldots, \xi_r$ from the funnel pre-compensator cannot be viewed as estimates for the derivatives $\dot y,\ldots, y^{(r-1)}$.
The following construction seeks to circumvent this shortfall. Choose admissible $(\mathbf p^i,\mathbf q^i,\widetilde\Gamma,\varphi_i)$ with $\widetilde\Gamma\in\Gl_m(\R) $ and $\mathbf{p}^i,\mathbf{q}^i\in\R^r$, $\varphi_i\in \Phi_{r}$
(defined as in \eqref{eq:Phir}), $i=1,\ldots,r-1$.
   Consider the  \textit{cascade of $(r-1)$ funnel pre-compensators}
\begin{multline*}
 \text{FP}_{r-1} \circ \text{FP}_{r-2} \circ \ldots \circ \text{FP}_{1}\colon \cP_r\to\cL^\infty(\R_{\ge 0},\R^{m}),\\
 (u,y)\mapsto \xi_{r-1,1}=: z,
\end{multline*}
where $\text{FP}_i:=\text{FP}(\mathbf p^i,\mathbf q^i,\widetilde\Gamma,\varphi_i)$, with implicitly-associated initial data $\bxi_i^0:=(\xi_{i,1}^0,\ldots,\xi_{i,r}^0)\in\R^{m}\times\cdots\times\R^m$.
Thus, for $(u,y)\in\cP_r$ and notationally identifying $\xi_{0,1}$ with $y$, the $\R^{rm}$-valued function
 $\bxi_i:=(\xi_{i,1},\ldots,\xi_{i,r})$, where $\xi_{i,1}=\text{FP}_i(u,\xi_{i-1,1})$ and $i=1,\ldots,r-1$,
 is given by
 \noindent
\begin{equation}\label{eq:casc-fun-obs-i}
\begin{array}{l}
\dot\bxi_i(t)=\tilde A\bxi_i(t)+\big(\big(\mathbf q^i+k_i(t)\mathbf p^i\big)\otimes I_m\big) \big(\xi_{i-1,1}(t)-\xi_{i,1}(t)\big)\\[1ex]
\qquad\quad+\tilde B u(t),
\\[1ex]
\bxi_i(0)=\bxi_i^0,
\\[2ex]
k_i(t)= {\displaystyle\frac{1}{1-\varphi_i(t)^2 \|\xi_{i-1,1}(t)-\xi_{i,1}(t)\|^2}},
\end{array}
\end{equation}
where $\otimes$ is the Kronecker product, with
\[
      \tilde A :=\begin{bmatrix} 0&I_m&0 &\cdots&0\\0&0&I_m&\cdots&0\\
      \vdots&\vdots&&\ddots&\vdots
      \\0&0&0&\cdots&I_m
      \\0&0&0&\cdots&0\end{bmatrix},\quad \tilde B:=\begin{bmatrix} 0\\0\\\vdots\\0\\\widetilde\Gamma\end{bmatrix},
\]
and the cascade output is given by $z(t)=\xi_{r-1,1}(t)$.
The situation is illustrated in Fig.~\ref{Fig:cascade}.  The dynamic order of the cascade is $r(r-1)$.

\captionsetup[subfloat]{labelformat=empty}
\begin{figure}[h!t]
  \centering
  \resizebox{\columnwidth}{!}{
\begin{tikzpicture}[thick,node distance = 12ex, box/.style={fill=white,rectangle, draw=black}, blackdot/.style={inner sep = 0, minimum size=3pt,shape=circle,fill,draw=black},plus/.style={fill=white,circle,inner sep = 0, minimum size=5pt,thick,draw},metabox/.style={inner sep = 3ex,rectangle,draw,dotted,fill=gray!20!white}]

  \node (box1)		[box, minimum size=6ex]{\ $(u,y)\in\cP_r$\ \ };
  \node (box2) [box, right of = box1, minimum size=6ex, xshift=11ex]{$\begin{array}{c} \text{FP}_1\end{array}$};
  \node (box3) [box, right of = box2, minimum size=6ex, xshift=7ex]{$\begin{array}{c} \text{FP}_2\end{array}$};
  \node (P1)  [right of = box3, xshift = 3ex, minimum size=0pt, inner sep = 0pt] {};
  \node (P2)  [right of = P1, xshift = -6ex, minimum size=0pt, inner sep = 0pt] {};
  \node (L1)  [right of = P2, xshift = -2ex, minimum size=0pt, inner sep = 0pt] {};
  \node (L2)  [right of = P2, xshift = -2ex, yshift = -6ex, minimum size=0pt, inner sep = 0pt] {};
  \node (box4) [box, right of = L2, yshift = 3ex, minimum size=9ex, xshift=-6.8ex]{$\begin{array}{c} \text{FP}_{r-1}\end{array}$};
  \node (R1)  [right of = box4, xshift = 5ex, minimum size=0pt, inner sep = 0pt] {};
  \node (P3)  [left of = L2, xshift = -4ex, minimum size=0pt, inner sep = 0pt] {};
  \node (P4)  [right of = P3, xshift = -6ex, minimum size=0pt, inner sep = 0pt] {};
  \node (P5)  [blackdot, below of = box2, xshift = 0ex, yshift=6ex] {};
  \node (P6)  [blackdot, below of = box3, xshift = 0ex, yshift=6ex] {};

  \draw (box1) |- (P5) node[pos=0.8,above]{$u$};
  \draw[->] (P5) -- (P3);
  \draw[dashed] (P3) -- (P4);
  \draw[->] (P4) -- (L2);

  \draw[->] (P5) -- (box2);
  \draw[->] (P6) -- (box3);

  \draw[->] (box1) -- (box2) node[midway,above]{$y$};
  \draw[->] (box2) -- (box3) node[midway,above]{$\xi_{1,1}$};
  \draw[->] (box3) -- (P1) node[midway,above]{$\xi_{2,1}$};
  \draw[dashed] (P1) -- (P2);
  \draw[->] (P2) -- (L1) node[midway,above]{$\xi_{r-2,1}$};
  \draw[->] (box4) -- (R1.east) node[midway,above]{$z=\xi_{r-1,1}$};

  \end{tikzpicture}
  }
\caption{Cascade of funnel pre-compensators~\eqref{eq:casc-fun-obs-i}
applied to signals~$(u,y)\in\cP_r$; taken from~\cite{BergReis18b}.}
\label{Fig:cascade}
\end{figure}
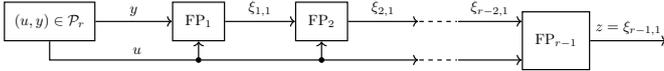

It is shown in~\cite{BergReis18b} that for signals $(u,y)\in\cP_r$ with~$r\ge 2$ such that $y,\dot y, \ldots, y^{(r-1)}$ are bounded, the funnel pre-compensator cascade~\eqref{eq:casc-fun-obs-i}
has bounded (absolutely continuous) solutions
$\bxi_i=(\xi_{i,1},\ldots,\xi_{i,r})$ with bounded gain functions~$k_i$, $i=1,\ldots, r-1$, and
\begin{multline*}\label{eq:casc-funnel-obserr}
   \forall\, i\in\{1,\ldots,{r-1}\}\ \exists\, \varepsilon_i>0\ \forall\, t>0:\\  \| \xi_{i-1,1}(t) - \xi_{i,1}(t)\| < \varphi_i(t)^{-1} - \varepsilon_i,
\end{multline*}
where $\xi_{0,1}\equiv y$. Furthermore,
\begin{equation}\label{eq:obserr-y-z}
    \forall\, t>0:\ \| y(t) -  z(t))\| < \sum_{i=1}^{r-1} \big( \varphi_i(t)^{-1} - \varepsilon_i\big).
\end{equation}

\begin{rem}\label{rem:deriv-z}
The output $z$ of the pre-compensator cascade is $(r-1)$-times continuously differentiable with explicitly-computable
(in terms of available signals) derivatives.  In particular, recursively defining functions~$\Xi_i$, $i=1,\ldots,r-1$, by
\begin{align*}
\Xi_1(t) &:= \big(q^{r-1}_1+p^{r-1}_1 k_{r-1}(t)\big)\big(\xi_{r-2,1}(t)-\xi_{r-1,1}(t)\big),
\\
\Xi_i(t) &:= \big(q^{r-1}_i+p^{r-1}_i k_{r-1}(t)\big)\big(\xi_{r-2,1}(t)-\xi_{r-1,1}(t)\big)\\
&\quad\ +\Xi_{i-1}^\prime (t),
\end{align*}
we have
\[
z^{(i)}(t) = \xi_{r-1, i+1} (t)+\Xi_i (t),\quad i=1,\ldots r-1.
\]
The essence of the pre-compensation approach to funnel control is to feedback the known variables
$z,\dot z,\ldots,z^{(r-1)}$ as surrogates for the output variable $y$ and its unknown derivatives
$\dot y,\ldots,y^{(r-1)}$.  Detailed characterizations of the surrogate variables and their dependencies on available signals
are contained in~\cite{BergReis18b}.
\end{rem}

\paragraph{Application to systems with stable internal dynamics}
We may now turn to the application of the funnel pre-compensator cascade
 in the control of system~\eqref{eq:Sysmeth}--\eqref{eq:Sysmethic}.  In particular, the input-output pair $(u,y)$, associated with the
latter system, is used to drive the cascade, generating the variable $z$.
The resulting augmented system, viewed with input $u$ and output $z$,  is amenable to funnel control as in the context of Theorem~\ref{Thm:FunCon-Nonl}.
The output~$z$  satisfies the relation~\eqref{eq:obserr-y-z}, and its derivatives (up to order $r-1$) are known explicitly as shown in Remark~\ref{rem:deriv-z}.
Thus, the funnel controller~\eqref{eq:FC-} may be applied in order to achieve the tracking objective of prescribed transient behaviour (of the primal
system output $y$) in the absence of knowledge of the derivatives~$y^{(i)}$, $i=1,\ldots,r-1$, cf.\ Fig.~\ref{Fig:intercon-1}.

Since the funnel controller~\eqref{eq:FC-} requires a bounded-input, bounded-output property of the internal dynamics of the system (cf.\ Theorem~\ref{Thm:FunCon-Nonl}; we speak of ``stable internal dynamics'' for brevity) we need to ensure that this property is preserved under interconnection with the funnel pre-compensator cascade. This can be achieved for the generic system~\eqref{eq:Sysmeth}--\eqref{eq:Sysmethic}, as shown in~\cite{BergReis18b} for relative degree two or three and, for arbitrary relative degree, in the recent work~\cite{Lanz22}. In essence, what needs to be established is that the augmented system (the conjunction of~\eqref{eq:Sysmeth} and~\eqref{eq:casc-fun-obs-i} with input~$u$ and output~$z := \xi_{{r-1},1}$) can be
equivalently written as
\begin{equation}\label{eq:Sys-Conj}
    z^{(r)}(t) = F\big( \tilde d(t), \widetilde{\mathbf{T}}(z,\dot z, \ldots, z^{(r-1)})(t)\big) + \widetilde \Gamma u(t),
    \end{equation}
 with initial data
 \begin{equation}\label{eq:Sys-Conjic}
 \left.
  \begin{array}{ll}
   \left.z\right|_{[-h,0]} = z^0\in \cC^{r-1}([-h,0],\R^m),&\text{if}~h >0,
   \\[1ex]
   (z(0),\dot z(0),\ldots,z^{r-1}(0))=(z_1^0,z_2^0,\ldots,z_{r-1}^0),\!\!&\text{if}~h=0,\end{array}\!\!\right\}
   \end{equation}
for some $\tilde d\in\cL^\infty(\R_{\ge 0},\R^r)$, $F\in\cC(\R^r\times\R^{\tilde q},\R^m)$ and an operator $\widetilde{\mathbf{T}}\in {\mathbb T}_h^{rm,\tilde q}$.
The initial data is determined by the initial data on the primal system in conjunction with the initial data on the pre-compensator cascade, the latter being open to choice
and the former being such that $y(0)$ is known. The following result is taken from~\cite{Lanz22}.

\begin{thm}\label{Thm:LinSys-MP}
Consider a system~\eqref{eq:Sysmeth}--\eqref{eq:Sysmethic} and assume that $\Gamma = \Gamma^\top \succ 0$.
Further consider the cascade of funnel pre-compensators
$FP_{r-1} \circ \ldots \circ FP_{1}$ defined by~\eqref{eq:casc-fun-obs-i}
with $\varphi_1\in\Phi_{r}$ and $\varphi_2 = \ldots = \varphi_{r-1} := \rho\, \varphi_1$ for some $\rho>1$.  Choose
pre-compensator initial data such that
\[
     \varphi_1(0) \, \|y^0(0)- \xi_{1,1}^0\| < 1,\quad
       \rho\varphi(0) \, \|\xi_{i-1,1}^0- \xi_{i,1}^0\| < 1
\]
for $i=2,\ldots,r-1$.\footnote{For example, $\xi_{i,1}^0=y^0(0)$, $i=1,\dots,r-1$ suffices.}
Furthermore, let ${\mathbf p}$ and ${\mathbf q}$ be such that \eqref{eq:Q}, \eqref{eq:lyap}, \eqref{eq:rel-piqi} hold and set
$({\mathbf p}^i,{\mathbf q}^i)=({\mathbf p},{\mathbf q})$, $i=1,\ldots,r-1$.
Moreover, assume that $\widetilde\Gamma_i = \widetilde \Gamma\in\R^{m\times m}$, $i=1,\ldots,r-1$, such that $\widetilde \Gamma = \widetilde\Gamma^\top \succ 0$ and $\Gamma \widetilde \Gamma^{-1} = \big(\Gamma \widetilde \Gamma^{-1}\big)^\top \succ 0$. Finally, assume that, $r\ge 3$, then
\begin{equation}\label{eq:ass-gam}
  \|I_m - \Gamma \widetilde \Gamma^{-1}\| < \min\left\{ \frac{\rho-1}{r-2}, \frac{\rho}{4\rho^2 (\rho+1)^{r-2} - 1}\right\}.
\end{equation}
Then the conjunction of~\eqref{eq:Sysmeth} and~\eqref{eq:casc-fun-obs-i} can be equivalently written in the form of a system~\eqref{eq:Sys-Conj} with input~$u$, output~$z := \xi _{{r-1},1}$
 and initial data~\eqref{eq:Sys-Conjic}.  Moreover, for any $u\in\cL^{\infty}_{\loc}(\R_{\ge 0},\R^m)$ it holds that
\[
 \exists\, \eps\in (0,1)\ \forall\, t>0:\
\rho_1\varphi_1(t)\|y(t)-z(t)\|\leq\eps,
\]
where $ \rho_1:=\rho/(\rho+r-2)$.
\end{thm}

By virtue of the above result, the funnel controller~\eqref{eq:FC-} may be applied to the conjunction of~\eqref{eq:Sysmeth} and~\eqref{eq:casc-fun-obs-i} with input~$u$ and output~$z := \xi_{{r-1},1}$, i.e., to system~\eqref{eq:Sys-Conj}. For the case of relative degree $r=2$ the resulting controller structure was already discussed in Section~\ref{Sssec:LinSISO-RD2}. In the following we consider the general case.  The additional combination of this controller structure (for the cases $r=2$ and $r=3$) with an open-loop control strategy is discussed in~\cite{BergOtto19} with some applications to underactuated multibody systems.

\begin{cor}
Consider system~\eqref{eq:Sysmeth}--\eqref{eq:Sysmethic} with the notation and assumptions of Theorem~\ref{Thm:LinSys-MP} in force.  Choose a triple $(\alpha,N,\varphi)$ of funnel control design parameters as in~\eqref{eq:fcts-FC} and let $y_{\mathrm{ref}}\in \cW^{r,\infty}(\R_{\ge 0},\R^m)$ be arbitrary.  Assume that, for some $\hat r\in \{1,\ldots,r\}$,
the instantaneous values $y_{\text{\rm ref}}(t),\ldots,y_{\text{\rm ref}}^{(\hat r-1)}(t)$ are known and so, setting $e^{(0)}(t)\equiv e(t):=z(t)-y_{\text{\rm ref}}(t)$,
the vector
\[
{\mathbf{e}}(t)=(e^{(0)}(t),\ldots,e^{(\hat r -1)}(t),z^{(\hat r)}(t),\ldots,z^{(r-1)}(t))
\]
(that is, ~\eqref{eq:fback-quantities} with $y(t)$ replaced by~$z(t)$) is available for feedback. Choose pre-compensator initial data such that
$\varphi(0){\mathbf{e}}(0)\in \cD_r$.
Then the funnel control
\[
u(t)=\big(N\circ \alpha\big)(\|w(t)\|^2) \, w(t),\qquad w(t)=\rho_r\big(2\varphi (t)\mathbf{e}(t)\big)
\]
(corresponding to  \eqref{eq:FC-} with $\varphi$ replaced by $2\varphi$) applied to the augmented system~\eqref{eq:Sys-Conj} yields an initial-value problem which has a solution (in the sense of Carath\'{e}odory), every solution can be maximally extended and every maximal solution
$z:\left[-h,\omega\right)\rightarrow \R^m$ has the properties:
\begin{enumerate}[(i)]
\item
$\omega=\infty$ (global existence);
\item
$u\in\cL^\infty (\R_{\ge0},\R^m)$, $z\in\cW^{r,\infty}([-h,\infty),\R^m)$;
\item
 there exists $\varepsilon_1 \in (0,1)$ such that
$2\varphi(t)\|z(t) - y_{\mathrm{ref}}(t)\|\leq \varepsilon_1$ for all $t\geq 0$.
\end{enumerate}
Moreover, setting $\varphi_1 := 2\rho^{-1}(\rho+r-2)\varphi$ in the pre-compensator, then, by Theorem \ref{Thm:LinSys-MP}, there exists $\eps_2\in (0,1)$
such that
\[
2\varphi(t)\|y(t)-z(t)\|\leq \eps_2,\quad\text{for all}~t\geq 0.
\]
Writing $\eps := \half (\eps_1+\eps_2)$, gives
\begin{enumerate}[(iv)]
\item
$\varphi(t)\|y(t)-y_{\text{\rm ref}}(t)\|\leq \eps\quad\text{for all}~t\geq 0$,
\end{enumerate}
and so the performance objective is achieved.
\end{cor}

\begin{rem}
The funnel pre-compensator successfully circumvents ``the curse of dimensionality" associated with the filtering approach (as discussed in  Remark~\ref{rem:disc-filter}).
However, the adage ``there ain't no such thing as a free lunch" applies\footnote{In an optimization context, \textit{Wolpert and Macready}~\cite{WolpMacr97,WolpMacr05} paraphrase their concept  of a no-free-lunch theorem  as ``any two algorithms are equivalent when their performance is averaged across all possible problems".}: circumvention of the curse via pre-compensation comes with a price.  First, the system
matrix $\Gamma$ in \eqref{eq:Sysmeth} is required to be symmetric and positive definite (only sign definiteness is required for filtering).  More restrictive is
assumption~\eqref{eq:ass-gam} in Theorem~\ref{Thm:LinSys-MP} which essentially means that the controller matrix~$\widetilde \Gamma$ in the funnel pre-compensators
needs to be ``sufficiently close'' to the (unknown) system matrix~$\Gamma$. How close is specified by the bound on the right-hand side, which becomes tighter as the
relative degree~$r$ increases.  For example, for $r=3$ and~$r=5$, maximizing this bound with respect to the choice of design parameter~$\rho>1$ gives
\[
\|I_m - \Gamma \tilde\Gamma^{-1}\|
\ < \
\begin{cases}
0.117, & \text{if}~r=3\\
0.027,& \text{if}~r=5.
\end{cases}
\]
This indicates that~$\Gamma$ must be known to a high degree of accuracy.
For more comments on the role of assumption~\eqref{eq:ass-gam}  see~\cite[Rem.~3.10]{Lanz22}.
\end{rem}

\section{Input constraints}\label{Sec:input-constraints}\setcounter{equation}{0}
%
\noindent
Up to this point, all exposition and discussion of funnel control has been predicated on an implicit assumption that the input variables
are unconstrained in magnitude.
From a practical point of view, this may be deemed unrealistic.  In most physically-based applications, control inputs are subject to constraints.
Can funnel control accommodate such constraints?  Given that the idea underlying the methodology is that inputs can take
remedial control action of sufficiently large magnitude so as to avoid contact with the funnel boundary, it is clear that some additional feasibility conditions
are mandatory if the inputs are constrained.  Not unexpectedly, such feasibility conditions translate into restrictions on the
initial data, disturbances and reference signals associated with the process to be controlled, and on the underlying
performance funnel.

\subsection{Funnel control with saturation}\label{Ssec:FC-sat}
\noindent
If the vector of control inputs is restricted to take its values in the closed ball ${\mathbb B}_{\widehat u}^m =\{w\in\R^m\,|~\|w\|\leq\widehat u\}$ for
some $\widehat u >0$,  then it is natural to accommodate this input constraint by adopting the saturation function:
\begin{equation}
        \label{eq:sat_mult}
        \satu \colon \R^m \to {\mathbb B}_{\widehat u}^m, \
        v \mapsto \left\{ \begin{array}{ll}
                                                {\widehat  u}\,{\|v\|^{-1}} v, & \|v\| > \widehat  u \\
                                                v, & \text{otherwise.}
                                            \end{array} \right.
\end{equation}
For the purpose of motivation,  consider again the scalar linear prototype~\eqref{abc} with $cb>0$,  but now with input values
constrained to the interval $[-\widehat u,\widehat u]$.
The unconstrained funnel controller~\eqref{eq:FC} is  replaced by  the saturated strategy
\begin{equation}\label{eq:FC-constr}
u(t)  = -\satu (k(t)e(t)),~~\\
k(t)  = \varphi(t)\big(1- (\varphi(t)e(t))^2\big)^{-1}.
\end{equation}
We compare the unconstrained closed-loop system~\eqref{abc},~\eqref{eq:FC}, i.e.,
\begin{equation}
        \label{eq:e-sc-unconst}
\begin{aligned}
\dot e(t) &=  \big(a-cb \ k(t)\big) \, e(t)
+  ay_{\text{\rm ref}}(t)-\dot y_{\text{\rm ref}}(t),~\\
&e(0) = cx^0 -y_{\text{\rm ref}}(0)
\end{aligned}
\end{equation}
with the constrained closed-loop system~\eqref{abc},~\eqref{eq:FC-constr}, i.e.,
\begin{equation}
        \label{eq:e-sc-constr}
\begin{aligned}
\dot e(t) &=  a e(t) - cb \  \satu (k(t)e(t)) +  ay_{\text{\rm ref}}(t)-\dot y_{\text{\rm ref}}(t),\\
e(0) &= cx^0 -y_{\text{\rm ref}}(0).
\end{aligned}
\end{equation}

In either case, the initial data condition
 $\varphi(0)|e(0)| < 1$ (trivially satisfied if $\varphi(0)=0$) is clearly necessary  for attainment of the funnel control objective. However,
 whilst this condition is also sufficient in the unconstrained case, it fails to be so in the constrained case.
    Feasibility of the tracking objective in the presence of input saturation inevitably involves   an interplay between the plant data $(a,b,c,x^0)$, the reference signal~$y_{\rm ref}$, the function $\varphi\in\Phi$ and the saturation level $\widehat u$. For instance, if~$a>0$, then it is readily verified that $a|cx^0|/(cb)\leq \widehat u$ is a necessary condition for feasibility; furthermore, the saturation level~$\widehat u$ should also, loosely speaking, be commensurate with the
 $\cW^{1,\infty}$ norm of the reference signal~$y_{\rm ref}$. To illustrate the interplay between $\widehat u$ and the  funnel function~$\varphi$,
consider the case wherein $a=0$, $y_{\rm ref}(\cdot)=0$ and~$\varphi$ is such that its reciprocal $\psi =1/\varphi$ is a monotonically decreasing, globally Lipschitz function with
Lipschitz constant~$\Lambda$. Assume feasibility of the tracking objective. Then,
\begin{multline*}
\Lambda t\leq\psi(0)-\psi(t)<\psi(0)-y(t)\\
=\psi(0)-y(0)-\int_0^t \dot y(s)\,\dd s< cb\,\widehat u\,t
\end{multline*}
for all $t\ge 0$, and so $cb \, \widehat u\geq \Lambda$ is a necessary condition for feasibility.  This case serves to illustrate that the saturation level must be large enough so that the control
can accommodate local ``steepness" of  the funnel boundary.

For multi-input, multi-output linear systems~\eqref{eq:ABC} of relative degree one,
with $CB + (CB)^\top \succ 0 $ and asymptotically stable zero dynamics, the efficacy of the funnel control \eqref{eq:FC-constr}
 is established in
\textit{Hopfe, Ilchmann,  Ryan}~(2010)~\cite[Thm.\,4.1]{HopfIlch10a}, assuming that  a feasibility inequality holds.
The latter means that~$\widehat u$ must be sufficiently large in terms
of the system data, the initial data, $\varphi$, $y_{\rm \scriptsize ref}$,
$\dot y_{\rm \scriptsize ref}$, and  $\dot\varphi$. This inequality is a conservative bound, but it ensures
feasibility of funnel control.  For the case of componentwise saturation constraints, which requires a componentwise funnel control strategy, see~\cite[Thm.\,4.3]{HopfIlch10a}.

In the highly specialized context of the scalar system~\eqref{abc}, the result of~\cite[Thm.\,4.1]{HopfIlch10a} translates into the following: if
        \begin{multline}\label{eq:feas_scalar}
                \varphi(0)|cx^0-y_{\rm ref}(0)| < 1 \quad \text{and}
                \\
                cb\, \widehat u \geq |a|  \big(\|\psi\|_\infty + \|y_{\rm ref}\|_\infty\big) + \|\dot y_{\rm ref}\|_\infty + \|\dot \psi\|_\infty,
        \end{multline}
         then the simple control strategy~\eqref{eq:FC-constr}  ensures attainment of the tracking objective (and, moreover, the gain function~$k$ is bounded).  Furthermore, if the first inequality in~\eqref{eq:feas_scalar} is replaced by $\varphi(0)|cx^0-y_{\rm ref}(0)| <  \widehat u (1+\widehat u)^{-1}$, then input saturation does not occur and so the control strategy coincides with~\eqref{eq:FC}.

A generalization of the above  to single-input, single-output nonlinear systems is presented in
\textit{Hopfe, Ilchmann,  Ryan} (2010)~\cite{HopfIlch10b}. For single-input, single-output systems of relative degree two,  a variant of funnel control with
(scalar) input saturation is given in
\textit{Hackl,  Hopfe,   Ilchmann,  Mueller,   Trenn} (2013)~\cite[Thm.~3.3]{HackHopf13}.  A treatment of a class of nonlinear systems arising in chemical reactor
models is contained in \textit{Ilchmann and Trenn} (2004)~\cite{IlchTren04}.

\subsection{Bang-bang funnel control}\label{Ssec:bang-bang}
%
\noindent
To treat  constrained scalar-input systems with arbitrary (but known) relative degree, a bang-bang funnel control strategy has also been developed.
This approach avoids the backstepping procedure (cf.\ Section~\ref{Ssec:backstep}) and uses derivative feedback, similar to the funnel control methods discussed in Section~\ref{Ssec:dervi-fb}. However, the
scalar control input switches only between two values and is hence able to respect input constraints.
As is to be expected, the approach requires satisfaction of feasibility assumptions.

The bang-bang funnel controller was first introduced by \textit{Liberzon and Trenn} (2010)~\cite{LibeTren10} for nonlinear systems with relative degree $r\leq 2$  and later generalized to arbitrary relative degree in~\cite{LibeTren13b}. The case of time delays is discussed in~\cite{LibeTren13a} for relative degree two systems.
The systems considered in~\cite{LibeTren13b} are $n$-dimensional, time-invariant, control-affine, disturbance-free,  and are expressible in the form
\begin{align*}
 y^{(r)}(t)&=f(y(t),\dot y(t),\ldots,y^{(r-1)}(t),\eta(t))\\
 &\quad +g(y(t),\dot y(t),\ldots,y^{(r-1)}(t),\eta(t))u(t),
\\
\dot\eta (t)&=h(y(t),\dot y(t),\ldots,y^{(r-1)}(t),\eta(t)),
\end{align*}
where $f$, $g$, $h$ are locally Lipschitz and $g$ is positive valued.
 Temporarily regarding the second of the above subsystems
as an independent $(n-r)$-dimensional system with $\R^r$-valued input $v$, that is, the system $\dot\eta =h(v,\eta)$ with associated flow $\varrho$,  it is assumed that this
system has the {\it bounded-input,bounded-state} property and so, with initial data $\eta (0)=\eta^0$ and continuous input $v\in\cC (\R_{\ge 0},\R^r)$, the
unique solution $\eta (\cdot )=\varrho (\cdot,\eta^0,v)$ is globally defined.  Introducing the operator (more precisely, a family of
operators parameterized by $\eta^0$, but again we suppress this dependency)
\[
{\mathbf T}\colon \cC(\R_{\ge 0},\R^r)\to\cL^\infty_{\text{loc}}(\R_{\ge 0},\R^{n-r}),\ v\mapsto \big(v(\cdot),\varrho (\cdot,\eta^0,v)\big),
\]
the generic system takes the form
\begin{multline}\label{sys:bb}
y^{(r)}(t)=f({\mathbf T}(y,\dot y,\ldots,y^{(r-1)})(t))\\
+ g({\mathbf T}(y,\dot y,\ldots,y^{(r-1)})(t))\,u(t)
\end{multline}
which, in the absence of input constraints,  is a system of class $\cN^{1,r}$ amenable to funnel control.  In the presence of constraints,
the bang-bang funnel controller switches between two values and the control law is given by
\begin{equation}\label{eq:BBFC}
    u(t) = \begin{cases} U^-, & \text{if $q(t)=$ true},\\ U^+, & \text{if $q(t)=$ false}, \end{cases}
\end{equation}
where $U^- < U^+$ and $q:\R_{\ge 0}\to \{\text{true, false}\}$ is the switching signal determined by the switching logic~$\cS$ depending on the error signal. The situation is illustrated in Fig.~\ref{Fig:bangbang}, wherein $\varphi_0,\ldots,\varphi_{r-1}$ are functions defining individual performance funnels for the tracking error and its first $(r-1)$ derivatives.
\begin{figure}[h!tb]
\centering
\resizebox{\columnwidth}{!}{
    \begin{tikzpicture}[draw=black,thick,node distance = 6ex, box/.style={rectangle, draw=black}, blackdot/.style={inner sep = 0, minimum size=3pt,shape=circle,fill=black,draw=black},whitedot/.style={inner sep = 0, minimum size=3pt,shape=circle,draw=black}]
  \node (system) [box] {System~\eqref{sys:bb}};
   \node (yfork) [blackdot,right of = system,xshift=18ex] {};
   \node (y) [right of = yfork] {$y$};
   \node (logic) [box,below of = system,yshift=-2ex] {$\begin{array}{c}\text{switching}\\\text{logic}\end{array}$};
   \node (plus) [circle,inner sep = 0,very thick,draw,below of = yfork,yshift=-2ex] {$+$};
   \node (yref) [right of = plus,xshift=2ex] {$-y_{\rm ref}$};
   \node (funnel) [below of = logic,yshift=-4ex] {$\varphi_0,\ldots,\varphi_{r-1}$};
   \node (Uplus) [box, below of = logic,xshift=-22ex,yshift=-3ex] {$U^+$};
   \node (Uminus) [box, left of = Uplus] {$U^-$};
   \node (Uminusconnect) [whitedot, above of = Uminus] {};
   \node (Uplusconnect) [whitedot, above of = Uplus] {};
   \node (switch) [blackdot,above of=Uplusconnect,xshift=-3ex] {};
   \node (invisible) [left of=logic,xshift=-16ex] {};
   \draw (system) -- (yfork);
   \draw[->] (yfork) -- (y);
   \draw[->] (yfork) -- (plus);
   \draw[->] (yref) -- (plus);
   \draw[->] (plus) -- (logic) node[midway,above] {{$e$, $\dot{e}$, ..., $e^{(r-1)}$}};
   \draw[->] (logic) -- (invisible) node[midway,above] {$q$};
   \draw (Uplus) -- (Uplusconnect);
   \draw (Uminus) -- (Uminusconnect);
   \draw (Uplusconnect) -- (switch);
   \draw[gray!50] (Uminusconnect) -- (switch);
   \draw[->] (switch) |- (system) node[near end,above] {$u$};
   \draw[->] (funnel) -- (logic);
\end{tikzpicture}
}
\caption{Closed-loop system consisting of the bang-bang funnel controller applied to system~\eqref{sys:bb}; taken from~\cite{LibeTren13b}.}
\label{Fig:bangbang}
\end{figure}
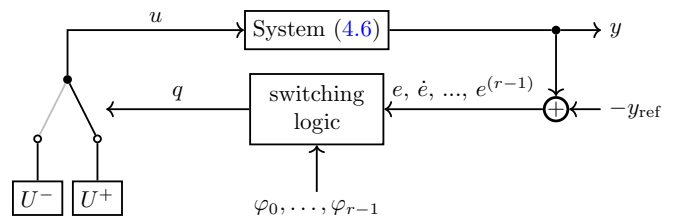
We refer to~\cite{LibeTren13b} for a precise description of the switching logic $\cS: (e, \dot e, \ldots, e^{(r-1)}) \mapsto q$.   Subject to inevitable feasibility conditions,
the closed-loop system has a global solution.  Moreover, the switching signal~$q$ has locally finitely many switches
(and so ``chattering" behaviour does not occur) and the tracking error~$e$ and its derivatives~$\dot e, \ldots, e^{(r-1)}$ evolve within their respective performance
funnels.
\subsection{Funnel control under arbitrary input constraints}\label{Ssec:ICFC}
%
\noindent
In the above described approaches, the saturation level~$\widehat u$ must be sufficiently large in order to ensure feasibility of funnel control under input saturation. The reason for this is the inflexibility of the output constraints, given by the performance funnel for the tracking error. In the recent work~\cite{Berg22pp} a different viewpoint is taken. There, the input constraints are considered to be \textit{hard constraints}, being imposed by the physical limitations of the system. On the other hand, the output constraints are considered to be \textit{soft constraints}, which can be weakened whenever this is inevitable in order to meet the input constraints. To achieve this, a modified control design was proposed, where the funnel boundary (determined by the reciprocal~$\psi(\cdot) = 1/\varphi(\cdot)$) is no longer prescribed for all $t\ge 0$, but it is dynamically generated and becomes part of the controller design. The generating mechanism for $\psi$ is such that the funnel has a prescribed shape (chosen {\it a priori by} the designer) whenever the saturation is not active, that is,
when  $u(t) = -k(t)e(t)$ in the context of~\eqref{eq:e-sc-constr}.
When the saturation is active, the funnel is dynamically ``widened" with the aim of de-activating saturation.  On achieving this aim, the funnel is adjusted to recover
its designed shape exponentially fast.

The idea to readjust the funnel boundary when the input saturation becomes active was already formulated in~\cite{HackJi07b} for relative degree one systems, however the saturation level must still be sufficiently large. The same control design as in~\cite{Berg22pp} was independently developed in~\cite{TrakBech22} for relative degree one systems in the context of prescribed performance control. Higher relative degree systems with input amplitude and rate constraints are considered in~\cite{TrakBech23}. Again, both works~\cite{TrakBech22,TrakBech23} still require sufficiently large saturation levels. In the more recent work~\cite{TrakBech24} this requirement is replaced by an input-to-state stability assumption on the system, which however is often not satisfied in practical applications. While it is further shown in~\cite{TrakBech24} that this assumption can be relaxed to a bounded-input-bounded-state assumption on the internal dynamics (equivalent to property~(TP3)), boundedness of closed-loop signals can then only be guaranteed for those evolving in a certain (unknown) compact set. Neither of those assumptions is required for the approach presented in~\cite{Berg22pp}.

To illustrate the idea of~\cite{Berg22pp}, consider the case  of the scalar system~\eqref{abc} with $cb>0$. Then the \textit{input-constrained funnel controller} is given by
\begin{equation}\label{eq:ICFC}
\hspace*{-1ex}
\boxed{
\begin{aligned}
    e(t) &= y(t) - y_{\rm ref}(t),\ k(t) = \left(1- \frac{\|e(t)\|^2}{\psi(t)^2}\right)^{-1}\\
    \dot \psi(t) &= -\alpha \psi(t) + \beta + \psi(t) \frac{\kappa(v(t))}{\|e(t)\|},\
    \psi(0) = \psi^0\\
    \kappa(v(t)) &= \|v(t)-\satu(v(t))\|,\\
     v(t) &= -k(t) e(t),\quad
    u(t) = \satu(v(t))
\end{aligned}
}
\end{equation}
with the controller design parameters
$\alpha >0$, $  \beta>0$,  $\psi^0>\beta/\alpha$.
The controller essentially consists of a standard funnel strategy appended by the dynamics for the funnel boundary. The idea is that, if the saturation is not active and hence $\kappa(v(t))=0$ on an interval~$[t_0,t_1]$, then the funnel boundary is of the form $\psi(t) = \psi(t_0) e^{-\alpha (t-t_0)} + \frac{\beta}{\alpha} \left(1-e^{-\alpha (t-t_0)}\right)$; if the saturation is active and hence~$\kappa(v(t)) > 0$, then the funnel boundary is widened according to the dynamics of the controller in order to guarantee the input constraints. After a period of saturation, the boundary reverts to its prescribed shape exponentially fast. However, in the absence of precise knowledge of the system to be controlled, a priori bounds on the widening of the funnel cannot be computed. Indeed, in some pathological situations, the funnel widening may be unbounded.

In general,~\cite{Berg22pp} covers nonlinear functional differential equations of relative degree $r\in\N$, satisfying a sector bound property. Additionally, as mentioned above, the high-gain property~(NP1) is not needed and for the internal dynamics only a ``local'' bounded-input bounded-output property is required. The controller~\eqref{eq:ICFC} for the case of~$r>1$ again consists of a version of the relative degree $r$ funnel controller, appended by the dynamics for the funnel boundaries, where the widening effect due to an active saturation propagates from the $r$-th funnel boundary to the first through the dynamic equations.
\section{Applications}\label{Sec:Appls}\setcounter{equation}{0}
%
\noindent
Funnel control has proved a useful tool in various applications in several fields of engineering. Straightforward applications can be found in mechanical engineering, robotics and mechatronics, but also in voltage and current control of electrical circuits or synchronous machines tracking problems are frequently encountered for which funnel control proved to be an appropriate choice. Moreover, funnel control has permeated areas in which its application is less obvious, such as control of chemical reactor processes,
artificial ventilation units and therapy.

In the following subsections we consider the applications for relative degree one systems and systems with higher relative degree separately. In each case we provide an overview of the available applications to the best of our knowledge. Additionally, for illustration purposes we pick one of the applications and discuss in detail that it fits into the respective system class and hence funnel control is feasible.

We note that applications for prescribed performance control~-- the relative of funnel control discussed in Subsection~\ref{Sssec:PPC}~-- can be found in the recent comprehensive survey~\cite{Bu23}.

\subsection{Relative degree one systems}\label{Ssec:RD1}
\noindent
The following applications are available for systems with relative degree one:

\begin{supertabular}{p{160pt}|p{70pt}}
    application & discussed in\\
      \hline\hline\\[-3mm]
      speed control of industrial servo-systems; & \cite{Hack13,HackHofm11,IlchSchu09,SchuHack05} and~\cite[Ch.~11]{Hack17} \\[2mm]
      speed control of wind turbine systems; & \cite{Hack14,Hack15b} and~\cite[Ch.~12]{Hack17} \\[2mm]
      current control of electric synchronous machines; & \cite{Hack15a} and~\cite[Ch.~14]{Hack17} \\[2mm]
      voltage and current control of electrical circuits; & \cite{BergReis14a} \\[2mm]
      power flow control in intermediate DC bus of electrical drives; & \cite{SenfPaug14} \\[2mm]
      temperature control of chemical reactor models; & \cite{IlchTren04} \\[2mm]
      control of peak inspiratory pressure of artificial ventilation units; & \cite{PompWeye15} \\[2mm]
      oxygenation control during artificial ventilation therapy; & \cite{PompAlfo14}\\[2mm]
      adaptive cruise control with guaranteed safety; & \cite{BergRaue18,BergRaue20} \\[2mm]
      synchronization of multi-agent systems; & \cite{LeeBerg21pp} \\[2mm]
       control of the containment of epidemics. & \cite{Berg22} \\
\end{supertabular}

As an example we consider a discretized transmission line~\cite{FlieMart99} (described by a differential-algebraic equation) and show that it is amenable to funnel control; this example is taken from~\cite{BergReis14a}. The discretized transmission line is depicted in Fig.~\ref{Fig:TML}, where $n$ is the number of spacial discretization points.
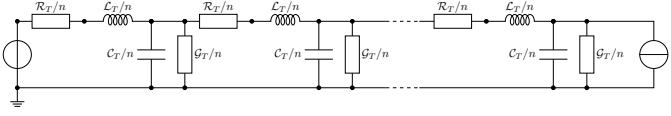
\begin{figure}[hbt]
\centering
\resizebox{\columnwidth}{!}{
\begin{circuitikz}
\draw
    (0,0) to[V,*-*] (0,2)
    to[R=$\cR_T/n$] (2,2)
    to[L=$\cL_T/n$,*-*] (4,2) -- (5,2)
    to[R=$\cR_T/n$,*-*] (7,2)
    to[L=$\cL_T/n$] (9,2) -- (10,2)
    to[*-] (11,2)
    (12,2) to[R=$\cR_T/n$] (14,2)
    to[L=$\cL_T/n$,*-*] (16,2) to[*-] (17,2) -- (19,2)
    to[I] (19,0) -- (17,0)
    to[R, l_=$\cG_T/n$,*-*] (17,2)
    (17,0) -- (16,0) to[C=$\cC_T/n$,*-*] (16,2)
    (16,0) -- (12,0)
    (11,0) -- (10,0) to[R, l_=$\cG_T/n$,*-*] (10,2)
    (10,0) -- (9,0) to[C=$\cC_T/n$,*-*] (9,2)
    (9,0) -- (5,0) to[R, l_=$\cG_T/n$,*-*] (5,2)
    (5,0) -- (4,0) to[C=$\cC_T/n$,*-*] (4,2)
    (4,0) -- (0,0)
    (0,0) node[ground] {} -- (0,0);
    \draw[dashed] (11,0) -- (12,0);
    \draw[dashed] (11,2) -- (12,2);
\end{circuitikz}
}
\caption{Discretized transmission line; taken from~\cite{BergReis14a}.}
\label{Fig:TML}
\end{figure}
Using modified nodal analysis (MNA), see~\cite{HoRueh75} and the survey~\cite{Reis14}, we may obtain a model of the circuit which is described by a linear differential-algebraic equation of the form~\eqref{eq:EABC}, where
\begin{multline*}
sE-A
=\begin{bmatrix}s\mA_{\cC}\cC \mA_{\cC}^\top +\mA_{\cR}\cG \mA_{\cR}^\top &\mA_{\cL}&\mA_{\cV}\\-\mA_{\cL}^\top &s\cL&0\\-\mA_{\cV}^\top &0&0\end{bmatrix},\\
B=C^\top =\begin{bmatrix}-\mA_{\cI}&0\\0&0\\0&-I_{n_{\cV}}\end{bmatrix},\label{eq:circpen}
\end{multline*}
\begin{equation*}\label{eq:circvar}
    x=(\eta^\top , i_{\cL}^\top , i_{\cV}^\top )^\top ,\quad u=(i_{\cI}^\top , v_{\cV}^\top )^\top ,\quad y=(-v_{\cI}^\top , -i_{\cV}^\top )^\top ,
\end{equation*}
and
\begin{equation*}\label{eq:circmatr}
\left.
\begin{array}{l}
 \cC\in\R^{n_{\cC}\times n_{\cC}}, \cG\in\R^{n_{\cG}\times n_{\cG}}, \cL\in\R^{n_{\cL}\times n_{\cL}},\\[2mm]
 \mA_{\cC}\in\R^{n_e\times n_{\cC}}, \mA_{\cR}\in\R^{n_e\times n_{\cG}}, \mA_{\cL}\in\R^{n_e\times n_{\cL}},\\[2mm]
  \mA_{\cV}\in\R^{n_e\times n_{\cV}}, \mA_{\cI}\in\R^{n_e\times n_{\cI}},\\[2mm]
 n=n_e+n_{\cL} + n_{\cV},\quad m=n_{\cI}+n_{\cV}.
\end{array}
\ \right\}
\end{equation*}
Here $\mA_{\cC}$, $\mA_{\cR}$, $\mA_{\cL}$, $\mA_{\cV}$ and $\mA_{\cI}$ denote the element-related incidence matrices, $\eta(t)$ is the vector of node potentials, $i_{\cL}(t)$, $i_{\cV}(t)$, $i_{\cI}(t)$ are the vectors of currents through inductances, voltage and current sources, $v_{\cV}(t)$, $v_{\cI}(t)$ are the voltages of voltage and current sources,
and $\cC$, $\cG$ and $\cL$ are the matrices expressing the constitutive relations of capacitances, resistances and inductances.
In \cite[Prop.~7.4]{BergReis14a}, it is shown that this system has asymptotically stable zero dynamics and so Theorem~\ref{Thm:DAE-funnel} is applicable to
conclude feasibility of  funnel control. For simulations of various scenarios,  we refer to~\cite{BergReis14a}.

\subsection{Higher relative degree systems}\label{Ssec:HRD}
\noindent
The following applications are available for systems with higher relative degree:\\

\begin{supertabular}{p{160pt}|p{70pt}}
    application & discussed in\\
      \hline\hline\\[-3mm]
      position control of industrial servo-systems; & \cite{Hack11,HackHofm11,HackHopf13} and~\cite[Ch.~11]{Hack17} \\[2mm]
      joint position control of rigid-link revolute-joint robotic manipulators; & \cite{HackKenn12,BergLanz21,BergLe18,BergOtto19} and~\cite[Ch.~13]{Hack17} \\[2mm]
      position control for a robotic manipulator with kinematic loop; & \cite{BergDrue21}\\[2mm]
      force control for a mass on car system; & \cite{BergLe18,BergOtto19} \\[2mm]
permanent magnet synchronous motor service system; & \cite{ChenTang17} \\[2mm]
      oxygenation control in artificial ventilation therapy .& \cite{PompAlfo14} \\
\end{supertabular}
\vspace{2mm}

As an example we consider a robotic manipulator from~\cite{HackKenn12}, see also~\cite[Ch.~13]{Hack17} and~\cite[p.~77]{KellDavi05}, as depicted in Fig.~\ref{Fig:robot-model}. The robotic manipulator is planar, rigid, with revolute joints and has two degrees of freedom.

\begin{figure}[h!tb]
\captionsetup[subfloat]{labelformat=empty}
%
%
%
\begin{center}
\begin{tikzpicture}[x=4mm, y=4mm,thick,node distance = 12ex, box/.style={fill=white,rectangle, draw=black}, blackdot/.style={inner sep = 0, minimum size=3pt,shape=circle,fill,draw=black},plus/.style={fill=white,circle,inner
sep = 0,thick,draw},metabox/.style={inner sep = 3ex,rectangle,draw,dashed,fill=gray!20!white}]
\tikzset{>=latex}
\tikzset{every path/.append style={line width=0.8pt}}

    \fill[gray] (7,0) rectangle (10,-0.5);
    \draw (7,0) rectangle (10,-0.5);

    \draw (7,0) -- (8.5,0.5);
    \draw (8.5,0.5) circle (0.2);
    \draw[->] (8.8,1.25) arc (70:270:0.8);
    \node[left] at (7.8,0.7) {$u_1$};

    \draw (8.5,0.5) -- (10,0);
    \draw (8.5,0.5) -- (15,5);
    \node[left] at (11.5,3) {$l_1$};
    \fill (15,5) circle (0.2);
    \node[right] at (15.5,5) {$m_1$};

    \draw[->] (15.3,5.7) arc (70:290:0.8);
    \node[left] at (14.1,5) {$u_2$};

    \draw[dashed] (8.5,0.5) -- (13,0.5);

    \draw[dashed,->] (11.5,0.5) arc (0:37:2.8);
    \node[right] at (9.7,1) {$y_1$};

    \draw (15,5) -- (12,8);
    \node[left] at (13.2,6.8) {$l_2$};
    \fill (12,8) circle (0.2);
    \node[left] at (11.8,8) {$m_2$};

    \draw[dashed] (15,5) -- (17.6,6.8);

    \draw[dashed,->] (16.7,6.5) arc (45:140:2);
    \node[right] at (14.6,6.4) {$y_2$};
\end{tikzpicture}
\end{center}
\vspace{-3mm}
\caption{Planar rigid revolute joint robotic manipulator; taken from~\cite{BergLe18}.}
\label{Fig:robot-model}
\end{figure}
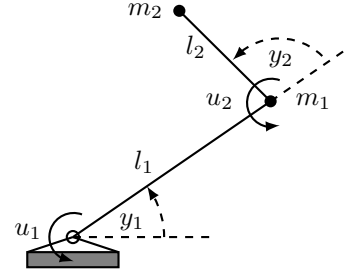

The two joints are actuated by~$u_1$ and~$u_2$. We assume that the links are massless, have lengths~$l_1$ and~$l_2$, and point masses~$m_1$ and~$m_2$
are attached to their ends. The two outputs are the joint angles~$y_1$ and~$y_2$  and the equations of motion are given by (see also~\cite[pp.~259]{SponHutc06})
\begin{equation}\label{eq:robot}
M(y(t))\ddot{y}(t)+C(y(t),\dot{y}(t))\dot{y}(t)+G(y(t))=u(t)
\end{equation}
with initial value $(y(0),\dot{y}(0))=\left(0,0\right)$, inertia matrix $M:\R^2\to \R^{2\times 2},$
\begin{align*}
&M(y_1,y_2)\\
&:=\!
\begin{bmatrix}
m_1l_1^2\!+\!m_2(l_1^2\!+\!l_2^2\!+\!2l_1l_2\cos(y_2)) & m_2(l_2^2\!+\!l_1l_2\cos(y_2))\\
m_2(l_2^2\!+\!l_1l_2\cos(y_2)) & m_2l_2^2\\
\end{bmatrix}
\end{align*}
centrifugal and Coriolis force matrix $C:\R^2\times \R^2\to \R^{2 \times 2}$,
\begin{align*}
 &C(y_1,y_2,v_1,v_2)\\
 &:=
\begin{bmatrix}
-2m_2l_1l_2\sin(y_2)v_1 & -m_2l_1l_2\sin(y_2)v_2\\
-m_2l_1l_2\sin(y_2)v_1 & 0\\
\end{bmatrix},
\end{align*}
and gravity vector $G:\R^2\to \R^2$,
\begin{align*}
&G(y_1,y_2)\\
&:=
g \begin{pmatrix}
m_1l_1\cos(y_1)+m_2(l_1\cos(y_1)+l_2\cos(y_1+y_2)) \\
m_2l_2\cos(y_1+y_2), \\
\end{pmatrix}
\end{align*}
where $g$ is the acceleration of gravity. If we right- multiply system~\eqref{eq:robot} with $M(y(t))^{-1}$, which is pointwise positive definite, we see that the resulting system belongs to the class~\eqref{eq:nonlSys} with $r=m=2$. Therefore,
Theorem~\ref{Thm:FunCon-Nonl} yields that funnel control is feasible.

In the following, we present a simulation of the controllers~\eqref{eq:FC-} and~\eqref{eq:fun-con} for this example to illustrate their performance. This simulation is taken from~\cite{BergIlch21}. We choose the parameters~$m_1=m_2=1$, $l_1=l_2=1$ and the reference signal $y_{\rm ref}\colon t\mapsto (\sin t,\,\sin 2t)$. We choose $\varphi(t) = (4e^{-2t}+0.1)^{-1} = \varphi_0((t) = \varphi_1(t)$ for $t\ge 0$, as well as $\alpha(s) = 1/(1-s)$ for~$s\in[0,1)$ and, in view of Remark~\ref{Rem:N(s)=-s}, $N(s) = -s$ for $s\ge 0$. The simulation of the controllers~\eqref{eq:FC-} and~\eqref{eq:fun-con} applied to~\eqref{eq:robot} over the time interval $[0,10]$ is MATLAB generated (solver: {\tt ode45}, rel.\ tol.: $10^{-10}$, abs.\ tol.: $10^{-10}$) and depicted in Fig.~\ref{fig:robot}. It can be seen that for this example both controllers exhibit a nearly identical performance.
\begin{figure}
  \centering
  \subfloat[Fig.~\ref{fig:robot}a: Funnel and first tracking error components]
{
\centering
\hspace{-2mm}
  \includegraphics[width=0.47\textwidth]{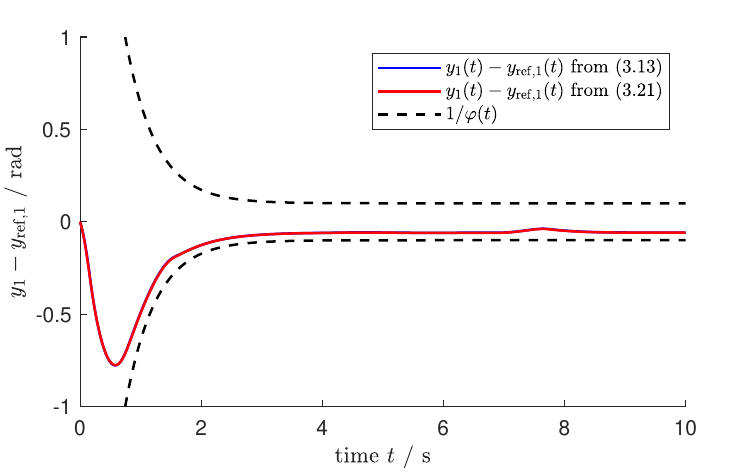}
\label{fig:robot-error1}
}\\
\subfloat[Fig.~\ref{fig:robot}b: Funnel and second tracking error com-\newline ponents]
{
\centering
\hspace{-2mm}
  \includegraphics[width=0.47\textwidth]{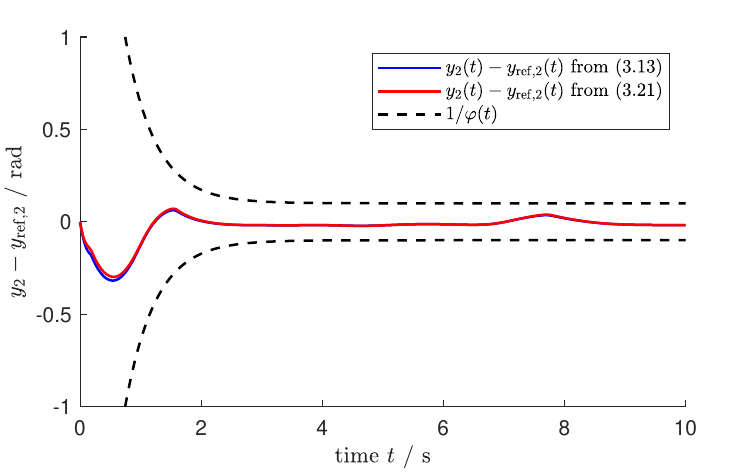}
\label{fig:robot-error2}
}\\
\subfloat[Fig.~\ref{fig:robot}c: First input components]
{
\centering
\hspace{-2mm}
  \includegraphics[width=0.47\textwidth]{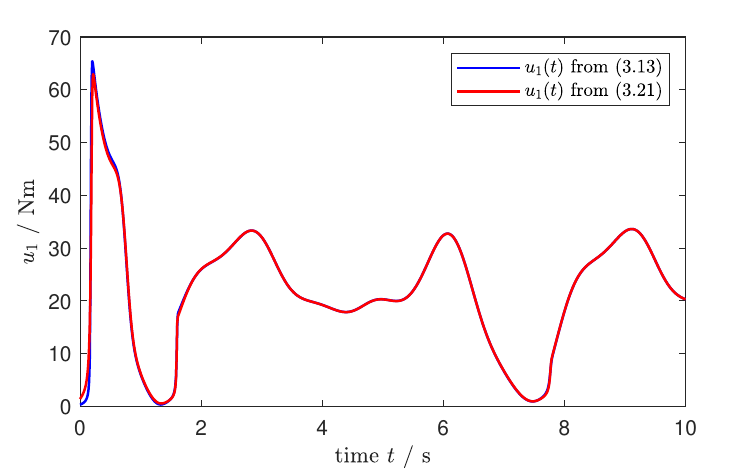}
\label{fig:robot-input1}
}\\
\subfloat[Fig.~\ref{fig:robot}d: Second input components]
{
\centering
\hspace{-2mm}
  \includegraphics[width=0.47\textwidth]{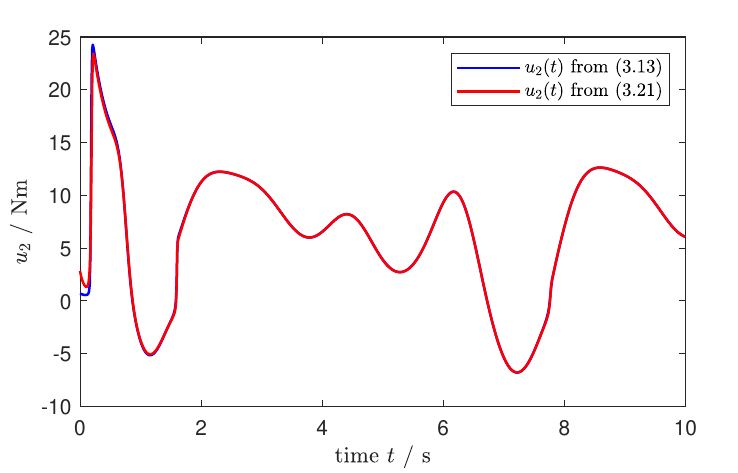}
\label{fig:robot-input2}
}
\caption{Simulation of the controllers~\eqref{eq:FC-} and~\eqref{eq:fun-con} applied to~\eqref{eq:robot}; taken from~\cite{BergIlch21} and re-generated with color.}
\label{fig:robot}
\end{figure}
For simulations of various additional scenarios and corresponding figures we refer to the
works~\cite{BergIlch21,BergLe18}.

\subsection{Systems with partial differential equations}\label{Ssec:Appl-PDE}
\noindent
The following applications are available for systems containing partial differential equations:\\

\begin{supertabular}{p{170pt}|p{60pt}}
    application & discussed in\\
      \hline\hline\\[-3mm]
      boundary control of heat propagation problems; & \cite{ReisSeli15b}\\[2mm]
      control of a lossy transmission line; & \cite{PuchReis21}\\[2mm]
      mean value control of molecular systems; & \cite{BergNues22pp} \\[2mm]
      control of defibrillation processes for the human heart; & \cite{BergBrei21}\\[2mm]
      force control for a moving water tank.& \cite{BergPuch22}\\
\end{supertabular}
%
\noindent
As an example we consider the moving water tank system from~\cite{BergPuch22}, which is depicted in Fig.~\ref{Fig:Tank}.

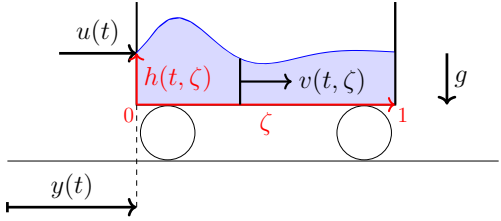
\begin{figure}[h!t]
  \centering
\resizebox{0.75\columnwidth}{!}{
\begin{tikzpicture}

\draw (0.6,-0.55) circle (15pt); \draw (4.4,-0.55) circle (15pt);
\draw[line width=1.1pt] (0,2) -- (0,0) -- (5,0) -- (5,2);

\draw[->,line width=1.5pt] (-1.5,1) -- (0,1)  node[midway, above]{\Large $u(t)$};

\draw[dashed] (0,0) -- (0,-2);
\draw[->,line width=1.5pt] (-2.5,-2) -- (0,-2)  node[midway, above]{\Large $y(t)$};
\draw[line width=1.5pt] (-2.5,-1.9) -- (-2.5,-2.1);

\draw[->,line width=1.5pt] (6,1) -- (6,0) node[midway, right]{\Large $g$};

\draw (-2.5,-1.1) -- (7,-1.1);

\draw[line width=1.1pt, domain=0:5,smooth,variable=\x,blue] plot ({\x},{1*sqrt(3*\x)*sin(100*\x)*exp(-1*\x)+1});
\fill[color=blue!15,domain=0:5] (0,0)-- plot (\x,{1*sqrt(3*\x)*sin(100*\x)*exp(-1*\x)+1})--(5,0)-- (0,0);

\draw[red,line width=1.1pt,->] (0,0) -- (0,1) node[midway, right]{\Large $h(t,\zeta)$};
\draw[red,line width=1.1pt,->] (0,0) -- (5,0) node[midway, below]{\Large $\zeta$};
\node at (-0.15,-0.2) {\large \color{red} 0};
\node at (5.15,-0.2) {\large \color{red} 1};

\draw[line width=1.2pt] (2,0) -- (2,0.9);
\draw[->,line width=1.2pt] (2,0.45) -- (3,0.45) node[pos=1, right]{\Large $v(t,\zeta)$};
\draw[line width=1.1pt] (5,0) -- (5,2);
\end{tikzpicture}
}
\caption{Horizontal movement of a water tank; taken from~\cite{BergPuch22}.}
\label{Fig:Tank}
\end{figure}

We neglect the wheels' inertia and friction between the wheels and the ground, and assume that there is an external force acting on the water tank, denoted by $u(t)$. The measurement output is the horizontal position~$y(t)$ of the water tank, and the mass of the empty tank is denoted by~$m$. The dynamics of the water
can be described by the {\it Saint-Venant equations}, cf.~\cite{Sain71}, as
\begin{equation}\label{eq:SVeq}
\begin{aligned}
  \partial_t h + \partial_\zeta(h v)= 0,\ \
  \partial_t v +  \partial_\zeta \left( \frac{v^2}{2} + gh\right)+hS\left(\frac{v}{h}\right)= - \ddot y
\end{aligned}
\end{equation}
with  boundary conditions $v(t,0) = v(t,1) = 0$ and friction term~$S:\R\rightarrow\R$. Here $h:\R_{\ge 0}\times [0,1]\to \R$ denotes the height profile and $v:\R_{\ge 0}\times [0,1]\to \R$ the (relative) horizontal velocity profile, where the length of the container is normalized to~$1$. As in~\cite{BergPuch22} we use a linearized version of these equations:
\begin{equation}\label{eq:SVlin}
    \partial_t z = Az +b \ddot y = -\begin{bmatrix}0&h_0 \partial_\zeta\\g\partial_\zeta &2\mu\end{bmatrix}z+\begin{pmatrix}0\\-1\end{pmatrix}\ddot{y},
\end{equation}
with boundary conditions $z_2(t,0) = z_2(t,1) = 0$, $b=(0,-1)^\top$ and friction coefficient
$\mu = \tfrac12 S'(0) > 0$. The state space in which $z(t)$ evolves is $X=L^2([0,1];\R^2)$
and $A: \cD(A) \subseteq  X \to X$,
\begin{equation}\label{eq:domA}
    \cD(A) = \setdef{(z_1,z_2)\in X}{ \!\!\begin{array}{l} z_1,z_2\in W^{1,2}([0,1];\R),\\ z_2(0) = z_2(1) = 0\end{array}\!\!\!}.
\end{equation}
By conservation of mass in  \eqref{eq:SVlin}, $\int_{0}^{1}z_{1}(t,\zeta)\mathrm{d}\zeta=h_{0}$ for all $t\ge0$.
The model is completed by the momentum
\begin{equation}\label{eq:momentum}
    p(t) := m \dot y(t) + \int_0^1 z_1(t,\zeta) \big(z_2(t,\zeta) + \dot y(t)\big) \ds{\zeta},\ t\ge0.
\end{equation}
Substituting the absolute velocity $x_2=z_{2}+\dot{y}$ for $z_2$, $x_{1}=z_{1}$ and using the balance law $\dot p(t) = u(t)$ and~\eqref{eq:SVlin} we obtain the nonlinear model on the state space $X$:
\begin{subequations}\label{eq:InpOut}
\begin{align}\label{eq:InpOutlin}
	  \partial_t x &= A(x+b\dot{y}) \\
 \!  m \ddot y(t) &= \tfrac{g}{2}  x_1(t,\cdot)^{2}|_{0}^{1}\!+\!{2\mu}\langle x_{1}(t),x_{2}(t)\rangle-2\mu h_0\dot{y}(t)+{u(t)} \label{eq:InpOutlin2}
\end{align}
\end{subequations}
where $\langle f,g\rangle=\int_{0}^{1}f(s)g(s)\mathrm{d}s$. This system can be written as
\begin{align}%
\ddot{y}(t)&=\fT(\dot y)(t)+\frac{u(t)}{m},\label{eq:Teq1}
\end{align}
where the operator $\fT$ is formally given by
\begin{align*}
\fT(\eta)(t)={}&\frac{g}{2m}  x_1(t,\cdot)^{2}|_{0}^{1}  + \frac{2\mu}{m} \big( \langle x_{1}(t),x_{2}(t)\rangle - h_0\eta(t)\big)\notag
\end{align*}
with $x$ being the strong solution of
\begin{equation*}
  \dot{x}(t)= A\big(x(t)+ b\eta(t)\big), \quad x(0)=x_0.\label{eq:Teq0}
\end{equation*}
It is then shown in~\cite{BergPuch22} that $\fT\in {\mathbb T}^{2,1}_0$ and hence~\eqref{eq:Teq1} belongs to the class $\cN^{1,2}$, thus
Theorem~\ref{Thm:FunCon-Nonl} yields that funnel control is feasible. For simulations and corresponding figures we refer to~\cite{BergPuch22}.

%
\section{Future research and open problems}\label{Sec:Fut-Res-Open_problems}
%
\subsection{Model predictive control~(MPC)}\label{Ssec:MPC}
\noindent
MPC is a well-established control technique which relies on the iterative solution of optimal control problems~(OCPs), see the textbooks~\cite{GrunPann17,RawlMayn20}. Recently,~\cite{BergDenn22a,BergDenn22b,BergKast19}
have introduced funnel-like ideas to overcome some limitations in~MPC. The latter means that  ``artificial'' assumptions are imposed to find  an initially feasible solution and to ensure recursive feasibility  of~MPC (i.e., solvability of the~OCP at a particular time instant automatically implies solvability of the~OCP at the successor time instant). It was shown that these assumptions are superfluous  when ``funnel-like'' stage costs are introduced so that the costs grow unbounded when the tracking error approaches the funnel boundary. More precisely, in contrast to simply adding the constraints on the tracking error to the~OCP with standard quadratic stage costs, funnel~MPC is initially and recursively feasible, without imposing state constraints or terminal conditions and independent of the length of the prediction horizon. These results hold for a large class of nonlinear multi-input multi-output systems with relative degree one, very similar to the class~${\mathcal N}^{m,1}$, as shown in~\cite{BergDenn22a}. The applicability of the method has been demonstrated for a magnetic levitation system in~\cite{OppeLanz24}. Utilizing so called feasibility constraints and extending the stage costs by additional terms (similar to the gain functions in~\eqref{eq:fun-con}), applicability of funnel~MPC to systems with arbitrary relative degree was shown in~\cite{BergDenn22b}. However, the parameters involved in the feasibility constraints are very hard to determine and usually conservative estimates must be used. But then again, initial and recursive feasibility cannot be guaranteed. Furthermore, the stage cost function used in~\cite{BergDenn22b} is rather complex and (together with the feasibility constraints in the optimal control problem) leads to an increased computational effort. These drawbacks have been resolved in the recent work~\cite{BergDenn23bpp}, where a modified and simple stage cost is used and the feasibility constraints are avoided.
In~\cite{BergDenn23pp} the combination of funnel~MPC with an additional funnel control feedback loop was investigated, and it was shown that this leads to a control scheme which achieves the tracking objective even in case of severe model-plant mismatches. This resolves another limitation of classical~MPC: it requires a sufficiently accurate model to predict the system behaviour and compute the optimal control in each step. In the approach from~\cite{BergDenn23pp},  funnel~MPC is safeguarded by the additional funnel controller, to guarantee the evolution of the tracking error within the funnel boundaries. Another extension of this approach is presented in~\cite{LanzBerg23pp}, where a framework to improve the model by learning its parameters from data is introduced, while it is still safeguarded by the funnel controller component.
A limitation of the approach is that the learning scheme must guarantee that the ``improved model'' is again a member of the considered class of models, which, so far, is only clear for simple learning algorithms restricted to linear models. Furthermore, the extension to arbitrary relative degree is an open problem.
%

\subsection{Partial differential equations}\label{Ssec:PDEs}
\noindent
In the context of systems containing partial differential equations, an important challenge is the treatment of
systems with inputs and outputs which are  \emph{not co-located},
that means the actuators and sensors are not placed at the same position.
Note that all boundary control systems considered  in Section~\ref{Ssec:Inf-hard}, as e.g.\ the heat equation~\eqref{eq:heateq}, have co-located inputs and outputs.
In the following we describe two prototypical examples of
systems  in one spatial variable  which illustrate the more realistic situation where the input and output are not co-located.

 First, consider the  wave equation
\begin{equation}\label{eq:wave}
\hspace*{-1ex}
\begin{aligned}
&\partial^2_t x(\xi,t)=c^2\, \partial^2_\xi x(\xi,t),\quad (\xi,t)\in [0,\ell]\times\R_{>0},\\
&u(t)=\partial_\xi x(0,t), \ y(t)=\partial_\xi x(\ell,t), \
0=x(\ell,t),\  t >0.
\end{aligned}
\end{equation}
This equation describes
 a~vibrating string of length~$\ell$, where the input and output consist of a~scaled force at the left and right hand side, resp.
 Furthermore, the boundary condition~$0=x(\ell,t)$ means that the right hand side of the string is clamped.
The application of an input causes a~wave which is travelling with speed $c>0$ to the right hand side, where it is reflected.
Consequently,  any input action influences the output with a delay of $\tau =\tfrac{\ell}c$. A standard funnel controller is not able to deal with this behaviour, since a ``bad choice'' of the reference signal and funnel boundary may potentially drive the tracking error outside the performance funnel within the time interval~$[0,\tau]$, without the control being able to counteract.

As a second model problem, consider the heat equation
\begin{equation}\label{eq:heat}
\begin{aligned}
\hspace*{-1ex}
&\partial_t x(\xi,t)=k\, \partial^2_\xi x(\xi,t),\quad (\xi,t)\in [0,\ell]\times\R_{>0},\\
&u(t)=\partial_\xi x(0,t),\
y(t)=x(\ell,t),\
0=\partial_\xi x(\ell,t),\  t >0
\end{aligned}
\end{equation}
with $k>0$, and boundary control formed by the temperature flux at the left hand side. Then the output is given by the temperature at the right hand side, and the condition $0=\tfrac{\partial}{\partial \xi}x(\ell,t)$ describes a~perfect isolation at the right hand side. In contrast to the wave equation, the problem of a delayed control action does not
occur here.
However, the diffusive effect implies that the output~$y:\R_{>0}\to\R$ is infinitely smooth, regardless of a possibly non-smooth~$u\in L^\infty(\R_{>0})$. In a~certain sense, this corresponds to an infinite relative degree. Inspection of the zero dynamics, that is,~\eqref{eq:heat} under the additional condition~$y=0$,  results in an  equation with Neumann and Dirichlet  boundary values  at the same
part of the boundary, which is not well-posed.  Also for this example, standard funnel control is not feasible in general.

The above issues suggest~-- for completely different reasons~--  that standard funnel control does not achieve the objective of tracking with prescribed performance of the tracking error for the systems~\eqref{eq:wave} and~\eqref{eq:heat}.
Suitable modifications of the funnel controller and, probably, additional (smoothness, quantitative) assumptions on the funnel boundary~$\varphi$ and the
reference signal~$y_{\rm ref}$ warrant further investigation.

\subsection{Control barrier functions (CBFs)}

The objective of CBFs, originally introduced in~\cite{wieland2007constructive}, is closely related to that of funnel control, i.e., to provide a means to guarantee the satisfaction of safety-critical constraints. Due to their straightforward design and modularity, CBFs have seen widespread use, particularly in robotics, see~\cite{ames2016control,ames2019control}. The control input under CBF-based schemes is typically obtained by solving a quadratic program, which enforces a desired bound on the derivative of the barrier function. This procedure ensures the positive invariance and asymptotic stability of a safe set. However, implementing CBF-based control necessitates a model capturing the nonlinear system dynamics and accurate state measurements, with the overall implementation complexity increasing alongside the model's dimension. These challenges have been partially addressed in the literature by considering robustness to bounded disturbances~\cite{jankovic2018robust,gurriet2018towards}, handling sector nonlinearities~\cite{buch2021robust}, and accounting for state estimation errors~\cite{dean2021guaranteeing}; moreover, input-to-state safety has been studied~\cite{alan2023control,kolathaya2018input}. Recent approaches have also focused on learning the necessary model information directly from data~\cite{taylor2020learning,dhiman2021control}, though this tends to further increase the complexity of implementation. To address the complexity issue, \cite{cohen2024safety} explores the use of simplified reduced-order models. Additionally, for a particular class of reduced-order kinematic models,~\cite{molnar2021model} proposes a model-free CBF-based controller that requires only minimal knowledge of the system's dynamics. Recently, funnel control and prescribed performance control have been utilized to develop model-free CBF-based feedback laws, see~\cite{namerikawa2024equivalence} for prescribed performance control and~\cite{LanzKoeh25} for funnel control. However, these works are restricted to systems with relative degree one and future work should concentrate on the extension to higher relative degrees and other system classes.

\subsection{Other open problems}\label{Ssec:div}
\noindent
\paragraph{Systems with unstable zero dynamics}
Recently, funnel control for systems which do not have asymptotically stable zero dynamics has been investigated. First results on funnel control for systems which are not minimum phase are given in~\cite{Berg20} for uncertain linear systems and in~\cite{BergLanz21} for a nonlinear robotic manipulator. Further research is necessary to extend the results to general nonlinear systems.

\paragraph{Sampled-data funnel control}
Recently, \textit{Lanza et al.}~\cite{LanzDenn23} have investigated funnel control for sampled-data systems with zero-order hold, showing that for a sampling rate below a certain threshold (depending on the system parameters, the reference signal and the funnel boundary) the tracking error evolves within the prescribed performance funnel~-- also between the sampling instances. This result even covers the general class~$\cN^{m,r}$ of nonlinear systems with arbitrary relative degree. However, the estimates for the sampling rate threshold are quite conservative. In another recent work~\cite{ChenRen23}, \textit{Cheng et al.} develop a funnel control design for discrete-time nonlinear systems. However, their approach is based on sliding-mode control methods and parameter estimations, which lead to a high controller complexity. Future research should focus on relaxing the limitations of both approaches.

\paragraph{Funnel cruise control}
 \textit{Berger and Rauert}~\cite{BergRaue20} have developed a ``funnel cruise controller'' as a universal adaptive cruise control mechanism for vehicle following which guarantees that a safety distance to the leader vehicle is never violated. A complementary approach, which also takes input constraints into account, has been developed by \textit{Trakas et al.} in the recent work~\cite{TrakModu25}. Based on the funnel cruise controller, a decentralized controller which achieves string stability of vehicle platoons has been introduced in the recent work~\cite{BergBess24}. Open problems are the treatment of acceleration constraints and the investigation of synchronization behaviour.

\paragraph{Funnel control with internal models}
There are contributions on funnel control in combination with
internal models~-- i.e., models of the exogenous signals such as reference signals or disturbances, cf.~\cite{Wonh85}.
It is shown in~\cite{IlchRyan06a} for linear systems with relative degree one that this combination achieves asymptotic tracking. In~\cite[Ch.~7~\&~10]{Hack17} it is shown that
this control is also efficient in the presence of measurement noise: the tracking error does not ``follow'' the noise and hence it does not get close to the funnel boundary and, as a consequence, the gain function does not attain unnecessary large values. In the end, the incorporation of internal models leads to an increased level of robustness of the overall controller design and, from an applications point of view, implementation of funnel control in real-world systems without internal models is challenging. {First results for asymptotic tracking by funnel control with internal models has been reported for linear systems with arbitrary relative degree in~\cite{BergHack24}. However, the treatment of measurement noise and nonlinear systems remain open problems.}

\paragraph{Robustness in the gap metric}
Robustness of adaptive controllers has been an issue since the 1980s, see e.g.~\cite{IoanKoko84,RohrVala85}. So-called universal adaptive controllers, including the funnel controller,
satisfy the desired control objective for  a whole \textit{class} of systems.
In this sense, these controllers are already robust.
However, an issue still remains as to whether  the controller continues to maintain
performance if a system of the underlying class
is subjected to perturbations,for example to unmodelled dynamics, which take it outside the class.
One established tool to quantify robustness is the {{\it gap metric}}, with which
the distance between two systems is measured as the ``gap" separating  their corresponding graphs.
Robustness in the gap metric of the classical high-gain adaptive controller~\eqref{vark}, \eqref{kdot} was studied in
\textit{French, Ilchmann, and Ryan} (2006) \cite{FrenIlch06}.
In~\cite{IlchMuel09b} it is shown that the funnel controller~\eqref{eq:FC}
 applied to a linear system~\eqref{eq:ABC} is robust in the following sense:
it may be applied to a system not satisfying any of
the classical conditions of relative degree one,  known sign of the high-frequency gain and asymptotically
stable zero dynamics
as long as the initial conditions and
the disturbances are ``small'' and the system is ``close'' (in terms of a
``small'' gap) to a system satisfying the classical conditions.
An extension of this result to systems with relative degree two is derived in~\cite{HackHopf13}. It is an open problem as to whether similar gap metric results hold for funnel control
for higher relative degree nonlinear and/or differential-algebraic systems.

\paragraph{Fault tolerant funnel control}
 \textit{Berger}~\cite{Berg21b} has recently developed a fault tolerant funnel control mechanism for nonlinearly perturbed linear systems. The method utilizes the Byrnes-Isidori form for time-varying linear systems from~\cite{IlchMuel07}. The extension to fully nonlinear systems is a topic of future research. A first approach in this direction has been presented by \textit{Zhang et al.}~\cite{ZhanDing25}.

\paragraph{Funnel control versus prescribed performance control}
Prescribed performance control (see Subsection~\ref{Sssec:PPC}) and funnel control
are closely related. A thorough comparison of the complexity of the controllers
and the assumptions on the system class is still missing.
This may lead to   new controllers with
less complexity which work for a larger class of systems.

\section*{Acknowledgments}
We are indebted to our colleagues
Christoph M.\ Hackl (Hochschule München),
Lukas Lanza (Technische Universität Ilmenau),
Timo Reis (Technische Universität Ilmenau) and
George A.\ Rovithakis (Aristotle University of Thessaloniki)
 for their constructive comments.

\small
\bibliographystyle{elsarticle-harv}

\end{document}